\documentclass[11pt,a4paper]{article}
\usepackage{mathrsfs, amsmath, amsthm, amssymb, bm, mathtools, thm-restate}
\usepackage{graphicx, xcolor, tikz}
\usetikzlibrary{calc, shapes}
\usepackage{caption, subcaption}
\usepackage{array}
\makeatletter
\def\th@plain{%
  \upshape 
}
\makeatother

\makeatletter
\renewenvironment{proof}[1][\proofname]{\par
  \pushQED{\qed}%
  \normalfont \topsep6\p@\@plus6\p@\relax
  \trivlist
  \item[\hskip\labelsep
        \bfseries
    #1\@addpunct{.}]\ignorespaces
}{%
  \popQED\endtrivlist\@endpefalse
}
\makeatother

\newtheorem{theorem}{Theorem}[section]
\newtheorem{lemma}[theorem]{Lemma}
\newtheorem{corollary}[theorem]{Corollary}
\newtheorem{case}{Case}
\newtheorem{subcase}{Subcase}[case]
\newtheorem{subsubcase}{Subcase}[subcase]
\newtheorem{claim}{Claim}
\theoremstyle{definition}
\newtheorem{definition}{Definition}
\newtheorem{remark}{Remark}
\newtheorem{proposition}{Proposition}

\usepackage[left=25mm, top=25mm, bottom=25mm, right=25mm]{geometry}
\usepackage[T1]{fontenc}
\usepackage[inline]{enumitem}
\usepackage[square, numbers, sort&compress]{natbib}

\usepackage[
  bookmarks=true,%
  bookmarksnumbered=true, 
  bookmarksopen=true, 
  plainpages=false,%
  pdfpagelabels,%
  colorlinks=true, 
  linkcolor=black, 
  citecolor=black,%
  anchorcolor=green,
  urlcolor=blue,
  breaklinks=true,
  hyperindex=true]{hyperref}

\newcommand{\etal}{et~al.\ }
\newcommand{\ie}{i.e.,\ }
\usepackage{marvosym,wasysym}

\def\int(#1){\mathrm{int}(#1)}
\def\ext(#1){\mathrm{ext}(#1)}
\def\Int(#1){\mathrm{Int}(#1)}
\def\Ext(#1){\mathrm{Ext}(#1)}

\usetikzlibrary{decorations.pathreplacing,decorations.markings}
\tikzset{
  on each segment/.style={
    decorate,
    decoration={
      show path construction,
      moveto code={},
      lineto code={
        \path [#1]
        (\tikzinputsegmentfirst) -- (\tikzinputsegmentlast);
      },
      curveto code={
        \path [#1] (\tikzinputsegmentfirst)
        .. controls
        (\tikzinputsegmentsupporta) and (\tikzinputsegmentsupportb)
        ..
        (\tikzinputsegmentlast);
      },
      closepath code={
        \path [#1]
        (\tikzinputsegmentfirst) -- (\tikzinputsegmentlast);
      },
    },
  },
  mid arrow/.style={postaction={decorate,decoration={
        markings,
        mark=at position .55 with {\arrow[#1]{stealth}}
      }}},
}

\newcommand{\ee}{\mathrm{EE}}
\renewcommand{\oe}{\mathrm{OE}}
\title{Variable degeneracy of graphs with restricted structures}
\author{Qianqian Wang \qquad Tao Wang\footnote{\tt Corresponding
author: wangtao@henu.edu.cn; https://orcid.org/0000-0001-9732-1617} \qquad Xiaojing Yang\\
{\small Center for Applied Mathematics}\\
{\small Henan University, Kaifeng, 475004, P. R. China}}
\begin{document}
\date{}
\maketitle
\begin{abstract}
Bernshteyn and Lee defined a new notion, weak degeneracy, which is slightly weaker than the ordinary degeneracy. It is proved that strictly $f$-degenerate transversal is a common generalization of list coloring, $L$-forested-coloring and DP-coloring. In this paper, we consider three classes of graphs, including planar graphs without any configuration in \autoref{FIGPAIRWISE3456}, toroidal graphs without any configuration in \autoref{A345}, and  planar graphs without intersecting $5$-cycles. We give structural results for each class of graphs, and prove each structure is reducible for weakly $3$-degenerate and the existence of strictly $f$-degenerate transversals. As consequences, these three classes of graphs are weakly $3$-degenerate, and have a strictly $f$-degenerate transversal. Then these three classes of graph have DP-paint number at most four, and have list vertex arboricity at most two. This greatly improve all the results in  \cite{MR4362322,MR4212281,MR4078909,MR4089638,MR3969022,MR3996735,MR3802151,MR4112063,MR3979933,MR3761240,MR3699856,MR3508765,MR3320048}. Furthermore, the first and the third classes of graphs have Alon-Tarsi number at most four. 
\end{abstract}



\section{Introduction}
All graphs in this paper are finite, undirected and simple. Let $\mathbb{N}$ stand for the set of nonnegative integers, $[n]$ stand for the set $\{1, 2, \dots, n\}$. Let $G$ be a graph, and $f$ be a function from $V(G)$ to $\mathbb{N}$, the graph $G$ is \emph{strictly $f$-degenerate} if every subgraph $\Gamma$ has a vertex $v$ with $d_{\Gamma}(v) < f(v)$. When $f$ is an integer $k$, the strictly $f$-degenerate is equivalent to the ordinary $(k-1)$-degenerate. A sequence of vertices $v_{1}, v_{2}, \dots, v_{n}$ is a \emph{strictly $f$-degenerate order (resp. $f$-removing order)} if every vertex $v_{i}$ has less than $f(v_{i})$ neighbors with smaller (resp. larger) subscripts. Then a graph is strictly $f$-degenerate if and only if it has a strictly $f$-degenerate order or $f$-removing order.

Let $G$ be a graph and $L$ be a list assignment for $G$. For each vertex $v \in V(G)$, let $L_{v} = \{v\} \times [s]$; for each edge $uv \in E(G)$, let $\mathscr{M}_{uv}$ be a matching between the sets $L_{u}$ and $L_{v}$. Note that $\mathscr{M}_{uv}$ is not required to be a perfect matching between $L_{u}$ and $L_{v}$, and possibly it is empty. Let $\mathscr{M} := \bigcup_{uv \in E(G)}\mathscr{M}_{uv}$, and call $\mathscr{M}$ a \emph{matching assignment}. A \emph{cover} of $G$ is a graph $H_{L, \mathscr{M}}$ (simply write $H$) satisfying the following two conditions: 
\begin{enumerate}[label =(C\arabic*)]
\item the vertex set of $H$ is the disjoint union of $L_{v}$ for all $v \in V(G)$; 
\item the edge set of $H$ is the matching assignment $\mathscr{M}$.
\end{enumerate}

Note that the induced subgraph $H[L_{v}]$ is an independent set for each vertex $v \in V(G)$. Let $S$ be a vertex subset or a subgraph of $G$, we use $H_{S}$ to denote the subgraph induced by $\bigcup_{v \in S}L_{v}$ in $H$. Let $H$ be a cover of $G$ and $f$ be a function from $V(H)$ to $\mathbb{N}$, we call the pair $(H, f)$ a \emph{valued cover} of $G$. A vertex subset $R \subseteq V(H)$ is a \emph{transversal} of $H$ if $|R \cap L_{v}| = 1$ for each $v \in V(G)$. A transversal $R$ is a \emph{strictly $f$-degenerate transversal} if $H[R]$ is strictly $f$-degenerate. We say that a vertex $v$ in $G$ \emph{is colored with $p$} if $(v, p)$ is chosen in a strictly $f$-degenerate transversal of $H$. 

Let $H$ be a cover of $G$ and $f$ be a function from $V(H)$ to $\{0, 1\}$. A \emph{DP-coloring} of $H$ is a strictly $f$-degenerate transversal of $H$. Note that a DP-coloring is a special independent set in $H$.

The \emph{DP-chromatic number $\chi_{\mathrm{DP}}(G)$} of $G$ is the smallest integer $k$ such that $(H, f)$ has a DP-coloring whenever $H$ is a cover of $G$ and $f$ is a function from $V(H)$ to $\{0, 1\}$ with $f(v, 1) + f(v, 2) + \dots + f(v, s) \geq k$ for each $v \in V(G)$. A graph $G$ is \emph{DP-$k$-colorable} if its DP-chromatic number is at most $k$.

The \emph{vertex arboricity} of a graph $G$ is the minimum number of subsets (color classes) in which $V(G)$ can be partitioned so that each subset induces a forest. An \emph{$L$-forested-coloring} of $G$ for a list assignment $L$ is a coloring $\phi$ (not necessarily proper) such that $\phi(v) \in L(v)$ for each $v \in V(G)$ and each color class induces a forest. The \emph{list vertex arboricity} of a graph $G$ is the least integer $k$ such that $G$ has an $L$-forested-coloring for any list $k$-assignment. The list vertex arboricity of a graph is the list version of vertex arboricity.

It is proved \cite {MR4357325} that strictly $f$-degenerate transversal is a common generalization of list coloring, $L$-forested-coloring and DP-coloring.

Motivated by the study of greedy algorithms for graph coloring problems, Bernshteyn and Lee \cite{Bernshteyn2021a} defined a notion of \emph{weak degeneracy}.

\begin{definition}[\textsf{Delete} Operation]
Let $G$ be a graph and $f : V(G) \longrightarrow \mathbb{N}$ be a function. For a vertex $u \in V(G)$, the operation \textsf{Delete}$(G, f, u)$ outputs the graph $G' = G - u$ and the function $f': V(G') \longrightarrow \mathbb{Z}$ given by 
\begin{align*}
f'(v) \coloneqq
&\begin{cases}
f(v) - 1, & \text{if $uv \in E(G)$};\\[0.3cm]
f(v), & \text{otherwise}.
\end{cases}
\end{align*}
An application of the operation \textsf{Delete} is \emph{legal} if the resulting function $f'$ is nonnegative. 
\end{definition}

\begin{definition}[\textsf{DeleteSave} Operation]
Let $G$ be a graph and $f : V(G) \longrightarrow \mathbb{N}$ be a function. For a pair of adjacent vertices $u, w \in V(G)$, the operation \textsf{DeleteSave}$(G, f, u, w)$ outputs the graph $G' = G - u$ and the function $f' : V(G') \longrightarrow \mathbb{Z}$ given by
\begin{align*}
f'(v) \coloneqq
&\begin{cases}
f(v) - 1, & \text{if $uv \in E(G)$ and $v \neq w$};\\[0.3cm]
f(v), & \text{otherwise}.
\end{cases}
\end{align*}
An application of the operation \textsf{DeleteSave} is \emph{legal} if $f(u) > f(w)$ and the resulting function $f'$ is nonnegative. 
\end{definition}

A graph $G$ is \emph{weakly $f$-degenerate} if it is possible to remove all vertices from $G$ by a sequence of legal applications of the operations \textsf{Delete} and \textsf{DeleteSave}. A graph is \emph{$f$-degenerate} if it is weakly $f$-degenerate with no the \textsf{DeleteSave} operation. Given $d \in \mathbb{N}$, we say that $G$ is \emph{weakly $d$-degenerate} if it is weakly $f$-degenerate with respect to the constant function of value $d$. We say that $G$ is \emph{$d$-degenerate} if it is $f$-degenerate with respect to the constant function of value $d$. The \emph{weak degeneracy} of $G$, denote by $\textsf{wd}(G)$, is the minimum integer $d$ such that $G$ is weakly $d$-degenerate. The \emph{degeneracy} of $G$, denote by $\textsf{d}(G)$, is the minimum integer $d$ such that $G$ is $d$-degenerate.

Bernshteyn and Lee \cite{Bernshteyn2021a} gave the following inequalities.

\begin{proposition}\label{prop}
For any graph $G$, we always have 
\begin{equation*}
\chi(G) \leq \chi_{\ell}(G) \leq \chi_{\textsf{DP}}(G) \leq \chi_{\textsf{DPP}}(G) \leq \textsf{wd}(G) + 1 \leq \textsf{d}(G) + 1, 
\end{equation*}
where $\chi_{\textsf{DP}}(G)$ is the DP-chromatic number of $G$, and $\chi_{\textsf{DPP}}(G)$ is the DP-paint number of $G$. 
\end{proposition}

Bernshteyn and Lee \cite{Bernshteyn2021a} showed that every planar graph is weakly $4$-degenerate, and hence its DP-paint number is at most 5. 

A digraph $D$ is a {\em circulation} if $d_{D}^{+}(v) = d_{D}^{-}(v)$ for each vertex $v \in V(D)$. In particular, a digraph with no arcs is a circulation. Let $G$ be a graph and $D$ be an orientation of $G$. Let $\ee(D)$ be the family of spanning sub-circulations of $D$ with even number of arcs, and let $\oe(D)$ be the family of spanning sub-circulations of $D$ with odd number of arcs. Let 
\[
\mathrm{diff}(D) = |\ee(D)| - |\oe(D)|. 
\]
We say that an orientation $D$ of $G$ is a \emph{$k$-Alon-Tarsi orientation} ($k$-AT-orientation for short) if $\mathrm{diff}(D) \neq 0$ and $\Delta(D)^{+} < k$. The {\em Alon-Tarsi number} $AT(G)$ of a graph $G$ is the smallest integer $k$ such that $G$ has a $k$-AT-orientation $D$. The following proposition is well-known \cite{MR2578908}.

\begin{proposition}
For any graph $G$, we always have 
\begin{equation*}
\textsf{ch}(G) \leq \chi_{\textsf{P}}(G) \leq AT(G), 
\end{equation*}
where $\chi_{\textsf{P}}(G)$ is the paint number of $G$. 
\end{proposition}

Zhu \cite{MR3906645} proved that every planar graph has Alon-Tarsi number at most five. Zhu \cite{MR4152773} proved that every planar graph $G$ has a matching such that $AT(G - M) \leq 4$. It is known that $\textsf{ch}(G) \leq AT(G)$ and $\textsf{ch}(G) \leq \textsf{wd}(G) + 1$, but there are no relations between $AT(G)$ and $\textsf{wd}(G) + 1$ for general graphs. 

In this paper, we consider three classes of graphs, proving that each considered graph has weak degeneracy at most $3$, and each one has list vertex arboricity at most two. Moreover, plane graphs without any configuration in \autoref{FIGPAIRWISE3456} have Alon-Tarsi number at most four.

It is proved that the following classes of graphs have DP-chromatic number at most four: planar graphs without $3$-cycles adjacent to $4$-cycles \cite{MR3969022}; planar graphs without $3$-cycles adjacent to $5$-cycles \cite{MR4078909}; planar graphs without $3$-cycles adjacent to $6$-cycles \cite{MR4078909}; planar graphs without $4$-cycles adjacent to $5$-cycles \cite{MR3996735}; planar graphs without $4$-cycles adjacent to $6$-cycles \cite{MR3996735}; planar graphs without $5$-cycles adjacent to $6$-cycles \cite{Zhang2019}; planar graphs without pairwise adjacent $3$-, $4$- and $5$-cycles \cite{MR4089638}.

For planar graphs, Li and Wang \cite{Li2019} gave the following result on strictly $f$-degenerate transversal which improves the mentioned results with only restrictions on $5^{-}$-cycles. 

\begin{figure}%
\centering
\subcaptionbox{\label{fig:subfig:a--}}
{\begin{tikzpicture}[scale = 0.8]
\coordinate (A) at (45:1);
\coordinate (B) at (135:1);
\coordinate (C) at (225:1);
\coordinate (D) at (-45:1);
\coordinate (H) at (90:1.414);
\draw (A)--(H)--(B)--(C)--(D)--cycle;
\draw (A)--(B);
\node[circle, inner sep = 1.5, fill = white, draw] () at (A) {};
\node[circle, inner sep = 1.5, fill = white, draw] () at (B) {};
\node[circle, inner sep = 1.5, fill = white, draw] () at (C) {};
\node[circle, inner sep = 1.5, fill = white, draw] () at (D) {};
\node[circle, inner sep = 1.5, fill = white, draw] () at (H) {};
\end{tikzpicture}}\hspace{1.5cm}
\subcaptionbox{\label{fig:subfig:b--}}
{\begin{tikzpicture}[scale = 0.8]
\coordinate (A) at (30:1);
\coordinate (B) at (150:1);
\coordinate (C) at (225:1);
\coordinate (D) at (-45:1);
\coordinate (H) at (90:1.414);
\coordinate (X) at (60:1.4);
\coordinate (Y) at (120:1.4);
\coordinate (T) at ($(H)!(A)!(X)$);
\coordinate (Z) at ($(A)!2!(T)$); 
\draw (A)--(X)--(Z)--(H)--(Y)--(B)--(C)--(D)--cycle;
\draw (A)--(H)--(B);
\draw (H)--(X);
\node[circle, inner sep = 1.5, fill = white, draw] () at (A) {};
\node[circle, inner sep = 1.5, fill = white, draw] () at (B) {};
\node[circle, inner sep = 1.5, fill = white, draw] () at (C) {};
\node[circle, inner sep = 1.5, fill = white, draw] () at (D) {};
\node[circle, inner sep = 1.5, fill = white, draw] () at (H) {};
\node[circle, inner sep = 1.5, fill = white, draw] () at (X) {};
\node[circle, inner sep = 1.5, fill = white, draw] () at (Y) {};
\node[circle, inner sep = 1.5, fill = white, draw] () at (Z) {};
\end{tikzpicture}}\hspace{1.5cm}
\subcaptionbox{\label{fig:subfig:c--}}
{\begin{tikzpicture}[scale = 0.8]
\coordinate (A) at (30:1);
\coordinate (B) at (150:1);
\coordinate (C) at (225:1);
\coordinate (D) at (-45:1);
\coordinate (H) at (90:1.414);
\coordinate (X) at (60:1.4);
\coordinate (Y) at (120:1.4);
\coordinate (T) at ($(A)!(H)!(X)$);
\coordinate (Z) at ($(H)!2!(T)$); 
\draw (A)--(Z)--(X)--(H)--(Y)--(B)--(C)--(D)--cycle;
\draw (X)--(A)--(H)--(B);
\node[circle, inner sep = 1.5, fill = white, draw] () at (A) {};
\node[circle, inner sep = 1.5, fill = white, draw] () at (B) {};
\node[circle, inner sep = 1.5, fill = white, draw] () at (C) {};
\node[circle, inner sep = 1.5, fill = white, draw] () at (D) {};
\node[circle, inner sep = 1.5, fill = white, draw] () at (H) {};
\node[circle, inner sep = 1.5, fill = white, draw] () at (X) {};
\node[circle, inner sep = 1.5, fill = white, draw] () at (Y) {};
\node[circle, inner sep = 1.5, fill = white, draw] () at (Z) {};
\end{tikzpicture}}
\caption{Forbidden configurations in \autoref{MRTHREE}.}
\label{E}
\end{figure}
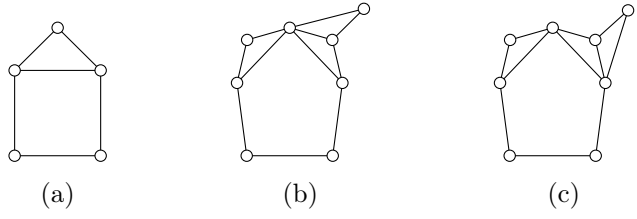

\begin{theorem}[Li and Wang \cite{Li2019}]\label{MRTHREE}
Let $G$ be a planar graph without subgraphs isomorphic to the configurations in \autoref{E}. Let $H$ be a cover of $G$ and $f$ be a function from $V(H)$ to $\{0, 1, 2\}$. If $f(v, 1) + f(v, 2) + \dots + f(v, s) \geq 4$ for each $v \in V(G)$, then $H$ has a strictly $f$-degenerate transversal. 
\end{theorem}

\begin{figure}%
\centering
\subcaptionbox{\label{fig:subfig:a}}
{\begin{tikzpicture}[scale = 0.8]
\def\s{1.3}
\coordinate (A) at (\s, 0);
\coordinate (B) at (60:\s);
\coordinate (O) at (0, 0);
\coordinate (C) at (-150:\s);
\coordinate (D) at (0,-\s);
\coordinate (E) at (\s, -\s);
\draw (A)--(B)--(O)--(C)--(D)--(E)--cycle;
\draw (A)--(O)--(D);
\node[circle, inner sep = 1.5, fill = white, draw] () at (O) {};
\node[circle, inner sep = 1.5, fill = white, draw] () at (A) {};
\node[circle, inner sep = 1.5, fill = white, draw] () at (B) {};
\node[circle, inner sep = 1.5, fill = white, draw] () at (C) {};
\node[circle, inner sep = 1.5, fill = white, draw] () at (D) {};
\node[circle, inner sep = 1.5, fill = white, draw] () at (E) {};
\end{tikzpicture}}\hspace{1.5cm}
\subcaptionbox{\label{fig:subfig:b}}
{\begin{tikzpicture}[scale = 0.8]
\def\s{1}
\coordinate (A) at (\s, 0);
\coordinate (B) at (0.5*\s, 0.7*\s);
\coordinate (O) at (0, 0);
\coordinate (E) at (\s, -\s);
\coordinate (D) at (0,-\s);
\coordinate (C) at (0.5*\s, -1.7*\s);
\draw (A)--(B)--(O)--(D)--(C)--(E)--cycle;
\draw (A)--(O);
\draw (D)--(E);
\node[circle, inner sep = 1.5, fill = white, draw] () at (O) {};
\node[circle, inner sep = 1.5, fill = white, draw] () at (A) {};
\node[circle, inner sep = 1.5, fill = white, draw] () at (B) {};
\node[circle, inner sep = 1.5, fill = white, draw] () at (C) {};
\node[circle, inner sep = 1.5, fill = white, draw] () at (D) {};
\node[circle, inner sep = 1.5, fill = white, draw] () at (E) {};
\end{tikzpicture}}\hspace{1.5cm}
\subcaptionbox{\label{fig:subfig:433}}
{\begin{tikzpicture}[scale = 0.8]
\def\s{1.2}
\coordinate (O) at (0, 0);
\coordinate (A) at (0, -\s);
\coordinate (B) at (-\s,-\s);
\coordinate (C) at (-\s,0);
\coordinate (D) at (135:\s);
\coordinate (E) at (0,\s);
\draw (O)--(A)--(B)--(C)--(D)--(E)--cycle;
\draw (C)--(O)--(D);
\node[circle, inner sep = 1.5, fill = white, draw] () at (O) {};
\node[circle, inner sep = 1.5, fill = white, draw] () at (A) {};
\node[circle, inner sep = 1.5, fill = white, draw] () at (B) {};
\node[circle, inner sep = 1.5, fill = white, draw] () at (C) {};
\node[circle, inner sep = 1.5, fill = white, draw] () at (D) {};
\node[circle, inner sep = 1.5, fill = white, draw] () at (E) {};
\end{tikzpicture}}\hspace{1.5cm}
\subcaptionbox{\label{fig:subfig:3444}}
{\begin{tikzpicture}[scale = 0.8]
\def\s{1}
\coordinate (O) at (0, 0);
\coordinate (E) at (\s, 0);
\coordinate (W) at (-\s,0);
\coordinate (N) at (0,\s);
\coordinate (S) at (0,-\s);
\coordinate (NW) at (-\s,\s);
\coordinate (SW) at (-\s,-\s);
\coordinate (SE) at (\s,-\s);
\draw (E)--(N)--(NW)--(SW)--(SE)--cycle;
\draw (N)--(S); 
\draw (E)--(W); 
\node[circle, inner sep = 1.5, fill = white, draw] () at (O) {};
\node[circle, inner sep = 1.5, fill = white, draw] () at (E) {};
\node[circle, inner sep = 1.5, fill = white, draw] () at (W) {};
\node[circle, inner sep = 1.5, fill = white, draw] () at (N) {};
\node[circle, inner sep = 1.5, fill = white, draw] () at (S) {};
\node[circle, inner sep = 1.5, fill = white, draw] () at (SE) {};
\node[circle, inner sep = 1.5, fill = white, draw] () at (SW) {};
\node[circle, inner sep = 1.5, fill = white, draw] () at (NW) {};
\end{tikzpicture}}\\
\subcaptionbox{\label{fig:subfig:533}}
{\begin{tikzpicture}[scale = 0.8]
\def\s{1}
\foreach \ang in {1, 2, 3, 4, 5}
{
\def\pointname{v\ang}
\coordinate (\pointname) at ($(\ang*360/5-54:\s)$);}
\coordinate (S) at ($(v2)!1!60:(v1)$);
\coordinate (S') at ($(v2)!1!60:(S)$);
\draw (v1)--(v2)--(v3)--(v4)--(v5)--cycle;
\draw (v1)--(S)--(v2);
\draw (S)--(S')--(v2);
\node[circle, inner sep = 1.5, fill = white, draw] () at (S) {};
\node[circle, inner sep = 1.5, fill = white, draw] () at (S') {};
\foreach \ang in {1, 2, 3, 4, 5}
{
\node[circle, inner sep = 1.5, fill = white, draw] () at (v\ang) {};
}
\end{tikzpicture}}\hspace{1cm}
\subcaptionbox{\label{fig:subfig:534f}}
{\begin{tikzpicture}[scale = 0.8]
\def\s{1}
\foreach \ang in {1, 2, 3, 4, 5}
{
\def\pointname{v\ang}
\coordinate (\pointname) at ($(\ang*360/5-54:\s)$);
}
\coordinate (S) at ($(v2)!1!60:(v1)$);
\coordinate (S') at ($(v2)!1!90:(S)$);
\coordinate (S'') at ($(S)!1!-90:(v2)$);
\draw (v1)--(v2)--(v3)--(v4)--(v5)--cycle;
\draw (v1)--(S)--(v2);
\draw (S)--(S'')--(S')--(v2);
\node[circle, inner sep = 1.5, fill = white, draw] () at (S) {};
\node[circle, inner sep = 1.5, fill = white, draw] () at (S') {};
\node[circle, inner sep = 1.5, fill = white, draw] () at (S'') {};
\foreach \ang in {1, 2, 3, 4, 5}
{
\node[circle, inner sep = 1.5, fill = white, draw] () at (v\ang) {};
}
\end{tikzpicture}}\hspace{1cm}
\subcaptionbox{\label{fig:subfig:g}}
{\begin{tikzpicture}[scale = 0.8]
\def\s{1.2}
\coordinate (O) at (0, 0);
\coordinate (v1) at (0:\s);
\coordinate (v2) at (60:\s);
\coordinate (v3) at (120:\s);
\coordinate (v4) at (180:\s);
\coordinate (OO) at ($(v3)!1!60:(v2)$);
\coordinate (A) at ($(OO)!1!-60:(v3)$);
\coordinate (B) at ($(OO)!1!-60:(A)$);
\coordinate (C) at ($(OO)!1!-60:(B)$);
\coordinate (D) at ($(OO)!1!-60:(C)$);
\draw (v1)--(v2)--(v3)--(v4)--(O)--cycle;
\draw (O)--(v2);
\draw (O)--(v3);
\draw (v3)--(A)--(B)--(C)--(D)--(v2);
\node[circle, inner sep = 1.5, fill = white, draw] () at (O) {};
\node[circle, inner sep = 1.5, fill = white, draw] () at (v3) {};
\node[circle, inner sep = 1.5, fill = white, draw] () at (v1) {};
\node[circle, inner sep = 1.5, fill = white, draw] () at (v2) {};
\node[circle, inner sep = 1.5, fill = white, draw] () at (v4) {};
\node[circle, inner sep = 1.5, fill = white, draw] () at (A) {};
\node[circle, inner sep = 1.5, fill = white, draw] () at (B) {};
\node[circle, inner sep = 1.5, fill = white, draw] () at (C) {};
\node[circle, inner sep = 1.5, fill = white, draw] () at (D) {};
\end{tikzpicture}}\hspace{1cm}
\subcaptionbox{\label{fig:subfig:h}}
{\begin{tikzpicture}[scale = 0.8]
\def\s{1.2}
\coordinate (O) at (0, 0);
\coordinate (v1) at (0:\s);
\coordinate (v2) at (60:\s);
\coordinate (v3) at (120:\s);
\coordinate (v4) at (180:\s);
\coordinate (OO) at ($(v2)!1!60:(v1)$);
\coordinate (A) at ($(OO)!1!-60:(v2)$);
\coordinate (B) at ($(OO)!1!-60:(A)$);
\coordinate (C) at ($(OO)!1!-60:(B)$);
\coordinate (D) at ($(OO)!1!-60:(C)$);
\draw (v1)--(v2)--(v3)--(v4)--(O)--cycle;
\draw (O)--(v2);
\draw (O)--(v3);
\draw (v2)--(A)--(B)--(C)--(D)--(v1);
\node[circle, inner sep = 1.5, fill = white, draw] () at (O) {};
\node[circle, inner sep = 1.5, fill = white, draw] () at (v3) {};
\node[circle, inner sep = 1.5, fill = white, draw] () at (v1) {};
\node[circle, inner sep = 1.5, fill = white, draw] () at (v2) {};
\node[circle, inner sep = 1.5, fill = white, draw] () at (v4) {};
\node[circle, inner sep = 1.5, fill = white, draw] () at (A) {};
\node[circle, inner sep = 1.5, fill = white, draw] () at (B) {};
\node[circle, inner sep = 1.5, fill = white, draw] () at (C) {};
\node[circle, inner sep = 1.5, fill = white, draw] () at (D) {};
\end{tikzpicture}}
\caption{Forbidden configurations in \autoref{PAIRWISE3456}, \autoref{PAIRWISE3456-Weak3} and \autoref{PAIRWISE3456-SFDT}.}
\label{FIGPAIRWISE3456}
\end{figure}
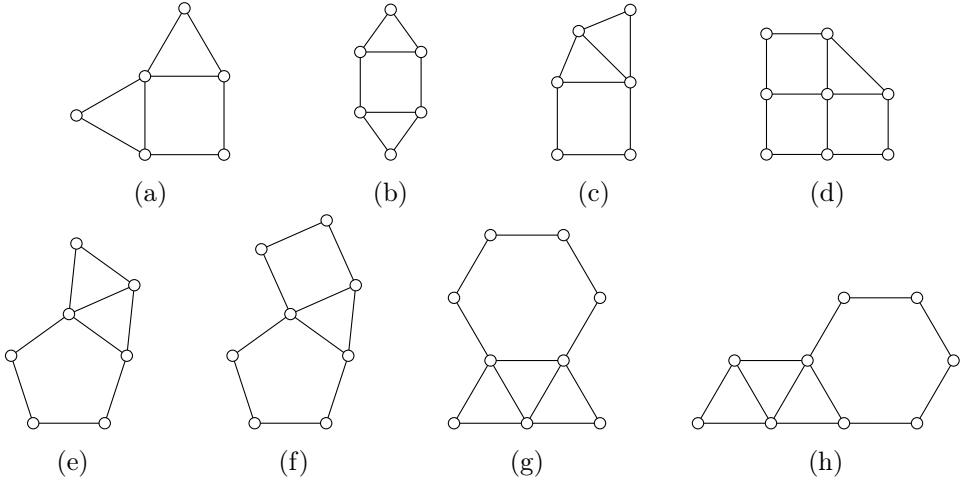

We first consider plane graphs without any configuration in \autoref{FIGPAIRWISE3456}. A $4^{-}$-cycle is \emph{good} if there is no vertex having four neighbors on the cycle, otherwise it is \emph{a bad cycle}. Note that every $3$-cycle is good. In a plane graph, for a fixed cycle, a vertex is \emph{internal} if it does not belong to the fixed cycle, and a subgraph $\Gamma$ is \emph{internal} if $\Gamma$ has no common vertices with the fixed cycle.

For a plane graph and a cycle $K$, we use $\int(K)$ to denote the set of vertices inside of $K$, and $\ext(K)$ to denote the set of vertices outside of $K$. If neither $\int(K)$ nor $\ext(K)$ is empty, then we say that $K$ is a \emph{separating cycle}.

\begin{theorem}\label{PAIRWISE3456}
Let $G$ be a connected plane graph without any configuration in \autoref{FIGPAIRWISE3456}, and let $[x_{1}x_{2}\dots x_{l}]$ be a good $4^{-}$-cycle in $G$. Then one of the followings holds:
\begin{enumerate}[label = (\roman*)]
\item $V(G) = \{x_{1}, x_{2}, \dots, x_{l}\}$; or
\item there is an internal $3^{-}$-vertex; or
\item there is an internal induced subgraph isomorphic to \autoref{Kite}; or
\item there is an internal induced subgraph isomorphic to \autoref{F35}. 
\end{enumerate}
\end{theorem}

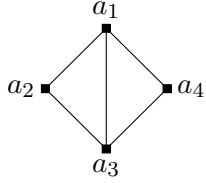
\begin{figure}%
\centering
\begin{tikzpicture}
\def\s{0.8}
\coordinate (E) at (\s, 0);
\coordinate (N) at (0, \s);
\coordinate (W) at (-\s, 0);
\coordinate (S) at (0, -\s);
\draw (N)node[above]{$a_{1}$}--(W)node[left]{$a_{2}$}--(S)node[below]{$a_{3}$}--(E)node[right]{$a_{4}$}--cycle;
\draw (N)--(S);
\node[rectangle, inner sep = 1.5, fill, draw] () at (E) {};
\node[rectangle, inner sep = 1.5, fill, draw] () at (N) {};
\node[rectangle, inner sep = 1.5, fill, draw] () at (W) {};
\node[rectangle, inner sep = 1.5, fill, draw] () at (S) {};
\end{tikzpicture}
\caption{$a_{2}a_{4} \notin E(G)$. Here and in all figures below, a solid quadrilateral represents a $4$-vertex, a solid pentagon represents a $5$-vertex, and a solid hexagon represents a $6$-vertex.}
\label{Kite}
\end{figure}

\begin{figure}%
\centering
\begin{tikzpicture}
\def\s{0.8}
\coordinate (v1) at (45: 1.414*\s);
\coordinate (v2) at (135: 1.414*\s);
\coordinate (v3) at (225: 1.414*\s);
\coordinate (v4) at (-45: 1.414*\s);
\coordinate (A) at (2*\s, 0);
\coordinate (B) at (-2*\s, 0);
\draw (v1)--(v2)--(B)--(v3)--(v4)--(A)--cycle;
\draw (v2)--(v3);
\node[rectangle, inner sep = 1.5, fill, draw] () at (v1) {};
\node[rectangle, inner sep = 1.5, fill, draw] () at (v2) {};
\node[rectangle, inner sep = 1.5, fill, draw] () at (v3) {};
\node[rectangle, inner sep = 1.5, fill, draw] () at (v4) {};
\node[rectangle, inner sep = 1.5, fill, draw] () at (A) {};
\node[rectangle, inner sep = 1.5, fill, draw] () at (B) {};
\end{tikzpicture}
\caption{Reducible configuration.}
\label{F35}
\end{figure}
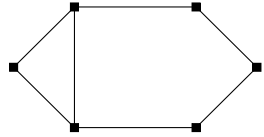

Using this structural result, we can prove the following three theorems. 

\begin{theorem}\label{PAIRWISE3456-AT}
Let $G$ be a plane graph without any configuration in \autoref{FIGPAIRWISE3456}. Then its Alon-Tarsi number is at most four. 
\end{theorem}

\begin{theorem}\label{PAIRWISE3456-Weak3}
Let $G$ be a planar graph without any configuration in \autoref{FIGPAIRWISE3456}. Then $G$ is weakly $3$-degenerate. 
\end{theorem}

\begin{theorem}\label{PAIRWISE3456-SFDT}
Let $G$ be a planar graph without any configuration in \autoref{FIGPAIRWISE3456}. Let $H$ be a cover of $G$ and $f$ be a function from $V(H)$ to $\{0, 1, 2\}$. If $f(v, 1) + f(v, 2) + \dots + f(v, s) \geq 4$ for each $v \in V(G)$, then $H$ has a strictly $f$-degenerate transversal.  
\end{theorem}

\begin{corollary}
Let $G$ be a planar graph. For each vertex $v \in V(G)$, there exists an integer $k_{v} \in \{3, 4, 5, 6\}$ such that $v$ is not contained in any $k_{v}$-cycle. Then its DP-paint number is at most $4$, and it has list vertex arboricity at most two. 
\end{corollary}

\begin{corollary}
If $G$ is a planar graph without pairwise adjacent $3$-, $4$, $5$- and $6$-cycles, then its DP-paint number is at most $4$, and it has list vertex arboricity at most two. 
\end{corollary}

This strengthens all the results in \cite{MR4212281,MR4078909,MR4089638,MR3969022,MR3996735,Zhang2019} for DP-chromatic number and the results in \cite{MR4135602, MR4112063, MR3979933, MR3699856} for (list) vertex arboricity.

\autoref{PAIRWISE3456-SFDT} cannot be strengthened from ``planar graph'' to ``toroidal graph''. For each integer $k \geq 6$, Choi \cite[Theorem~4.1]{MR3634481} constructed an infinite family of toroidal graphs without cycles of lengths from $6$ to $k$ with chromatic number five. Note that the infinite family of graphs are embeddable on any surface, orientable or non-orientable, except the plane and projective plane. Note that the family of toroidal graphs do not contain any configuration in \autoref{FIGPAIRWISE3456}. Then the family of graphs are not $4$-colorable, and then it has no strictly $f$-degenerate transversals as in \autoref{PAIRWISE3456-SFDT}.

The (list) vertex arboricity of toroidal graphs without some small cycles or adjacent small cycles are extensively investigated.

\begin{theorem}\label{WithoutTriangle}\mbox{}
\begin{enumerate}
\item (Kronk and Mitchem \cite{MR0360334})
Every toroidal graph without $3$-cycles has vertex arboricity at most two. 
\item (Zhang \cite{MR3570576})
Every toroidal graph without $5$-cycles has list vertex arboricity at most two. 
\item (Choi and Zhang \cite{MR3233411})
Every toroidal graph without $4$-cycles has vertex arboricity at most two. 
\item (Huang, Chen and Wang \cite{MR3320048})
Every toroidal graph without $3$-cycles adjacent to $5$-cycles has list vertex arboricity at most two. 
\item (Chen, Huang and Wang \cite{MR3508765})
Every toroidal graph without $3$-cycles adjacent to $4$-cycles has list vertex arboricity at most two. 
\end{enumerate}
\end{theorem}

More results on (list) vertex arboricity, we refer the reader to \cite{MR2889524, MR2926103, MR2408378, MR1871345, MR3164029}. For other DP-$4$-coloring of planar graphs, we refer the reader to \cite{MR3802151, MR3881665, MR4294211, MR4362322}.

The second class of graphs is the class of toroidal graphs without any configuration in \autoref{A345}.

\begin{theorem}\label{T345}
Let $G$ be a connected toroidal graph without any configuration in \autoref{A345}. Then one of the followings holds:
\begin{enumerate}[label = (\roman*)]
\item $\delta(G) \leq 3$;
\item there is a subgraph isomorphic to the configuration in \autoref{F35}; 
\item there is a subgraph isomorphic to a configuration in \autoref{RC}; 
\item $G$ is $4$-regular. 
\end{enumerate}
\end{theorem}

\begin{figure}%
\centering
\subcaptionbox{\label{fig:subfig:TA}}{\begin{tikzpicture}[scale = 0.8]
\coordinate (A) at (45:1);
\coordinate (B) at (135:1);
\coordinate (C) at (225:1);
\coordinate (D) at (-45:1);
\coordinate (H) at (90:1.414);
\draw (A)--(H)--(B)--(C)--(D)--cycle;
\draw (A)--(B);
\node[circle, inner sep = 1.5, fill = white, draw] () at (A) {};
\node[circle, inner sep = 1.5, fill = white, draw] () at (B) {};
\node[circle, inner sep = 1.5, fill = white, draw] () at (C) {};
\node[circle, inner sep = 1.5, fill = white, draw] () at (D) {};
\node[circle, inner sep = 1.5, fill = white, draw] () at (H) {};
\end{tikzpicture}}\hspace{1.5cm}
\subcaptionbox{\label{fig:subfig:TB}}{\begin{tikzpicture}[scale = 0.8]
\coordinate (A) at (45:1);
\coordinate (B) at (135:1);
\coordinate (C) at (225:1);
\coordinate (D) at (-45:1);
\coordinate (H) at (90:1.414);
\coordinate (X) at ($(A)+(0, 1)$);
\coordinate (Y) at (45:2);
\draw (A)--(Y)--(X)--(H)--(B)--(C)--(D)--cycle;
\draw (X)--(A)--(H);
\node[circle, inner sep = 1.5, fill = white, draw] () at (A) {};
\node[circle, inner sep = 1.5, fill = white, draw] () at (B) {};
\node[circle, inner sep = 1.5, fill = white, draw] () at (C) {};
\node[circle, inner sep = 1.5, fill = white, draw] () at (D) {};
\node[circle, inner sep = 1.5, fill = white, draw] () at (H) {};
\node[circle, inner sep = 1.5, fill = white, draw] () at (X) {};
\node[circle, inner sep = 1.5, fill = white, draw] () at (Y) {};
\end{tikzpicture}}
\caption{Forbidden configurations in \autoref{T345}, \autoref{T345-Weak3} and \autoref{T345-SFDT}.}
\label{A345}
\end{figure}
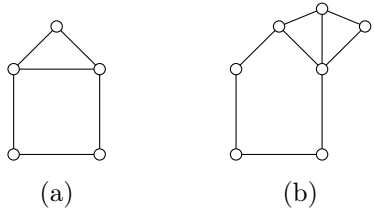

\begin{figure}%
\centering
\def\s{1.3}
\subcaptionbox{$t = 7$ \label{fig:subfig:-b}}[0.3\linewidth]
{
\begin{tikzpicture}
\def\n{6}
\pgfmathsetmacro \i{\n-2}
\foreach \x in {1,...,\i}
{
\def\pointnameA{A\x}
\coordinate (\pointnameA) at ($(\x*360/\n+120:0.73*\s)$);
}
\coordinate (A0) at (90:0.73);

\foreach \x in {0,...,23}
{
\def\pointnameB{B\x}
\coordinate (\pointnameB) at ($(\x*90/\n+120:1.4*\s)$);
}

\coordinate (A2B9) at ($(A2)!.5!(B9)$);
\coordinate (A2B9') at ($(A2B9)!1.73!90:(A2)$);
\coordinate (B9B11) at ($(B9)!.5!(B11)$);
\coordinate (B9B11') at ($(B9B11)!1.73!-90:(B11)$);
\coordinate (A3B11) at ($(A3)!.5!(B11)$);
\coordinate (A3B11') at ($(A3B11)!1.73!90:(B11)$);

\draw[opacity=0.4, gray, line width = 32pt, line join=round] (A1)--(A2)--(B9)--(B11)--(A3)--(A4)--(0, 0)--cycle; 
\foreach \x in {0,...,\i}
{
\draw (0, 0)node[above]{$a_{1}$}--(A\x);
}
\draw (A2)--(A2B9')--(B9);
\draw (B9)--(B9B11')--(B11);
\draw (A3)--(A3B11')--(B11);
\draw (A2)node[left]{$a_{3}$}--(B9)node[left]{$a_{4}$}--(B11)node[right]{$a_{5}$}--(A3)node[right]{$a_{6}$};
\draw (A1)node[above]{$a_{2}$}--(A2);
\draw (A3)--(A4)node[above]{$a_{7}$};

\node[regular polygon, inner sep = 1.5, fill=blue, draw=blue] () at (0, 0) {};
\node[rectangle, inner sep = 1.5, fill = white, draw] () at (A0) {};
\node[rectangle, inner sep = 1.5, fill, draw] () at (B9) {};
\node[rectangle, inner sep = 1.5, fill, draw] () at (B11) {};
\node[rectangle, inner sep = 1.5, fill, draw] () at (A1) {};
\node[rectangle, inner sep = 1.5, fill, draw] () at (A2) {};
\node[rectangle, inner sep = 1.5, fill, draw] () at (A3) {};
\node[rectangle, inner sep = 1.5, fill, draw] () at (A4) {};
\node[rectangle, inner sep = 1.5, fill = white, draw] () at (A2B9') {};
\node[rectangle, inner sep = 1.5, fill = white, draw] () at (B9B11') {};
\node[rectangle, inner sep = 1.5, fill = white, draw] () at (A3B11') {};
\end{tikzpicture}}
\subcaptionbox{$t = 9$ \label{fig:subfig:-a}}[0.3\linewidth]
{
\begin{tikzpicture}
\def\n{6}
\pgfmathsetmacro \i{\n-2}
\foreach \x in {0,...,\i}
{
\def\pointnameA{A\x}
\coordinate (\pointnameA) at ($(\x*360/\n+120:0.73*\s)$);
}

\foreach \x in {0,...,23}
{
\def\pointnameB{B\x}
\coordinate (\pointnameB) at ($(\x*90/\n+120:1.4*\s)$);
}

\coordinate (A2B9) at ($(A2)!.5!(B9)$);
\coordinate (A2B9') at ($(A2B9)!1.73!90:(A2)$);
\coordinate (B9B11) at ($(B9)!.5!(B11)$);
\coordinate (B9B11') at ($(B9B11)!1.73!-90:(B11)$);
\coordinate (A3B11) at ($(A3)!.5!(B11)$);
\coordinate (A3B11') at ($(A3B11)!1.73!90:(B11)$);

\draw[opacity=0.4, gray, line width = 32pt, line join=round] (A0)--(B1)--(B3)--(A1)--(A2)--(B9)--(B11)--(A3)--(0, 0)--cycle; 
\foreach \x in {0,...,\i}
{
\draw (0, 0)node[above]{$a_{1}$}--(A\x);
}

\draw (A2)--(A2B9')--(B9);
\draw (B9)--(B9B11')--(B11);
\draw (A3)--(A3B11')--(B11);
\draw (A0)node[above]{$a_{2}$}--(B1)node[above]{$a_{3}$}--(B3)node[left]{$a_{4}$}--(A1)node[left]{$a_{5}$}; 
\draw (A2)node[left]{$a_{6}$}--(B9)node[left]{$a_{7}$}--(B11)node[right]{$a_{8}$}--(A3)node[right]{$a_{9}$};
\draw (A1)--(A2);
\draw (A3)--(A4);

\node[regular polygon, inner sep = 1.5, fill=blue, draw=blue] () at (0, 0) {};
\node[rectangle, inner sep = 1.5, fill, draw] () at (A0) {};
\node[rectangle, inner sep = 1.5, fill, draw] () at (B1) {};
\node[rectangle, inner sep = 1.5, fill, draw] () at (B3) {};
\node[rectangle, inner sep = 1.5, fill, draw] () at (B9) {};
\node[rectangle, inner sep = 1.5, fill, draw] () at (B11) {};
\node[rectangle, inner sep = 1.5, fill, draw] () at (A1) {};
\node[rectangle, inner sep = 1.5, fill, draw] () at (A2) {};
\node[rectangle, inner sep = 1.5, fill, draw] () at (A3) {};
\node[rectangle, inner sep = 1.5, fill = white, draw] () at (A4) {};
\node[rectangle, inner sep = 1.5, fill = white, draw] () at (A2B9') {};
\node[rectangle, inner sep = 1.5, fill = white, draw] () at (B9B11') {};
\node[rectangle, inner sep = 1.5, fill = white, draw] () at (A3B11') {};
\end{tikzpicture}}
\subcaptionbox{$t = 10$ \label{fig:subfig:-c}}[0.3\linewidth]
{
\begin{tikzpicture}
\def\n{6}
\pgfmathsetmacro \i{\n-1}
\foreach \x in {0,...,\i}
{
\def\pointnameA{A\x}
\coordinate (\pointnameA) at ($(\x*360/\n+120:0.73*\s)$);
}

\foreach \x in {0,...,23}
{
\def\pointnameB{B\x}
\coordinate (\pointnameB) at ($(\x*90/\n+120:1.4*\s)$);
}

\coordinate (A2B9) at ($(A2)!.5!(B9)$);
\coordinate (A2B9') at ($(A2B9)!1.73!90:(A2)$);
\coordinate (B9B11) at ($(B9)!.5!(B11)$);
\coordinate (B9B11') at ($(B9B11)!1.73!-90:(B11)$);
\coordinate (A3B11) at ($(A3)!.5!(B11)$);
\coordinate (A3B11') at ($(A3B11)!1.73!90:(B11)$);

\draw[opacity=0.4, gray, line width = 32pt, line join=round] (A0)--(B1)--(B3)--(A1)--(A2)--(B9)--(B11)--(A3)--(0, 0)--(A5)--cycle; 
\foreach \x in {0,...,\i}
{
\draw (0, 0)node[above]{$a_{1}$}--(A\x);
}
\draw (A2)--(A2B9')--(B9);
\draw (B9)--(B9B11')--(B11);
\draw (A3)--(A3B11')--(B11);
\draw (A0)node[above]{$a_{3}$}--(B1)node[above]{$a_{4}$}--(B3)node[left]{$a_{5}$}--(A1)node[left]{$a_{6}$}; 
\draw (A2)node[left]{$a_{7}$}--(B9)node[left]{$a_{8}$}--(B11)node[right]{$a_{9}$}--(A3)node[right]{$a_{10}$};
\draw (A1)--(A2);
\draw (A3)--(A4);
\draw (A5)node[right]{$a_{2}$}--(A0);

\node[regular polygon, regular polygon sides=6, inner sep = 1.5, fill=brown, draw=brown] () at (0, 0) {};
\node[rectangle, inner sep = 1.5, fill, draw] () at (A0) {};
\node[rectangle, inner sep = 1.5, fill, draw] () at (B1) {};
\node[rectangle, inner sep = 1.5, fill, draw] () at (B3) {};
\node[rectangle, inner sep = 1.5, fill, draw] () at (B9) {};
\node[rectangle, inner sep = 1.5, fill, draw] () at (B11) {};
\node[rectangle, inner sep = 1.5, fill, draw] () at (A1) {};
\node[rectangle, inner sep = 1.5, fill, draw] () at (A2) {};
\node[rectangle, inner sep = 1.5, fill, draw] () at (A3) {};
\node[rectangle, inner sep = 1.5, fill = white, draw] () at (A4) {};
\node[rectangle, inner sep = 1.5, fill, draw] () at (A5) {};
\node[rectangle, inner sep = 1.5, fill = white, draw] () at (A2B9') {};
\node[rectangle, inner sep = 1.5, fill = white, draw] () at (B9B11') {};
\node[rectangle, inner sep = 1.5, fill = white, draw] () at (A3B11') {};
\end{tikzpicture}}
\caption{Note that $|N_{G}(a_{t}) \cap \{a_{1}, \dots, a_{t-1}\}| = 2$.}
\label{RC}
\end{figure}
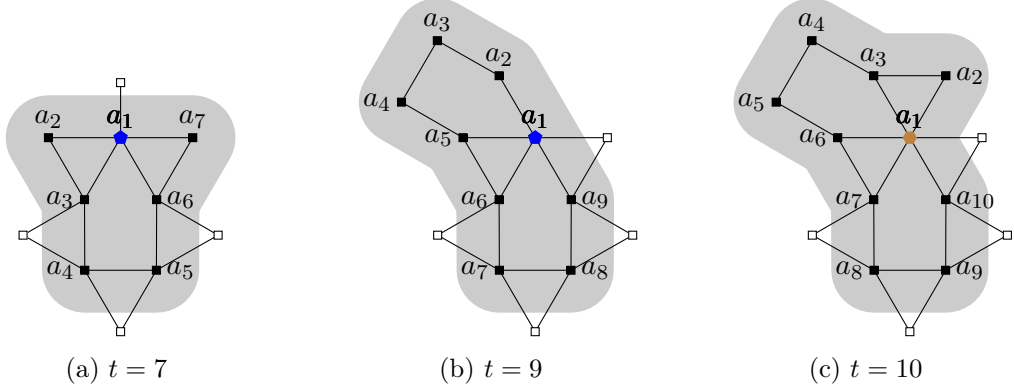

Similarly, we can prove the following two theorems by using the above structural result. 

\begin{theorem}\label{T345-Weak3}
Let $G$ be a toroidal graph without any configuration in \autoref{A345}. Then $G$ is weakly $3$-degenerate. 
\end{theorem}

\begin{theorem}\label{T345-SFDT}
Let $G$ be a toroidal graph without any configuration in \autoref{A345}. Let $H$ be a cover of $G$ and $f$ be a function from $V(H)$ to $\{0, 1, 2\}$. If $f(v, 1) + f(v, 2) + \dots + f(v, s) \geq 4$ for each $v \in V(G)$, then $H$ has a strictly $f$-degenerate transversal. 
\end{theorem}

\autoref{T345-SFDT} improves all the results in \autoref{WithoutTriangle}. Note that if $G$ contains no configurations in \autoref{A345}, then it contains no configurations in \autoref{FIGPAIRWISE3456}. 
\begin{corollary}
For every toroidal graph without any configuration in \autoref{A345}, the list vertex arboricity is at most two, and the DP-chromatic number is at most four. 
\end{corollary}

We say that two cycles are \emph{intersecting} if they have at least one common vertex. The third class of graphs is the class of planar graphs without intersecting 5-cycles. Cai \etal \cite{MR3761240} proved that planar graphs without intersecting 5-cycles have vertex arboricity at most two. Hu and Wu \cite{MR3648207} proved that planar graphs without intersecting 5-cycles are 4-choosable. Recently, Li \etal \cite{MR4362322} improved Hu and Wu's result by showing that planar graphs without intersecting 5-cycles are DP-4-colorable.

\begin{theorem}\label{Intersecting}
Let $G$ be a connected plane graph without intersecting $5$-cycles, and let $C_{0}$ be a triangle. Then one of the followings holds:
\begin{enumerate}[label = (\arabic*)]
\item $V(G) = V(C_{0})$; or
\item there is an internal $3^{-}$-vertex; or
\item there is an internal induced subgraph isomorphic to a configuration in \autoref{Kite}, \autoref{F35}, \autoref{fig:subfig:RC1-a}, \autoref{RC-1}, and \autoref{RC-2}.
\end{enumerate}
\end{theorem}

\begin{figure}%
\centering
\begin{tikzpicture}
\def\s{1}
\coordinate (A) at (45:\s);
\coordinate (B) at (135:\s);
\coordinate (C) at (225:\s);
\coordinate (D) at (-45:\s);
\coordinate (H) at (90:1.414*\s);
\coordinate (X) at ($(A)+(0, \s)$);
\coordinate (Y) at (45:2*\s);
\draw (A)node[right]{$a_{4}$}--(Y)node[right]{$a_{3}$}--(X)node[above]{$a_{2}$}--(H)node[left]{$a_{1}$}--(B)node[left]{$a_{7}$}--(C)node[below]{$a_{6}$}--(D)node[below]{$a_{5}$}--cycle;
\draw (X)--(A)--(H);
\node[rectangle, inner sep = 1.5, fill, draw] () at (A) {};
\node[rectangle, inner sep = 1.5, fill, draw] () at (B) {};
\node[rectangle, inner sep = 1.5, fill, draw] () at (C) {};
\node[rectangle, inner sep = 1.5, fill, draw] () at (D) {};
\node[rectangle, inner sep = 1.5, fill, draw] () at (H) {};
\node[regular polygon, inner sep = 1.5, fill=blue, draw=blue] () at (X) {};
\node[rectangle, inner sep = 1.5, fill, draw] () at (Y) {};
\end{tikzpicture}
\caption{Note that $|N_{G}(a_{7}) \cap \{a_{1}, \dots, a_{6}\}| = 2$.}
\label{fig:subfig:RC1-a}
\end{figure}

\begin{figure}%
\centering
\begin{tikzpicture}
\def\s{1}
\coordinate (O) at (0, 0);
\coordinate (v1) at (0:\s);
\coordinate (v2) at (60:\s);
\coordinate (v3) at (120:\s);
\coordinate (v4) at (180:\s);
\draw (v1)node[below]{$a_{5}$}--(v2)node[above]{$a_{1}$}--(v3)node[above]{$a_{2}$}--(v4)node[below]{$a_{3}$}--(O)node[below]{$a_{4}$}--cycle;
\draw (O)--(v2);
\draw (O)--(v3);
\node[rectangle, inner sep = 1.5, fill, draw] () at (O) {};
\node[regular polygon, inner sep = 1.5, fill=blue, draw=blue] () at (v3) {};
\node[rectangle, inner sep = 1.5, fill, draw] () at (v1) {};
\node[rectangle, inner sep = 1.5, fill, draw] () at (v2) {};
\node[rectangle, inner sep = 1.5, fill, draw] () at (v4) {};
\end{tikzpicture}
\caption{$a_{2}a_{5}, a_{3}a_{5} \notin E(G)$.}
\label{RC-1}
\end{figure}
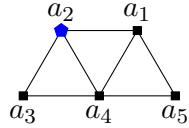

\begin{figure}%
\centering
\subcaptionbox{\label{fig:subfig:RC-2-a}$a_{3}a_{5} \notin E(G)$.}
{\begin{tikzpicture}
\def\s{1}
\coordinate (O) at (0, 0);
\coordinate (v1) at (45: \s);
\coordinate (v2) at (135: \s);
\coordinate (v3) at (225: \s);
\coordinate (v4) at (-45: \s);
\coordinate (A) at (-1.414*\s, 0);
\coordinate (B) at (1.414*\s, 0);
\draw (v1)node[above]{$a_{2}$}--(O)node[below]{$a_{3}$}--(v3)node[below]{$a_{6}$}--(A)node[left]{$a_{7}$}--(v2)node[above]{$a_{1}$}--(v4)node[below]{$a_{4}$}--(B)node[right]{$a_{5}$}--cycle;
\draw (v2)--(v3);
\draw (v1)--(v4);
\node[rectangle, inner sep = 1.5, fill, draw] () at (v1) {};
\node[rectangle, inner sep = 1.5, fill, draw] () at (v2) {};
\node[rectangle, inner sep = 1.5, fill, draw] () at (v3) {};
\node[rectangle, inner sep = 1.5, fill, draw] () at (v4) {};
\node[rectangle, inner sep = 1.5, fill, draw] () at (A) {};
\node[rectangle, inner sep = 1.5, fill, draw] () at (B) {};
\node[regular polygon, inner sep = 1.5, fill=blue, draw=blue] () at (O) {};
\end{tikzpicture}}\hspace{1cm}
\subcaptionbox{\label{fig:subfig:RC-2-b}$a_{3}a_{5} \notin E(G)$.}
{\begin{tikzpicture}
\def\s{1}
\coordinate (O) at (0, 0);
\coordinate (v1) at (45: \s);
\coordinate (v2) at (135: \s);
\coordinate (v3) at (225: \s);
\coordinate (v4) at (-45: \s);
\coordinate (A) at (-1.414*\s, 0);
\coordinate (B) at (1.414*\s, 0);
\draw (v1)node[above]{$a_{3}$}--(O)node[below]{$a_{4}$}--(v3)node[below]{$a_{6}$}--(A)node[left]{$a_{7}$}--(v2)node[above]{$a_{1}$}--(v4)node[below]{$a_{5}$}--(B)node[right]{$a_{2}$}--cycle;
\draw (v2)--(v3);
\draw (O)--(B);
\node[rectangle, inner sep = 1.5, fill, draw] () at (v1) {};
\node[rectangle, inner sep = 1.5, fill, draw] () at (v2) {};
\node[rectangle, inner sep = 1.5, fill, draw] () at (v3) {};
\node[rectangle, inner sep = 1.5, fill, draw] () at (v4) {};
\node[rectangle, inner sep = 1.5, fill, draw] () at (A) {};
\node[rectangle, inner sep = 1.5, fill, draw] () at (B) {};
\node[regular polygon, inner sep = 1.5, fill=blue, draw=blue] () at (O) {};
\end{tikzpicture}}
\caption{Note that $|N_{G}(a_{7}) \cap \{a_{1}, \dots, a_{6}\}| = 2$.}
\label{RC-2}
\end{figure}

\begin{theorem}\label{Intersecting-AT}
Let $G$ be a planar graph without intersecting $5$-cycles. Then its Alon-Tarsi number is at most four. 
\end{theorem}

\begin{theorem}\label{Intersecting-Weak3}
Let $G$ be a planar graph without intersecting $5$-cycles. Then $G$ is weakly $3$-degenerate. 
\end{theorem}

\begin{restatable}{theorem}{IntersectingSFDT}\label{IntersectingSFDT}
Let $G$ be a planar graph without intersecting $5$-cycles. Let $H$ be a cover of $G$ and $f$ be a function from $V(H)$ to $\{0, 1, 2\}$. If $f(v, 1) + f(v, 2) + \dots + f(v, s) \geq 4$ for each $v \in V(G)$, then $H$ has a strictly $f$-degenerate transversal. 
\end{restatable}

\begin{corollary}
Planar graphs without intersecting 5-cycles have list vertex arboricity at most two. 
\end{corollary}

We say that two cycles are \emph{normally adjacent} if their intersection is isomorphic to $K_{2}$. Two faces are \emph{normally adjacent} if the intersection of their boundaries is isomorphic to $K_{2}$.

Let $G$ be a plane graph. We use $F(G)$ to denote the set of faces in $G$. Let $\triangledown(f)$ denote the number of $3$-faces adjacent to a face $f$, and let $\triangledown(v)$ denote the number of 3-faces incident with a vertex $v$. A $d$-vertex (resp. $d^{+}$-vertex, $d^{-}$-vertex) is a vertex of degree $k$ (resp. at least $k$, at most $k$). A similar definition can be applied to faces. For a face $f \in F(G)$, we write $f = [v_{1}v_{2}\dots v_{n}]$ if $v_{1}, v_{2}, \dots, v_{n}$ are the vertices on the boundary of $f$ in a cyclic order. For convenience, we also use $[v_{1}v_{2}\dots v_{n}]$ to denote a cycle with vertices $v_{1}, v_{2}, \dots, v_{n}$ in a cyclic order. An $m$-face $[v_{1}v_{2}\dots v_{n}]$ is a $(d_{1}, d_{2}, \dots, d_{m})$-face if $d(v_{i}) = d_{i}$ for all $1 \leq i \leq m$. Throughout this paper, we use $\tau(x \rightarrow y)$ to denote the number of charges that the element $x$ sends to the element $y$.

The rest of the paper is organized as follows. In \autoref{sec:2}, we prove the structural result \autoref{PAIRWISE3456}; in \autoref{sec:3}, we prove the Alon-Tarsi result \autoref{PAIRWISE3456-AT}; in \autoref{sec:4}, we prove \autoref{PAIRWISE3456-Weak3}; in \autoref{sec:5}, we give some structural results for the critical graphs w.r.t. strictly $f$-degenerate transversal; in \autoref{sec:6}, we prove \autoref{PAIRWISE3456-SFDT}; in \autoref{sec:7}, we prove \autoref{T345}, \autoref{T345-Weak3} and \autoref{T345-SFDT}; in the final section, we prove \autoref{Intersecting}, \autoref{Intersecting-Weak3} and \autoref{IntersectingSFDT}.

\section{Proof of \autoref{PAIRWISE3456}}
\label{sec:2}
In this section, we prove the structural result, \autoref{PAIRWISE3456}. Let $G$ together with a good $4^{-}$-cycle $C_{0}=[x_{1}x_{2}\dots x_{l}]$ be a counterexample to the statement such that $|V(G)| + |E(G)|$ is minimum.

\begin{enumerate}[label = \textbf{(\arabic*)}, ref = (\arabic*)]
\item\label{CHORDLESS} $[x_{1}x_{2}\dots x_{l}]$ has no chords. 
\end{enumerate}

\begin{enumerate}[label = \textbf{(\arabic*)}, ref = (\arabic*), resume]
\item\label{SEPARATING} There are no separating good $4^{-}$-cycles in $G$. Then every induced good $4^{-}$-cycle bounds a face. So we may assume that $[x_{1}\dots x_{l}]$ is the boundary of the outer face $D$. 
\end{enumerate}
\begin{proof}
If $C_{0}$ is a separating cycle, then $G - \ext(C_{0})$ is a smaller counterexample, a contradiction. Since $C_{0}$ is an induced non-separating cycle, we may assume that $C_{0}$ bounds the outer face $D$. Assume $C$ is a separating good $4^{-}$-cycle in $G$. It is obvious that $V(C_{0}) \cap \int(C) = \emptyset$, the graph $G - \ext(C)$ is a smaller counterexample, a contradiction. Therefore, there are no separating good $4^{-}$-cycles in $G$.
\end{proof}

A face is \emph{internal} if none of its incident vertices is on the outer cycle. For simplicity, an internal $4$-face is called a \emph{$4^{*}$-face}, and an internal $5$-face is called a \emph{$5^{*}$-face}.

Let $[w_{1}w_{2}w_{3}w_{4}]$ be a separating bad $4$-cycle. Then there exists a vertex $w$ adjacent to each of $w_{1}, w_{2}, w_{3}$ and $w_{4}$. By \ref{SEPARATING}, each of $[ww_{1}w_{2}]$, $[ww_{2}w_{3}]$, $[ww_{3}w_{4}]$ and $[ww_{4}w_{1}]$ bounds a $3$-face. We say that \emph{a region bounded by $[w_{1}w_{2}w_{3}w_{4}]$ consists of four $3$-faces}. Note that a separating bad $4$-cycle has a feature that every edge on the cycle is incident with a $3$-face. 

\begin{enumerate}[label = \textbf{(\arabic*)}, ref = (\arabic*), resume]
\item\label{INTERNALVERTEX} Every internal vertex $v$ satisfies $d(v) \geq 4$. Then every $3^{-}$-vertex is incident with the outer face $D$. 
\end{enumerate}

\begin{enumerate}[label = \textbf{(\arabic*)}, ref = (\arabic*), resume]
\item\label{F35P} There is no internal subgraph (not necessarily induced) isomorphic to \autoref{F35}.
\end{enumerate}
\begin{proof}
Suppose to the contrary that there exists an internal subgraph isomorphic to \autoref{F35}. Let $w_{1}w_{2}w_{3}w_{4}w_{5}w_{1}$ be the $5$-cycle and $w_{1}w_{5}w_{6}w_{1}$ be the triangle. If $w_{1}w_{3}$ or $w_{3}w_{5}$ is an edge in $G$, then there exists a subgraph isomorphic to \autoref{fig:subfig:a}, a contradiction. If $w_{2}w_{4}$ is an edge in $G$, then there exists a subgraph isomorphic to \autoref{fig:subfig:b}. If $w_{1}w_{4}$ or $w_{2}w_{5}$ is an edge in $G$, then there exists a subgraph isomorphic to \autoref{fig:subfig:433}. Then the $5$-cycle $w_{1}w_{2}w_{3}w_{4}w_{5}w_{1}$ has no chords. If $w_{2}w_{6}$ or $w_{4}w_{6}$ is an edge in $G$, then there exists an internal induced subgraph isomorphic to \autoref{Kite}, a contradiction. It follows that $w_{2}w_{6}, w_{4}w_{6} \notin E(G)$. Suppose that $w_{3}w_{6}$ is an edge in $G$. Then both $w_{3}w_{2}w_{1}w_{6}w_{3}$ and $w_{3}w_{4}w_{5}w_{6}w_{3}$ are induced $4$-cycles. By \ref{SEPARATING}, the triangle $w_{1}w_{5}w_{6}w_{1}$ bounds a $3$-face. By \ref{INTERNALVERTEX}, the vertex $w_{6}$ has a neighbor not in $\{w_{1}, w_{2}, w_{3}, w_{4}, w_{5}\}$, and one of $w_{3}w_{2}w_{1}w_{6}w_{3}$ and $w_{3}w_{4}w_{5}w_{6}w_{3}$, say $w_{3}w_{4}w_{5}w_{6}w_{3}$, is a separating $4$-cycle. By \ref{SEPARATING}, $w_{3}w_{4}w_{5}w_{6}w_{3}$ must be a bad $4$-cycle. In other words, there exists a vertex $w \neq w_{5}$ such that $w$ is adjacent to each of $w_{3}, w_{4}, w_{5}$ and $w_{6}$. But there exists a subgraph isomorphic to \autoref{fig:subfig:a}, a contradiction.  
\end{proof}

\begin{enumerate}[label = \textbf{(\arabic*)}, ref = (\arabic*), resume]
\item\label{4FACECHORD} If $f = [w_{1}w_{2}w_{3}w_{4}]$ is a $4$-face, then $w_{1}w_{3} \notin E(G)$ and $w_{2}w_{4} \notin E(G)$. 
\end{enumerate}
\begin{proof}
Suppose that $w_{1}$ and $w_{3}$ are adjacent in $G$. By \ref{SEPARATING}, each of $[w_{1}w_{2}w_{3}]$ and $[w_{1}w_{3}w_{4}]$ bounds a $3$-face. Then $w_{2}$ and $w_{4}$ must be $2$-vertices. By \ref{INTERNALVERTEX}, $f$ is the outer face $D$, this contradicts the fact that the boundary of $D$ has no chords. 
\end{proof}

\begin{enumerate}[label = \textbf{(\arabic*)}, ref = (\arabic*), resume]
\item\label{NORMALLYADJACENT} If two $4^{-}$-faces are adjacent, then they are normally adjacent. 
\end{enumerate}
\begin{proof}
It is easy to see that if two $3$-faces are adjacent, then they are normally adjacent. Let $f = [w_{1}w_{2}w_{3}w_{4}]$ be a $4$-face, and let $f' = [w_{1}w_{4}w_{5}]$ be a $3$-face. By \ref{4FACECHORD}, $[w_{1}w_{2}w_{3}w_{4}]$ has no chords. Then $w_{1}, w_{2}, w_{3}, w_{4}$ and $w_{5}$ are five distinct vertices. Hence, $f$ and $f'$ are normally adjacent. 

Let $h = [w_{1}u_{2}u_{3}w_{4}]$ be a $4$-face. Suppose that $f$ and $h$ are not normally adjacent. By symmetry, assume that $u_{2} \in \{w_{2}, w_{3}\}$. By \ref{4FACECHORD}, $[w_{1}w_{2}w_{3}w_{4}]$ has no chords, thus $u_{2} \neq w_{3}$. It follows that $u_{2} = w_{2}$. It is easy to see that $w_{1}$ is a $2$-vertex, and $w_{1}$ is incident with the outer face $D$. Since $f$ and $h$ are symmetric, let $f$ be the outer face $D$. Then $u_{3}$ is an internal $4^{+}$-vertex, and $\int([w_{2}w_{3}w_{4}u_{3}]) \neq \emptyset$. By \ref{SEPARATING}, the cycle $[w_{2}w_{3}w_{4}u_{3}]$ is a separating bad $4$-cycle. It follows that the interior region bounded by $[w_{2}w_{3}w_{4}u_{3}]$ consists of four $3$-faces, and $u_{3}$ is an internal $3$-vertex, a contradiction. 
\end{proof}

\begin{enumerate}[label = \textbf{(\arabic*)}, ref = (\arabic*), resume]
\item\label{K4} There is no complete subgraph $K_{4}$. 
\end{enumerate}
\begin{proof}
Suppose that $\{w_{1}, w_{2}, w_{3}, w_{4}\}$ induces a complete subgraph $K_{4}$. Note that $[w_{1}w_{2}w_{3}]$ and $[w_{1}w_{3}w_{4}]$ are two normally adjacent $3$-faces. Similarly, $[w_{1}w_{2}w_{4}]$ and $[w_{2}w_{3}w_{4}]$ are two $3$-faces. Then $G = K_{4}$ and every vertex has degree three in $G$. By \ref{INTERNALVERTEX}, only the vertices on the outer face can be of degree three. Hence, there are at most three vertices of degree three, a contradiction. 
\end{proof}

\begin{enumerate}[label = \textbf{(\arabic*)}, ref = (\arabic*), resume]
\item\label{5FACECHORD} If $f = [w_{1}w_{2}w_{3}w_{4}w_{5}]$ is a $5$-face, then the boundary of $f$ has no chords. As a consequence, if a $5$-face is adjacent to a $3$-face, then they are normally adjacent. 
\end{enumerate}
\begin{proof}
Suppose that $w_{1}w_{3} \in E(G)$. By \ref{SEPARATING}, the $3$-cycle $[w_{1}w_{2}w_{3}]$ bounds a $3$-face. Thus, $w_{2}$ is a $2$-vertex, and it is incident with the outer face $D$. It follows that $[w_{1}w_{2}w_{3}]$ bounds the outer face $D$. Since $w_{4}$ is an internal $4^{+}$-vertex, we have that $\int([w_{1}w_{3}w_{4}w_{5}]) \neq \emptyset$ and $[w_{1}w_{3}w_{4}w_{5}]$ is a separating bad $4$-cycle. Hence, the interior region bounded by $[w_{1}w_{3}w_{4}w_{5}]$ consists of four $3$-faces, and $w_{4}$ is an internal $3$-vertex, this contradicts \ref{INTERNALVERTEX}. 
\end{proof}

\begin{enumerate}[label = \textbf{(\arabic*)}, ref = (\arabic*), resume]
\item\label{5FACEADJACENT34FACE} If a $5^{*}$-face is adjacent to a $4^{-}$-face, then they are normally adjacent. 
\end{enumerate}
\begin{proof}
Let $f = [w_{1}w_{2}w_{3}w_{4}w_{5}]$ be a $5^{*}$-face and $f' = [w_{1}w_{5}w_{6}\dots]$ be a $4^{-}$-face. By \ref{5FACECHORD}, $[w_{1}w_{2}w_{3}w_{4}w_{5}]$ has no chords. It follows that $w_{6} \notin \{w_{2}, w_{3}\}$. It is easy to see that $w_{6} \neq w_{4}$, for otherwise $w_{5}$ is a $2$-vertex and $f$ is not an internal face. Hence, $w_{1}, w_{2}, w_{3}, w_{4}, w_{5}$ and $w_{6}$ are distinct. If $f'$ is a $3$-face, then $f$ and $f'$ are normally adjacent. If $f'$ is a $4$-face, then the forth vertex on $f'$ is also distinct from $w_{1}, w_{2}, w_{3}, w_{4}, w_{5}$, thus $f$ and $f'$ are normally adjacent. 
\end{proof}

\begin{enumerate}[label = \textbf{(\arabic*)}, ref = (\arabic*), resume]
\item\label{4V33} If a $4$-face is adjacent to two $3$-faces, then it can only be as depicted in \autoref{4ADJACENTTWO3FACES}, and one of the two $3$-faces is the outer face $D$.
\end{enumerate}
\begin{proof}
Let $f = [w_{1}w_{2}w_{3}w_{4}]$ be a $4$-face, and let $f' = [w_{1}w_{4}w_{5}]$ be a $3$-face. By \ref{NORMALLYADJACENT}, $f$ and $f'$ are normally adjacent. 

Assume $w_{1}w_{2}$ is incident with a $3$-face $h = [w_{1}w_{2}w_{6}]$. By \ref{NORMALLYADJACENT}, $f$ and $h$ are normally adjacent. If $w_{5} \neq w_{6}$, then there is a subgraph isomorphic to \autoref{fig:subfig:a}, a contradiction. So we may assume that $w_{5} = w_{6}$. Then $w_{1}$ is a $3$-vertex, and it is incident with the outer face $D$. If $f$ is the outer face $D$, then $w_{5}$ is an internal vertex, and $[w_{2}w_{3}w_{4}w_{5}]$ is a separating bad $4$-cycle. It follows that the interior region bounded by $[w_{2}w_{3}w_{4}w_{5}]$ consists of four $3$-faces, so there is a subgraph isomorphic to the configuration in \autoref{fig:subfig:a}, a contradiction. Then one of $f'$ and $h$ must be the outer face $D$, so it is as depicted in \autoref{4ADJACENTTWO3FACES}. 

Assume $w_{2}w_{3}$ is incident with a $3$-face $[w_{2}w_{3}w_{7}]$. If $w_{5} \neq w_{7}$, then there is a subgraph isomorphic to \autoref{fig:subfig:b}, a contradiction. So we may assume that $w_{5} = w_{7}$. Note that $w_{5}$ is adjacent to each of $w_{1}, w_{2}, w_{3}$ and $w_{4}$. It follows that a region bounded by $[w_{1}w_{2}w_{3}w_{4}]$ consists of four $3$-faces, which implies that each of $w_{1}, w_{2}, w_{3}$ and $w_{4}$ is a $3$-vertex. Therefore, $f$ must be the outer face $D$, and the boundary of $D$ is a bad $4$-cycle in $G$, a contradiction. 
\end{proof}

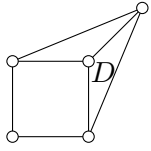
\begin{figure}%
\centering
\begin{tikzpicture}
\def\s{0.707}
\foreach \ang in {1, 2, 3, 4}
{
\def\pointname{v\ang}
\coordinate (\pointname) at ($(\ang*360/4-45:\s)$);
}
\draw (v1)--(v2)--(v3)--(v4)--cycle;
\coordinate (S) at ($(v1)!1!135:(v4)$);
\draw (v2)--(S)--(v4);
\draw (v1)--(S);
\node[circle, inner sep = 1.5, fill = white, draw] () at (S) {};
\node () at (27:1.1*\s) {$D$};
\foreach \ang in {1, 2, 3, 4}
{
\node[circle, inner sep = 1.5, fill = white, draw] () at (v\ang) {};
}
\end{tikzpicture}
\caption{A $4$-face is adjacent to two $3$-faces.}
\label{4ADJACENTTWO3FACES}
\end{figure}

\begin{enumerate}[label = \textbf{(\arabic*)}, ref = (\arabic*), resume]
\item\label{K4MINUS} If $f = [w_{1}w_{2}w_{3}]$ and $f' = [w_{1}w_{3}w_{4}]$ are two $(4, 4, 4)$-faces, then $\{w_{1}, w_{2}, w_{3}, w_{4}\} \cap V(D) \neq \emptyset$. 
\end{enumerate}
\begin{proof}
Let $\Gamma$ be the subgraph induced by $\{w_{1}, w_{2}, w_{3}, w_{4}\}$. Suppose that $\Gamma$ is an internal subgraph. By \ref{K4}, the induced subgraph $\Gamma$ is isomorphic to the configuration in \autoref{Kite}, a contradiction. 
\end{proof}

Let $f = [xy\dots]$ be an internal $(4, 4, 4, 4, 4)$-face. If $f$ is adjacent to a $3$-face $[xyz]$, then we say that $z$ is a \emph{source} of $f$ via the $3$-face $[xyz]$, and $f$ is a \emph{sink} of $z$ via the $3$-face $[xyz]$. 

\begin{enumerate}[label = \textbf{(\arabic*)}, ref = (\arabic*), resume]
\item\label{SOURCESINK} If $f = [w_{1}w_{2}w_{3}w_{4}w_{5}]$ is an internal $(4, 4, 4, 4, 4)$-face, then every source is on $D$ or has degree at least five. 
\end{enumerate}
\begin{proof}
Let $w_{1}w_{5}$ be incident with a $3$-face $f' = [w_{1}w_{5}w_{6}]$. By \ref{5FACEADJACENT34FACE}, $f$ and $f'$ are normally adjacent. If $w_{6}$ is an internal $4$-vertex, then there exists an internal subgraph isomorphic to \autoref{F35}, but this contradicts \ref{F35P}. 
\end{proof}

\begin{enumerate}[label = \textbf{(\arabic*)}, ref = (\arabic*), resume]
\item\label{BADSOURCE} Let $f = [ww_{1}xyw_{2}]$ be a $5^{*}$-face, where $d(w) \geq 5$ and $d(w_{2}) = d(y) = 4$. If $ww_{2}$ is incident with a $3$-face $f' = [ww_{2}w_{3}]$ and $w_{2}y$ is incident with a $3$-face $[uw_{2}y]$, then $w$ is not a source through $f'$. 
\end{enumerate}
\begin{proof}
By \ref{5FACEADJACENT34FACE}, we have that $f$ is normally adjacent to $[ww_{2}w_{3}]$ and $[uw_{2}y]$. Since $w_{2}$ is a $4$-vertex, it has four neighbors $w, y, u, w_{3}$ in a cyclic order. Suppose to the contrary that $w$ is a source via $f'$. Then $uw_{2}w_{3}$ is on the boundary of an internal $(4, 4, 4, 4, 4)$-face $h$. Note that $h$ is normally adjacent to an internal $(4, 4, 4)$-face $[uw_{2}y]$, this contradicts \ref{SOURCESINK}. 
\end{proof}

\begin{enumerate}[label = \textbf{(\arabic*)}, ref = (\arabic*), resume]
\item\label{4V3V4} Let $w$ be a $4^{+}$-vertex. If $w$ is incident with a $4$-face $f = [ww_{1}xw_{2}]$ and a $3$-face $f' = [ww_{2}w_{3}]$, then $ww_{3}$ is incident with a $4^{+}$-face. 
\end{enumerate}
\begin{proof}
By \ref{NORMALLYADJACENT}, $f$ and $f'$ are normally adjacent. Suppose that $ww_{3}$ is incident with a $3$-face $h = [ww_{3}w_{4}]$. It is clear that $f'$ and $h$ are normally adjacent. By \ref{4FACECHORD}, $wx \notin E(G)$, which implies that $w_{4} \neq x$. If $w_{4} = w_{1}$, then $w$ is a $3$-vertex, a contradiction. Hence, all the mentioned vertices are distinct, and there is a subgraph isomorphic to \autoref{fig:subfig:433}, a contradiction.   
\end{proof}

\begin{enumerate}[label = \textbf{(\arabic*)}, ref = (\arabic*), resume]
\item\label{5V3V34} Let $w$ be a $4^{+}$-vertex. If $w$ is incident with a $5$-face $f_{1} = [ww_{1}xyw_{2}]$ and a $3$-face $f_{2} = [ww_{2}w_{3}]$, then $ww_{3}$ is incident with a $5^{+}$-face. 
\end{enumerate}
\begin{proof}
By \ref{5FACECHORD}, the $5$-cycle $[ww_{1}xyw_{2}]$ has no chords. Then $f_{1}$ and $f_{2}$ are normally adjacent. 

Assume $ww_{3}$ is incident with another $3$-face $f_{3} = [ww_{3}w_{4}]$. By \ref{NORMALLYADJACENT}, $f_{2}$ and $f_{3}$ are normally adjacent. If $w_{1} = w_{4}$, then $w$ is a $3$-vertex, a contradiction. Recall that $[ww_{1}xyw_{2}]$ has no chords, we have $w_{4} \notin \{x, y\}$. It follows that $w_{4} \notin \{w_{1}, x, y\}$, and there is a subgraph isomorphic to \autoref{fig:subfig:533}, a contradiction. 

Assume $ww_{3}$ is incident with a $4$-face $f_{3} = [ww_{3}zw_{4}]$. By \ref{NORMALLYADJACENT}, $f_{2}$ and $f_{3}$ are normally adjacent. Similar to the above case, we have $w_{4} \notin \{w_{1}, x, y\}$. If $z = w_{1}$, then $[ww_{1}w_{3}]$ is a separating $3$-cycle, this contradicts \ref{SEPARATING}. If $z = x$, then $[ww_{3}xw_{1}]$ is a separating good $4$-cycle, this contradicts \ref{SEPARATING}. If $z = y$, then $[ww_{2}yw_{4}]$ is a separating good $4$-cycle, this contradicts \ref{SEPARATING}. Hence, all the mentioned vertices are distinct, now there is a subgraph isomorphic to \autoref{fig:subfig:534f}, a contradiction. 
\end{proof}

\begin{enumerate}[label = \textbf{(\arabic*)}, ref = (\arabic*), resume]
\item\label{6FACE} The boundary of every $6$-face is a $6$-cycle. 
\end{enumerate}
\begin{proof}
Suppose that the boundary of a $6$-face is not a $6$-cycle, say $[w_{1}w_{2}w_{3}w_{1}w_{4}w_{5}]$. Note that the $3$-cycle $[w_{1}w_{2}w_{3}]$ bounds a $3$-face, then both $w_{2}$ and $w_{3}$ are all $2$-vertices. Hence, $[w_{1}w_{2}w_{3}]$ bounds the outer face $D$. Therefore, $w_{4}$ and $w_{5}$ are two internal vertices, but $[w_{1}w_{4}w_{5}]$ is a separating $3$-cycle, this contradicts \ref{SEPARATING}.
\end{proof}

Let $v$ be an internal vertex. Let $\diamondsuit(v)$ denote the number of incident $4^{*}$-faces, and let $\pentagon(v)$ denote the number of incident $5^{*}$-faces. 

Let $v$ be an internal $4$-vertex incident with a $5^{*}$-face. The $4$-vertex $v$ is \emph{poor} if $\triangledown(v) = 2$, \emph{bad} if $\triangledown(v) = 1$ and $\diamondsuit(v) = 2$, \emph{light} if $\triangledown(v) = \diamondsuit(v) = 1$ and $\pentagon(v) = 2$, and \emph{special} otherwise. 

Let $f = [w_{1}w_{2}w_{3}w_{4}w_{5}]$ be an internal $(5^{+}, 4, 4, 4, 4)$-face with $d(w_{1}) \geq 5$. If each of $w_{1}w_{2}, w_{2}w_{3}$, $w_{4}w_{5}$ and $w_{5}w_{1}$ is incident with a $3$-face, then we say that $f$ is a \emph{$5_{\mathrm{a}}$-face}. If each of $w_{1}w_{2}, w_{2}w_{3}$ and $w_{4}w_{5}$ is incident with a $3$-face but $w_{5}w_{1}$ is incident with a $4^{*}$-face, then we say that $f$ is a \emph{$5_{\mathrm{b}}$-face}. If $w_{2}, w_{3}$ are poor vertices, and $w_{5}$ is a bad vertex, then we say that $f$ is a \emph{$5_{\mathrm{b}'}$-face}. If each of $w_{1}w_{2}, w_{2}w_{3}$ and $w_{4}w_{5}$ is incident with a $3$-face but the other face incident with $w_{5}w_{1}$ is a $5^{*}$-face, then we say that $f$ is a \emph{$5_{\mathrm{c}}$-face}. For convenience, $5_{\mathrm{b}}$-faces, $5_{\mathrm{b}'}$-faces and $5_{\mathrm{c}}$-faces are all called \emph{bad faces}. 

The definitions of $5_{\mathrm{a}}$-face and $5_{\mathrm{b}'}$-face together with \ref{5V3V34} imply that, if a $5$-vertex is incident with a $5_{\mathrm{a}}$-face, then it is not incident with any $5_{\mathrm{b}}$- or $5_{\mathrm{b}'}$-face.

\begin{enumerate}[label = \textbf{(\arabic*)}, ref = (\arabic*), resume]
\item\label{ADJACENTBADFACE} A $(4, 5^{+})$-edge is incident with at most one bad face. As a consequence, if a $5$-vertex is incident with a $5_{\mathrm{a}}$-face, then it is incident with at most one bad face. 
\end{enumerate}
\begin{proof}
Let $f = [w_{1}w_{2}w_{3}w_{4}w_{5}]$ be a bad face with $d(w_{1}) \geq 5$, and let $w_{1}, w_{4}, x, u$ be the four neighbors of $w_{5}$ in a cyclic order. Suppose that $w_{1}w_{5}$ is incident with two bad faces. According to the definition of bad face, each of $w_{5}w_{4}$ and $w_{5}u$ is incident with a $3$-face, but this contradicts \ref{5V3V34}.
\end{proof}

\begin{enumerate}[label = \textbf{(\arabic*)}, ref = (\arabic*), resume]
\item\label{ADJACENTPOOR} Let $f = [w_{1}w_{2}w_{3}w_{4}w_{5}]$ be a $5^{*}$-face with $d(w_{2}) = d(w_{3}) = 4$. If neither $w_{2}$ nor $w_{3}$ is a light or special vertex, then both $w_{2}$ and $w_{3}$ are poor vertices, \ie each of $w_{1}w_{2}, w_{2}w_{3}$ and $w_{3}w_{4}$ is incident with a $3$-face. As a consequence, if $w_{2}$ is a bad vertex, then $w_{3}$ is a light or special vertex. 
\end{enumerate}
\begin{proof}
Suppose that $w_{2}$ is a poor vertex. Then $f$ is normally adjacent to $3$-faces $[u_{1}w_{1}w_{2}]$ and $[u_{2}w_{2}w_{3}]$. By \ref{5V3V34}, we have that $ww_{3}$ is incident with a $5^{+}$-face. If $w_{3}w_{4}$ is not incident with a $3$-face, then $w_{3}$ is a light or special vertex, a contradiction. Then $w_{3}w_{4}$ is incident with a $3$-face, and then $w_{3}$ is a poor vertex.

Suppose that both $w_{2}$ and $w_{3}$ are bad vertices. By \ref{5V3V34}, $f$ is normally adjacent to a $4^{*}$-face $[aw_{2}w_{3}b]$ and each of $\{w_{2}a, w_{3}b\}$ is incident with a $3$-face, this contradicts \ref{4V33}. 
\end{proof}

Let $G^{*}$ be an auxiliary graph with vertex set $\{f : \text{$f$ is a $3$-face other than $D$}\}$, two vertices $f_{u}$ and $f_{v}$ in $G^{*}$ being adjacent if and only if the two $3$-faces are adjacent in $G$. A \emph{cluster} in $G$ consists of some $3$-faces which corresponds to a connected component of $G^{*}$. By the definition of cluster, every $3$-face other than $D$ belongs to a unique cluster. In the followings, we only need to consider clusters consisting of three or four $3$-faces, see \autoref{Clusters}. Recall that there is no complete subgraph $K_{4}$, it follows that $u_{2}u_{4}, u_{3}u_{5} \notin E(G)$ in \autoref{Clusters}. 

\begin{figure}%
\captionsetup[subfigure]{labelformat=empty}
\centering
\subcaptionbox{$\mathcal{C}_{1}$}[0.2\linewidth]{
\begin{tikzpicture}
\def\s{1}
\coordinate (O) at (0, 0);
\coordinate (v1) at (45:\s);
\coordinate (v2) at (135:\s);
\coordinate (v3) at (225:\s);
\coordinate (v4) at (-45:\s);
\draw (v1)node[right = 1pt]{$u_{4}$}--(v2)node[left = 1pt]{$u_{3}$}--(v3)node[left = 1pt]{$u_{2}$}--(v4)node[right = 1pt]{$u_{5}$}--cycle;
\draw (v1)--(O)node[below = 1pt]{$u_{1}$}--(v3);
\draw (v2)--(v4);
\node[rectangle, inner sep = 1.5, fill, draw] () at (O) {};
\node[circle, inner sep = 1.5, fill = white, draw] () at (v3) {};
\node[circle, inner sep = 1.5, fill = white, draw] () at (v1) {};
\node[circle, inner sep = 1.5, fill = white, draw] () at (v2) {};
\node[circle, inner sep = 1.5, fill = white, draw] () at (v4) {};
\end{tikzpicture}}
\subcaptionbox{$\mathcal{C}_{2}$}[0.2\linewidth]{
\begin{tikzpicture}
\def\s{1}
\coordinate (O) at (0, 0);
\coordinate (v1) at (0:\s);
\coordinate (v2) at (60:\s);
\coordinate (v3) at (120:\s);
\coordinate (v4) at (180:\s);
\draw (v1)node[below = 1pt]{$u_{5}$}--(v2)node[above = 1pt]{$u_{4}$}--(v3)node[above = 1pt]{$u_{3}$}--(v4)node[below = 1pt]{$u_{2}$}--(O)node[below = 1pt]{$u_{1}$}--cycle;
\draw (O)--(v2);
\draw (O)--(v3);
\node[circle, inner sep = 1.5, fill = white, draw] () at (O) {};
\node[circle, inner sep = 1.5, fill = white, draw] () at (v3) {};
\node[circle, inner sep = 1.5, fill = white, draw] () at (v1) {};
\node[circle, inner sep = 1.5, fill = white, draw] () at (v2) {};
\node[circle, inner sep = 1.5, fill = white, draw] () at (v4) {};
\end{tikzpicture}}
\caption{Clusters consisting of four or three $3$-faces.}
\label{Clusters}
\end{figure}
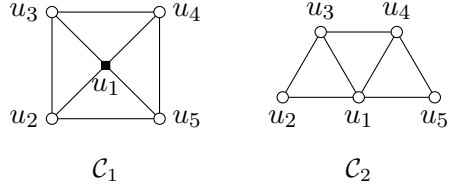

\begin{enumerate}[label = \textbf{(\arabic*)}, ref = (\arabic*), resume]
\item\label{CLUSTER2} For each cluster $\mathcal{C}_{1}$, each of $u_{2}u_{3}, u_{3}u_{4}, u_{4}u_{5}$ and $u_{5}u_{2}$ is incident with a $7^{+}$-face. 
\end{enumerate}
\begin{proof}
By symmetry, it suffices to prove that $u_{3}u_{4}$ is incident with a $7^{+}$-face. Assume $u_{3}u_{4}$ is incident with a $3$-face $[u_{3}u_{4}u]$ with $u \neq u_{1}$. By \ref{K4}, we have $u \notin \{u_{2}, u_{5}\}$. Then there is a subgraph isomorphic to \autoref{fig:subfig:a}, a contradiction. 

Assume $u_{3}u_{4}$ is incident with a $4$-face. By \ref{4V3V4}, both $u_{3}$ and $u_{4}$ are $3$-vertices, then the cycle $[u_{2}u_{3}u_{4}u_{5}]$ bounds the outer face $D$, and the outer cycle is a bad $4$-cycle, a contradiction.  

Assume $u_{3}u_{4}$ is incident with a $5^{+}$-face. If $u_{3}$ is a $3$-vertex, then it is incident with the outer face $D$, and then $u_{1}u_{3}$ is incident with the outer face $D$, but this contradicts the definition of the cluster $\mathcal{C}_{1}$. So we may assume that $u_{3}$, and symmetrically, $u_{4}$ are $4^{+}$-vertices. By \ref{5V3V34}, $u_{3}u_{4}$ cannot be incident with a $5$-face. 

Assume $u_{3}u_{4}$ is incident with a $6$-face $f = [w_{2}u_{3}u_{4}w_{5}w_{6}w_{1}]$. By \ref{6FACE}, the boundary of $f$ is a $6$-cycle. Observe that $u_{1}, u_{2}, \dots, u_{5}, w_{2}, w_{5}$ are distinct vertices. Note that $w_{2}$ and $u_{2}$ are nonadjacent, for otherwise there is a subgraph isomorphic to \autoref{fig:subfig:b}. It follows that $w_{1} \neq u_{2}$. If $w_{1} = u_{5}$, then $[w_{2}u_{3}u_{1}u_{5}]$ is a separating good $4$-cycle, a contradiction. By symmetry, we have that $\{u_{2}, u_{5}\} \cap \{w_{1}, w_{6}\} = \emptyset$. Then there is a subgraph isomorphic to the configuration in \autoref{fig:subfig:g}, a contradiction. 
\end{proof}

\begin{enumerate}[label = \textbf{(\arabic*)}, ref = (\arabic*), resume]
\item\label{CLUSTER1} In each cluster $\mathcal{C}_{2}$, the edge $u_{3}u_{4}$ is incident with a $7^{+}$-face. Moreover, if $u_{2}$ is a $4^{+}$-vertex, then $u_{2}u_{3}$ is also incident with a $7^{+}$-face. 
\end{enumerate}
\begin{proof}
Similar to \ref{CLUSTER2}, it is easy to prove that $u_{3}u_{4}$ is incident with a $7^{+}$-face (we leave it to the reader). Let $u_{2}$ be a $4^{+}$-vertex. Recall that $u_{2}u_{4}, u_{3}u_{5} \notin E(G)$. If $u_{2}u_{3}$ is incident with two $3$-faces, then there is a subgraph isomorphic to \autoref{fig:subfig:433}, a contradiction. Suppose that $u_{2}u_{3}$ is incident with a $4$-face $[w_{1}u_{2}u_{3}w_{4}]$. By \ref{4V3V4}, $u_{3}$ must be a $3$-vertex and $u_{4} = w_{4}$. It follows that $u_{3}$ is incident with the outer face $D$, then $[w_{1}u_{2}u_{3}u_{4}]$ bounds the outer face $D$. Since the outer cycle is a good cycle, we have that $w_{1} \neq u_{5}$. Note that $u_{5}$ cannot be adjacent to each of $w_{1}, u_{2}, u_{1}$ and $u_{4}$. Then $[w_{1}u_{2}u_{1}u_{4}]$ is a separating good $4$-cycle, a contradiction. Hence, $u_{2}u_{3}$ is incident with a $5^{+}$-face. If $u_{3}$ is a $3$-vertex, then it is incident with the outer face $D$, and then $u_{3}u_{1}$ is incident with $D$, but this contradicts the definition of the cluster $\mathcal{C}_{2}$. So we may assume that $u_{3}$ is a $4^{+}$-vertex. By \ref{5V3V34}, $u_{2}u_{3}$ cannot be incident with a $5$-face. 

Assume that $u_{2}u_{3}$ is incident with a $6$-face $f = [w_{1}u_{2}u_{3}w_{4}w_{5}w_{6}]$. By \ref{6FACE}, the boundary of $f$ is a $6$-cycle. Note that $u_{2}$ has degree at least four. If $w_{1} = u_{5}$, then the cycle $[u_{2}u_{1}u_{5}]$ is a separating $3$-cycle in $G$, a contradiction. Then $u_{1}, u_{2}, \dots, u_{5}, w_{1}, w_{4}$ are distinct vertices. Note that $w_{1}$ and $u_{1}$ are nonadjacent, for otherwise there is a subgraph isomorphic to \autoref{fig:subfig:a}. For the same reason, $w_{1}$ and $u_{5}$ are nonadjacent. If $w_{6} = u_{4}$, then $[w_{1}u_{2}u_{1}u_{4}]$ is a separating good $4$-cycle, a contradiction. It follows that $w_{6} \notin \{u_{1}, u_{4}, u_{5}\}$. Since the configuration in \autoref{fig:subfig:a} is forbidden, we have that $w_{4}$ and $u_{4}$ are nonadjacent. If $w_{5} = u_{1}$, then $[w_{4}u_{3}u_{1}]$ is a separating $3$-cycle, a contradiction. If $w_{5} = u_{5}$, then $[w_{4}u_{3}u_{1}u_{5}]$ is a separating good $4$-cycle, a contradiction. It follows that $w_{5} \notin \{u_{1}, u_{4}, u_{5}\}$. Therefore, all the mentioned vertices are distinct, but there is a subgraph isomorphic to \autoref{fig:subfig:h}, a contradiction. 
\end{proof}

\begin{enumerate}[label = \textbf{(\arabic*)}, ref = (\arabic*), resume]
\item\label{NO3444} Let $v$ be an internal $4$-vertex. If $\triangledown(v) = 1$ then $\diamondsuit(v) \leq 2$. 
\end{enumerate}
\begin{proof}
Let $v_{1}, v_{2}, v_{3}$ and $v_{4}$ be the four neighbors of $v$ in a cyclic order. Assume $f_{1} = [vv_{1}w_{1}v_{2}]$, $f_{2} = [vv_{2}w_{2}v_{3}]$, $f_{3} = [vv_{3}w_{3}v_{4}]$ are $4^{*}$-faces, and $f_{4} = [vv_{1}v_{4}]$ is a $3$-face. By \ref{NORMALLYADJACENT}, the two $4^{*}$-faces $f_{1}$ and $f_{2}$ are normally adjacent. By \ref{4FACECHORD}, we have $v_{4} \notin \{w_{1}, w_{2}\}$. Similarly, $f_{2}$ and $f_{3}$ are normally adjacent, $f_{3}$ and $f_{4}$ are normally adjacent. If $w_{1} = w_{3}$, then $[vv_{2}w_{1}v_{4}]$ is a separating good $4$-cycle, this contradicts \ref{SEPARATING}. It follows that $w_{1} \neq w_{3}$. But there is a subgraph isomorphic to \autoref{fig:subfig:3444}, a contradiction. 
\end{proof}

Let $w$ be a $5$-vertex incident with a $5_{\mathrm{a}}$-face $[ww_{1}xyw_{2}]$, where $[ww_{1}w_{5}]$ and $[ww_{2}w_{3}]$ are $3$-faces. By \ref{ADJACENTBADFACE}, at most one of the other two faces is a bad face. If $w$ is incident with two other $5^{*}$-faces, and one of which is a bad face (actually, it must be a $5_{\mathrm{c}}$-face), then we say that the other one is a \emph{$5_{\mathrm{d}}$-face}. 

The initial charge function $\mu$ is defined as $\mu(v) = 2d(v) - 6$ for all $v \in V(G)$, $\mu(f) = d(f) - 6$ for all $f \in F(G)\setminus D$ and $\mu(D) = d(D) + 6$. By the Handshaking Theorem and Euler's formula, we have 
\[
\sum_{v \in V(G)} \mu(v) + \sum_{f \in F(G)\setminus D} \mu(f) + \mu(D) = 0.
\] 
Following some appropriate discharging rules, a new charge function $\mu'$ is produced. In the discharging process, the sum of charges is preserved, but we will show that $\mu'(x) \geq 0$ for all $x \in V(G) \cup F(G)$, and the outer face $D$ has final charge $\mu'(D) > 0$, which leads to a contradiction. 

We use the following discharging rules. 

\begin{enumerate}[label = \textbf{R\arabic*.}, ref = R\arabic*]
\item Let $v$ be an internal $4$-vertex and $f$ be an incident $3$-face. Then 
\begin{align*}
\tau(v \rightarrow f) = 
&\begin{cases}
\frac{1}{2}, & \text{if $\triangledown(v) = 4$;}\\[0.2cm]
\frac{2}{3}, & \text{if $\triangledown(v) = 3$;}\\[0.2cm]
1, & \text{otherwise}.
\end{cases}
\end{align*}

\item\label{5+vertex3face} Let $v$ be an internal $5^{+}$-vertex and $f$ be an incident $3$-face. Then 
\begin{align*}
\tau(v \rightarrow f) = 
&\begin{cases}
\lambda = \frac{19}{14}, & \text{if $f$ is in a cluster $\mathcal{C}_{1}$ or $\mathcal{C}_{2}$, and $v = u_{3}$;}\\[0.2cm]
4 - 2\lambda = \frac{9}{7}, & \text{if $f$ is in a cluster $\mathcal{C}_{2}$ and $v \in \{u_{2}, u_{5}\}$;}\\[0.2cm]
\frac{4}{3}, & \text{if $f$ is in a cluster $\mathcal{C}_{2}$ and $v = u_{1}$;}\\[0.2cm]
1, & \text{otherwise, \ie $f$ is not in cluster $\mathcal{C}_{1}$ or $\mathcal{C}_{2}$.}
\end{cases}
\end{align*}
\item\label{vertex4face} Every internal vertex sends $\frac{1}{2}$ to each incident $4^{*}$-face. 
\item\label{4VERTEXTO5FACE} Let $v$ be an internal $4$-vertex and $f$ be an incident $5^{*}$-face. Then $\tau(v \rightarrow f) = \frac{2 - \triangledown(v) - \frac{1}{2}\diamondsuit(v)}{\pentagon(v)}$. 
\item\label{5+vertex5face} Let $v$ be an internal $5^{+}$-vertex and $f$ be an incident $5^{*}$-face. Then
\begin{align*}
\tau(v \rightarrow f) = 
&\begin{cases}
1, & \text{if $f$ is a $5_{\mathrm{a}}$-face;}\\[0.2cm]
\frac{3}{4}, & \text{if $f$ is a $5_{\mathrm{b}}$-face or a $5_{\mathrm{b}'}$-face;}\\[0.2cm]
\frac{2}{3}, & \text{if $f$ is a $5_{\mathrm{c}}$-face;}\\[0.2cm]
\frac{1}{3}, & \text{if $f$ is a $5_{\mathrm{d}}$-face;}\\[0.2cm]
\frac{1}{2}, & \text{otherwise.}
\end{cases}
\end{align*}
\item Every source (may be on $D$) sends $\rho = \frac{1}{4}$ to each associated sink. 
\item Every $7^{+}$-face sends $\frac{1}{7}$ to each adjacent $3$-face. 
\item Every vertex $v$ in $V(D)$ sends its initial charge $2d(v) - 6$ to the outer face $D$.  
\item\label{OUTERFACED} The outer face $D$ sends $2$ to each face having a common vertex with $D$, except that if $[vv_{1}v_{2}]$ has exactly one common vertex $v$ with $D$ and $v$ is a source through $[vv_{1}v_{2}]$, then $D$ sends $1$ to $[vv_{1}v_{2}]$ and sends $1$ to $v$ through $[vv_{1}v_{2}]$. 
\end{enumerate}

According to \ref{4VERTEXTO5FACE}, it is easy to obtain the following lemma. 
\begin{lemma}\label{SUMMARYRULE}
Let $v$ be an internal $4$-vertex and $f$ is an incident $5^{*}$-face. Then 
\begin{align*}
\tau(v \rightarrow f) =
&\begin{cases}
0, & \text{if $v$ is poor or bad;}\\[0.2cm]
\frac{1}{4}, & \text{if $v$ is light;}\\[0.2cm]
\geq \frac{1}{3}, & \text{if $v$ is special.}
\end{cases}
\end{align*}
\end{lemma}

\begin{proof}
By \ref{5V3V34}, we get $\triangledown(v) \leq 2$. Moreover, if $\triangledown(v) = 2$, then $\diamondsuit(v) = 0$. Then $\tau(v \rightarrow f) = 0$ when $v$ is poor or bad. By the definition of light $4$-vertex, $\tau(v \rightarrow f) = \frac{1}{4}$ when $v$ is light. Assume that $v$ is special. If $\triangledown(v) = \diamondsuit(v) = \pentagon(v) = 1$, then $\tau(v \rightarrow f) = \frac{1}{2}$; if $\triangledown(v) = 1$ and $\diamondsuit(v) = 0$, then $\tau(v \rightarrow f) \geq \frac{1}{3}$; if $\triangledown(v) = 0$, then $\tau(v \rightarrow f) \geq \frac{1}{2}$. 
\end{proof}

Let $v$ be an internal $4$-vertex. If $\triangledown(v) = 4$, then $\mu'(v) = 2 - 4 \times \frac{1}{2} = 0$. If $\triangledown(v) = 3$, then $\diamondsuit(v) = \pentagon(v) = 0$ and $\mu'(v) = 2 - 3 \times \frac{2}{3} = 0$. If $\triangledown(v) = 2$, then $\diamondsuit(v) = 0$ and $\mu'(v) = 2 - 2 \times 1 = 0$. Recall that if $\triangledown(v) = 1$ then $\diamondsuit(v) \leq 2$. If $\triangledown(v) = 1$ and $\diamondsuit(v) = 2$, then $\mu'(v) = 2 - 1 - 2 \times \frac{1}{2} = 0$. If $\triangledown(v) = 1$ and $\diamondsuit(v) \leq 1$, then $\mu'(v) \geq \min\{2 - 1 - \frac{1}{2}, 0\} = 0$. If $\triangledown(v) = 0$, then $\mu'(v) \geq \min\{2 - 4 \times \frac{1}{2}, 0\} = 0$. 

Assume that $v$ is an internal $5$-vertex incident with a $5_{\mathrm{a}}$-face. Let $v$ be an internal $5$-vertex incident with a $5_{\mathrm{a}}$-face $f_{1} = [vv_{1}xyv_{2}]$, let $f_{2} = [vv_{2}v_{3}]$ and $f_{5} = [vv_{5}v_{1}]$ be two $3$-faces. Note that $v$ is incident with exactly one $5_{\mathrm{a}}$-face. By \ref{5V3V34}, the other two incident faces $f_{3}$ and $f_{4}$ are $5^{+}$-faces, so neither $f_{3}$ nor $f_{4}$ is a $5_{\mathrm{b}}$- or $5_{\mathrm{b}'}$-face. Moreover, each of $v_{2}v_{3}$ and $v_{5}v_{1}$ is incident with a $5^{+}$-face. Note that neither $f_{2}$ nor $f_{5}$ is in a cluster $\mathcal{C}_{1}$ or $\mathcal{C}_{2}$. By \ref{BADSOURCE}, $v$ is not a source through $f_{2}$ or $f_{5}$. If $v$ is incident with a $5_{\mathrm{c}}$-face $f_{4}$, then $\mu'(v) \geq 4 - 1 - 2 \times 1 - \frac{2}{3} - \frac{1}{3} = 0$. If $v$ is not incident with a $5_{\mathrm{c}}$-face, then $\mu'(v) \geq 4 - 1 - 2 \times 1 - 2 \times \frac{1}{2} = 0$. 

Assume that $v$ is an internal $5$-vertex which is not incident with a $5_{\mathrm{a}}$-face. For convenience, the incident faces are labeled as $f_{1}, f_{2}, \dots, f_{5}$ in a cyclic order. It follows that $v$ is incident with at most three $3$-faces. 
\begin{itemize}
\item Let $\triangledown(v) = 3$. By \ref{4V3V4} and \ref{5V3V34}, the other two incident faces are $6^{+}$-faces. Suppose that $v$ is incident with three consecutive $3$-faces $f_{1}, f_{2}$ and $f_{3}$. By \ref{5V3V34}, $v$ is not a source through any $3$-face. Then $\mu'(v) = 4 - 3 \times \frac{4}{3} = 0$. Suppose that $v$ is incident with three $3$-faces $f_{1}, f_{2}$ and $f_{4}$. Similarly, $v$ is not a source through $f_{1}$ or $f_{2}$. If $f_{4}$ is in a cluster $\mathcal{C}_{2}$, then $v$ sends $4 - 2\lambda$ to $f_{4}$. If $f_{4}$ is not in a cluster $\mathcal{C}_{2}$, then $f$ sends at most $1 + \rho \leq 4 -2 \lambda$ to/via $f_{4}$. Thus, $\mu'(v) \geq 4 - 2\lambda - (4 - 2\lambda) = 0$. 

\item Let $\triangledown(v) = 2$ and $d(f_{1}) = d(f_{2}) = 3$. By \ref{4V3V4} and \ref{5V3V34}, $f_{3}$ and $f_{5}$ are $6^{+}$-faces. Note that $f_{4}$ is not a bad face. It follows that $v$ sends at most $\frac{1}{2}$ to $f_{4}$. Moreover, $v$ is not a source through $f_{1}$ or $f_{2}$. This implies that $\mu'(v) \geq 4 - 2\lambda - \frac{1}{2} > 0$. 

\item Let $\triangledown(v) = 2$ and $d(f_{1}) = d(f_{3}) = 3$. It follows that $d(f_{2}) \geq 5$,  $d(f_{4}) \geq 4$ and $d(f_{5}) \geq 4$. Suppose that one of $f_{4}$ and $f_{5}$, say $f_{4}$, is a bad face. Thus $f_{5}$ is a $4^{*}$- or $5^{*}$-face. By \ref{ADJACENTBADFACE}, $f_{5}$ is not a bad face. By \ref{4V3V4} and \ref{5V3V34}, $f_{1}$ is not in a cluster $\mathcal{C}_{2}$. By \ref{BADSOURCE}, $v$ is not a source through $f_{3}$. By \ref{5V3V34}, $f_{3}$ is not in a cluster $\mathcal{C}_{2}$. Maybe $v$ is a source through $f_{1}$. Thus, $\mu'(v) \geq 4 - (1 + \rho) - \frac{1}{2} - 1 - \frac{3}{4} - \frac{1}{2} = 0$. Suppose that neither $f_{4}$ nor $f_{5}$ is a bad face. If $f_{2}$ is a $5^{*}$-face, then neither $f_{1}$ nor $f_{3}$ is in a cluster $\mathcal{C}_{2}$, and then $\mu'(v) \geq 4 - 2(1 + \rho) - 3 \times \frac{1}{2} = 0$. If $f_{2}$ is not a $5^{*}$-face, then $\tau(v \rightarrow f_{2}) = 0$ and $\mu'(v) \geq 4 - 2(4 - 2\lambda)  - 2 \times \frac{1}{2} > 0$. 

\item Let $\triangledown(v) \leq 1$. Since $v$ is incident with at most one $3$-face, we have that $v$ is incident with at most two bad faces. Moreover, $v$ is contained in at most one cluster. Then $\mu'(v) \geq 4 - (4 - 2\lambda) - 2 \times \frac{3}{4} - 2 \times \frac{1}{2} > 0$. 
\end{itemize}

Let $v$ be an internal $6^{+}$-vertex. Our goal is to prove that $v$ averagely sends to and through each incident face with a total charge of at most $1$, thus $\mu'(v) \geq \mu(v) - d(v) \times 1 = d(v) - 6 \geq 0$. Let $f_{1}, \dots, f_{k}$ be some consecutive $3$-faces incident with $v$, where $vv_{i-1}$ and $vv_{i}$ are the boundary edges of $f_{i}$, the other face incident with $vv_{0}$ is a $4^{+}$-face $f_{0}$, and the other face incident with $vv_{k}$ is a $4^{+}$-face $f_{k+1}$. We call these consecutive $3$-faces \emph{a fan of order $k$}. Note that $v$ sends at most $1$ to each incident $4^{+}$-face. It suffices to prove that 
\[
\frac{1}{2}\Big(\tau(v \rightarrow f_{0}) + \tau(v \rightarrow f_{k+1})\Big) + \tau(v \rightarrow f_{1}) + \dots + \tau(v \rightarrow f_{k}) \leq k + 1. 
\] 
According to the discharging rules, the possible values that send to/via each incident $3$-face are $\lambda, 4 - 2\lambda, \frac{4}{3}, 1 + \rho$. Note that these values are all greater than $1$. Recall that if $v$ sends $\rho$ through a $3$-face, then it directly sends $1$ to this $3$-face. 
\begin{itemize}
\item Suppose that $v$ sends $\lambda$ to an incident $3$-face in a fan. Then $k = 2$, and $v$ sends $\lambda$ to each $3$-face in the fan. By \ref{4V3V4} and \ref{5V3V34}, each of $f_{0}$ and $f_{k+1}$ is a $6^{+}$-face. It follows that $\frac{1}{2}\big(\tau(v \rightarrow f_{0}) + \tau(v \rightarrow f_{3})\big) + \tau(v \rightarrow f_{1}) + \tau(v \rightarrow f_{2}) = 0 + 0 + \lambda + \lambda < k + 1$. 
\item Suppose that $v$ sends $4 - 2\lambda$ to an incident $3$-face in a fan. Then $k = 1$, $f_{1}$ is in a cluster $\mathcal{C}_{2}$, and $v \in \{u_{2}, u_{5}\}$. By \ref{4V3V4}, \ref{5V3V34} and \ref{CLUSTER1}, $f_{0}$ and $f_{2}$ are $6^{+}$-faces. It follows that $\frac{1}{2}\big(\tau(v \rightarrow f_{0}) + \tau(v \rightarrow f_{2})\big) + \tau(v \rightarrow f_{1}) = 0 + 0 + (4 - 2\lambda) < k + 1$. 
\item Suppose that $v$ sends $\frac{4}{3}$ to an incident $3$-face in a fan. Then $k = 3$, and $v$ sends $\frac{4}{3}$ to each $3$-face in this fan. It is true that $f_{0}$ and $f_{4}$ are $6^{+}$-faces. It follows that $\frac{1}{2}\big(\tau(v \rightarrow f_{0}) + \tau(v \rightarrow f_{4})\big) + \tau(v \rightarrow f_{1}) + \tau(v \rightarrow f_{2}) + \tau(v \rightarrow f_{3}) = 0 + 0 + 3 \times \frac{4}{3} = k + 1$. 
\item Suppose that $v$ sends $1 + \rho$ to and through an incident $3$-face in a fan. Then $v$ is a source of a sink via this $3$-face. By \ref{5V3V34}, we have that $k = 1$, each of $f_{0}$ and $f_{2}$ is a $5^{+}$-face. By \ref{BADSOURCE}, neither $f_{0}$ nor $f_{2}$ is a $5_{\mathrm{a}}$-face. It follows that $\frac{1}{2}\big(\tau(v \rightarrow f_{0}) + \tau(v \rightarrow f_{2})\big) + \tau(v \rightarrow f_{1}) \leq \frac{1}{2}(\frac{3}{4} + \frac{3}{4}) + (1 + \rho) = k + 1$. 
\end{itemize}

Assume that $f$ is a $3$-face which is not in a cluster $\mathcal{C}_{1}$ or $\mathcal{C}_{2}$. If $f$ is an internal face, then it receives $1$ from each incident vertex, and then $\mu'(f) \geq 3 - 6 + 3 \times 1 = 0$. If $f$ has exactly one common vertex with $D$, then it receives $1$ from each incident internal vertex and at least $1$ from the outer face $D$, which implies that $\mu'(f) \geq 3 - 6 + 2 \times 1 + 1 = 0$. If $f$ has exactly two common vertex with $D$, then it receives $1$ from the incident internal vertex and $2$ from the outer face $D$, and then $\mu'(f) \geq 3 - 6 + 1 + 2 = 0$. It is impossible that $f$ has three common vertices with $D$. 

Consider a cluster $\mathcal{C}_{1}$. By the definition of clusters, $u_{1}$ is an internal vertex, and it sends $\frac{1}{2}$ to each incident $3$-face. According to the discharging rules, every internal vertex other than $u_{1}$ sends at least $1$ to each incident $3$-face in $\mathcal{C}_{1}$. If $\mathcal{C}_{1}$ has a common vertex with $D$, then there are at least two $3$-faces in $\mathcal{C}_{1}$ having common vertices with $D$, and then $\mu'(\mathcal{C}_{1}) \geq 4 \times (3 - 6) + 4 \times \frac{1}{2} + 6 \times 1 + 2 \times 2 = 0$. Assume that $V(\mathcal{C}_{1}) \cap V(D) = \emptyset$. By \ref{K4}, the five vertices in $\mathcal{C}_{1}$ induced a wheel $W_{4}$. By \ref{K4MINUS}, there are at least two $5^{+}$-vertices in $\mathcal{C}_{1}$. By \ref{CLUSTER2}, $\mathcal{C}_{1}$ is adjacent to four $7^{+}$-faces. It follows that $\mu'(\mathcal{C}_{1}) \geq 4 \times (3 - 6) + 4 \times \frac{1}{2} + 4 \times 1 + 4\lambda + 4 \times \frac{1}{7} = 0$. 

Consider a cluster $\mathcal{C}_{2}$. Similarly, every internal vertex other than $u_{1}$ sends at least $1$ to each incident $3$-face in $\mathcal{C}_{2}$. Assume that $\mathcal{C}_{2}$ has more than one common vertex with the outer face $D$. It is easy to check that $u_{2}u_{5} \in E(G)$ and the $3$-cycle $[u_{2}u_{1}u_{5}]$ bounds the outer face $D$ (we leave it to the reader). Then $\mu'(\mathcal{C}_{2}) \geq 3 \times (3 - 6) + 4 \times 1 + 3 \times 2 > 0$. So we may assume that $\mathcal{C}_{2}$ has at most one common vertex with $D$. If $u_{1}$ is incident with $D$, then $\mu'(\mathcal{C}_{2}) \geq 3 \times (3 - 6) + 6 \times 1 + 3 \times 2 > 0$. If $u_{3}$ or $u_{4}$ is incident with $D$, then $\mu'(\mathcal{C}_{2}) \geq 3 \times (3 - 6) + 3 \times \frac{2}{3} + 4 \times 1 + 2 \times 2 > 0$. If $u_{2}$ or $u_{5}$ is incident with $D$, then $\mu'(\mathcal{C}_{2}) \geq 3 \times (3 - 6) + 3 \times \frac{2}{3} + 5 \times 1 + 2 = 0$. So we may assume that $\mathcal{C}_{2}$ is an internal cluster. If $u_{1}$ is a $5^{+}$-vertex, then $\mu'(\mathcal{C}_{2}) \geq 3 \times (3 - 6) + 3 \times \frac{4}{3} + 6 \times 1 + 3 \times \frac{1}{7}> 0$. Thus, let $u_{1}$ be a $4$-vertex. If $u_{3}$ or $u_{4}$ is a $5^{+}$-vertex, then $\mu'(\mathcal{C}_{2}) \geq 3 \times (3 - 6) + 3 \times \frac{2}{3} + 4 \times 1 + 2\lambda + 3 \times \frac{1}{7} > 0$. If $u_{3}$ and $u_{4}$ are all $4$-vertices, then $u_{2}$ and $u_{5}$ are all $5^{+}$-vertices by \ref{K4MINUS}, and then $\mu'(\mathcal{C}_{2}) \geq 3 \times (3 - 6) + 3 \times \frac{2}{3} + 2 \times (4 - 2\lambda) + 4 \times 1 + 3 \times \frac{1}{7} = 0$.

If $f$ is a $4^{*}$-face, then it receives $\frac{1}{2}$ from each incident vertex, and $\mu'(f) = 4 - 6 + 4 \times \frac{1}{2} = 0$. If $f$ is a $4$-face having at least one common vertex with $D$, then it receives $2$ from $D$, and $\mu'(f) \geq 4 - 6 + 2 = 0$.  

If $f$ is a $6$-face, then it does not send out any charge, and $\mu'(f) \geq \mu(f) = 0$. If $f$ is a $d$-face with $d \geq 7$, then $\mu'(f) \geq \mu(f) - d(f) \times \frac{1}{7} = \frac{6d}{7} - 6 \geq 0$. 

Assume that $f$ is a $5$-face and it has at least one common vertex with $D$. By \ref{OUTERFACED}, $f$ receives at least $2$ from $D$, then $\mu'(f) \geq 5 - 6 + 2 > 0$. Next, we consider internal $5$-faces. 

Assume that $f = [v_{1}v_{2}v_{3}v_{4}v_{5}]$ is an internal $5_{\mathrm{d}}$-face with $d(v_{1}) = 5$, and $v_{1}$ is incident with a $5_{\mathrm{a}}$-face. Let $v_{1}v_{5}$ is incident with a bad $5$-face $g$, and let $v_{1}, v_{4}, u, w$ be the four neighbors of $v_{5}$ in a cyclic order. By the definition of bad faces, $g$ is a $5_{\mathrm{c}}$-face and $[uv_{5}w]$ is a $3$-face. By \ref{4V3V4} and \ref{5V3V34}, the path $v_{4}v_{5}u$ is on the boundary of a $5^{+}$-face. Note that $v_{5}$ is a special vertex, then $\tau(v_{5} \rightarrow f) \geq \frac{1}{3}$ by \autoref{SUMMARYRULE}. If $v_{4}$ is a $5^{+}$-vertex, then $\tau(v_{5} \rightarrow f) = \frac{1}{3}$ by \ref{5+vertex5face}. If $v_{4}$ is a $4$-vertex, then it must be a special vertex and $\tau(v_{4} \rightarrow f) \geq \frac{1}{3}$ by \autoref{SUMMARYRULE}. Thus, $v_{4}$ always sends at least $\frac{1}{3}$ to $f$, this implies that $\mu'(f) \geq 5 - 6 + \frac{1}{3} + \frac{1}{3} + \frac{1}{3} = 0$. So we may assume that $f$ is an internal $5$-face which is not a $5_{\mathrm{d}}$-face in the followings. 

Assume that $f = [v_{1}v_{2}v_{3}v_{4}v_{5}]$ is an internal $(4, 4, 4, 4, 4)$-face. If $f$ is adjacent to at least four $3$-faces, then $f$ has at least four sources and $\mu'(f) \geq 5 - 6 + 4\rho = 0$. If $f$ is adjacent to two or three $3$-faces, then it is incident with at least two light or special vertices by \ref{ADJACENTPOOR}, and then $\mu'(f) \geq 5 - 6 + 2\rho + 2 \times \frac{1}{4} = 0$ by \autoref{SUMMARYRULE}. Suppose $f$ is adjacent to exactly one $3$-face, say $[v_{1}v_{5}v_{6}]$. Note that $v_{1}$ and $v_{5}$ are light or special vertices. By \ref{ADJACENTPOOR}, one of $v_{2}$ and $v_{3}$ is a special vertex. This implies that $\mu'(f) \geq 5 - 6 + \rho + 2 \times \frac{1}{4} + \frac{1}{3} > 0$ by \autoref{SUMMARYRULE}. So we may assume that $f$ is not adjacent to any $3$-face. By \ref{ADJACENTPOOR}, every edge on the boundary of $f$ is incident with a special vertex. Hence, $f$ is incident with at least three special vertices, and $\mu'(f) \geq 5 - 6 + 3 \times \frac{1}{3} = 0$ by \autoref{SUMMARYRULE}.

Assume that $f = [v_{1}v_{2}v_{3}v_{4}v_{5}]$ is an internal $(5^{+}, 4, 4, 4, 4)$-face with $d(v_{1}) \geq 5$. If $f$ is a $5_{\mathrm{a}}$-face, then $\mu'(f) \geq 5 - 6 + 1 = 0$. If $f$ is a $5_{\mathrm{b}}$-face or a $5_{\mathrm{b}'}$-face, then $f$ is incident with a light or special vertex, and then $\mu'(f) \geq 5 - 6 + \frac{3}{4} + \frac{1}{4} = 0$. If $f$ is a $5_{\mathrm{c}}$-face, then $v_{5}$ is a special vertex, and $\mu'(f) \geq 5 - 6 + \frac{2}{3} + \frac{1}{3} = 0$. So we may assume that $f$ is not a $5_{\mathrm{a}}$-, bad, or $5_{\mathrm{d}}$-face. Thus, $\tau(v_{1} \rightarrow f) = \frac{1}{2}$. If $f$ is incident with at least two light or special vertices, then $\mu'(f) \geq 5 - 6 + \frac{1}{2} + 2 \times \frac{1}{4} = 0$. Note that $v_{2}, v_{3}, v_{4}, v_{5}$ cannot be all poor or bad vertices; for otherwise, by \ref{ADJACENTPOOR}, they are all poor and $f$ is a $5_{\mathrm{a}}$-face. Thus, we may further assume that $f$ is incident with exactly one light or special vertex. By symmetry, exactly one of $v_{4}$ and $v_{5}$ is a light or special vertex. Suppose that $v_{5}$ is a light or special vertex. By \ref{ADJACENTPOOR}, $v_{2}, v_{3}$ and $v_{4}$ are poor vertices, thus each of $v_{1}v_{2}, v_{2}v_{3}, v_{3}v_{4}$ and $v_{4}v_{5}$ is incident with a $3$-face. By \ref{5V3V34} and the fact that $v_{1}v_{5}$ is not incident with a $3$-, $4^{*}$- or $5^{*}$-face, we have $\triangledown(v_{5}) = 1$, $\diamondsuit(v_{5}) = 0$ and $\pentagon(v_{5}) \leq 2$. By \ref{4VERTEXTO5FACE}, $\tau(v_{5} \rightarrow f) \geq \frac{1}{2}$, thus $\mu'(f) \geq 5 - 6 + 2 \times \frac{1}{2} = 0$. The remaining case: $v_{4}$ is a light or special vertex. By \ref{ADJACENTPOOR}, we have that $v_{2}$ and $v_{3}$ are poor vertices. If $v_{5}$ is a poor vertex, then $f$ is adjacent to five $3$-faces and $v_{4}$ is also a poor vertex, a contradiction. If $v_{5}$ is a bad vertex, then $f$ is a $5_{\mathrm{b}'}$-face, a contradiction. 

Assume that $f$ is a $5^{*}$-face incident with at least two $5^{+}$-vertices. Since $f$ is not a $5_{\mathrm{d}}$-face, we have that every incident $5^{+}$-vertex sends $\frac{1}{2}$ to $f$, and $\mu'(f) \geq 5 - 6 + 2 \times \frac{1}{2} = 0$. 

By \ref{OUTERFACED}, $D$ sends $2$ to/via each face having at least one common vertex with $D$. For convenience, the face incident with $x_{i}x_{i+1}$ is counted at $x_{i}$. Thus, there are $d(x_{i}) - 2$ faces other than $D$ at $x_{i}$, and there are $\sum_{v\, \in\, V(D)} \big(d(v) - 2\big)$ faces having at least one common vertex with $D$. Therefore, 
\[
\mu'(D) = \mu(D) + \sum_{v\, \in\, V(D)} \Big(2d(v) - 6\Big) - 2\left(\sum_{v\, \in\, V(D)} \Big(d(v) - 2\Big)\right) = \mu(D) - 2d(D) = 6 - d(D) \geq 2.
\]
This completes the proof of \autoref{PAIRWISE3456}.

\section{Proof of \autoref{PAIRWISE3456-AT}}
\label{sec:3}
The following lemma is easy and it can be obtained from \cite{MR4051856}. 
\begin{lemma}[Lu \etal \cite{MR4051856}]\label{L}
Let $D$ be a digraph with $V(D) = X_{1} \cup X_{2}$ and $X_{1} \cap X_{2} = \emptyset$. If all the arcs between $X_{1}$ and $X_{2}$ are oriented from $X_{1}$ to $X_{2}$, then $\mathrm{diff}(D) = \mathrm{diff}(D[X_{1}]) \times \mathrm{diff}(D[X_{2}])$, where $D[X_{1}]$ and $D[X_{2}]$ are sub-digraph induced by $X_{1}$ and $X_{2}$ respectively. 
\end{lemma}

\autoref{PAIRWISE3456-AT} can be obtained from the following stronger result. 

\begin{theorem}\label{PAIRWISE3456-AT'}
Let $G$ be a plane graph without any configuration in \autoref{FIGPAIRWISE3456}, and let $C = [x_{1}x_{2}\dots x_{l}]$ be a good $4^{-}$-cycle in $G$. Then $G - E(G[V(C)])$ has a $4$-AT-orientation such that all the edges incident with $V(C)$ are directed to $V(C)$. 
\end{theorem}

\begin{proof}
Suppose to the contrary that $G$ together with a good $4^{-}$-cycle $C = [x_{1}x_{2}\dots x_{l}]$ is a counterexample to the statement with $|V(G)| + |E(G)|$ as small as possible. By the minimality, the good cycle $C$ has no chords.

Suppose that there is an internal $3^{-}$-vertex $v$. By the minimality, $G - v - E(C)$ has a desired $4$-AT-orientation. We orient all the edges incident with $v$ by going out from $v$. The resulting orientation is a desired $4$-AT-orientation of $G - E(C)$, a contradiction.

Suppose that there is an internal induced subgraph $\Gamma$ isomorphic to \autoref{Kite} or \autoref{F35}. By the minimality, $G - V(\Gamma) - E(C)$ has a desired $4$-AT-orientation $D$. We orient all the edges incident with a vertex in $V(\Gamma)$ as in \autoref{OR}. Note that all the edges between $V(\Gamma)$ and $V(G) \setminus V(\Gamma)$ are directed from $V(\Gamma)$ to $V(G) \setminus V(\Gamma)$. By \autoref{L}, the resulting orientation is a desired $4$-AT-orientation of $G - E(C)$, a contradiction. 

But this contradicts \autoref{PAIRWISE3456}. 
\end{proof}

\begin{figure}
\centering
\def\s{0.8}
\subcaptionbox{}
{\begin{tikzpicture}%
\coordinate (E) at (\s, 0);
\coordinate (N) at (0, \s);
\coordinate (W) at (-\s, 0);
\coordinate (S) at (0, -\s);

\path [draw=black, postaction={on each segment={mid arrow=red}}]
(N)--(W)--(S)--(E)--cycle
(N)--(S)
(W)--($(W)!0.5!180:(S)$)
(W)--($(W)!0.5!180:(N)$)
(E)--($(E)!0.5!180:(S)$)
(E)--($(E)!0.5!180:(N)$)
(N)--($1.5*(N)$)
(S)--($1.5*(S)$)
;
\node[rectangle, inner sep = 1.5, fill, draw] () at (E) {};
\node[rectangle, inner sep = 1.5, fill, draw] () at (N) {};
\node[rectangle, inner sep = 1.5, fill, draw] () at (W) {};
\node[rectangle, inner sep = 1.5, fill, draw] () at (S) {};
\end{tikzpicture}}\hspace{1.5cm}
\subcaptionbox{}
{\begin{tikzpicture}
\coordinate (v1) at (45: 1.414*\s);
\coordinate (v2) at (135: 1.414*\s);
\coordinate (v3) at (225: 1.414*\s);
\coordinate (v4) at (-45: 1.414*\s);
\coordinate (A) at (2*\s, 0);
\coordinate (B) at (-2*\s, 0);
\path [draw=black, postaction={on each segment={mid arrow=red}}]
(v1)--(v2)--(B)--(v3)--(v4)--(A)--cycle
(v2)--(v3)
(A)--($(A)!0.5!-90:(v1)$)
(A)--($(A)!0.5!90:(v4)$)
(B)--($(B)!0.5!90:(v2)$)
(B)--($(B)!0.5!-90:(v3)$)
(v1)--($(v1)!0.25!-90:(v2)$)
(v1)--($(v1)!0.25!-135:(v2)$)
(v2)--($(v2)!0.25!90:(v1)$)
(v3)--($(v3)!0.25!-90:(v4)$)
(v4)--($(v4)!0.25!90:(v3)$)
(v4)--($(v4)!0.25!135:(v3)$)
;

\node[rectangle, inner sep = 1.5, fill, draw] () at (v1) {};
\node[rectangle, inner sep = 1.5, fill, draw] () at (v2) {};
\node[rectangle, inner sep = 1.5, fill, draw] () at (v3) {};
\node[rectangle, inner sep = 1.5, fill, draw] () at (v4) {};
\node[rectangle, inner sep = 1.5, fill, draw] () at (A) {};
\node[rectangle, inner sep = 1.5, fill, draw] () at (B) {};
\end{tikzpicture}}
\caption{Orientations of some configurations.}
\label{OR}
\end{figure}
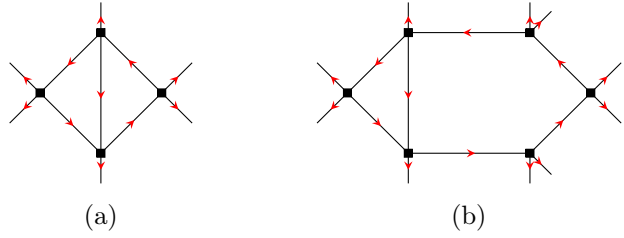

\noindent\textbf{Proof of \autoref{PAIRWISE3456-AT} from \autoref{PAIRWISE3456-AT'}.} Choose an arbitrary vertex $x \in V(G)$, and add a triangle $xyzx$, where $y$ and $z$ are new vertices. Note that the triangle $xyzx$ is a good cycle in $G + xyzx$. By \autoref{PAIRWISE3456-AT'}, $(G + xyzx) - \{xy, yz, xz\}$ has a $4$-AT-orientation $D$. Observe that $(G + xyzx) - \{xy, yz, xz\}$ is the graph obtained from $G$ by adding two isolated vertices $y$ and $z$. Then $D$ is a $4$-AT-orientation of $G$, and $AT(G) \leq 4$.

\section{Proof of \autoref{PAIRWISE3456-Weak3}}
\label{sec:4}

Bernshteyn and Lee \cite{Bernshteyn2021a} prove the following Brooks-type result. 
\begin{lemma}[Theorem 1.6 in \cite{Bernshteyn2021a}]\label{Brooks-Weak}
Let $G$ be a connected graph. The following statement are equivalent:
\begin{enumerate}
\item $G$ is weakly $(\deg - 1)$-degenerate;
\item $G$ is not a GDP-tree, where a GDP-tree is a connected graph in which every block is either a cycle or a complete graph. 
\end{enumerate}
\end{lemma}

We need the following Gallai-type result for critical graphs w.r.t. weakly $g$-degenerate. 

\begin{lemma}[Lemma 5.5 in \cite{Bernshteyn2021a}]\label{Gallai-Weak}
Assume $G$ is a graph which is not weakly $(h-1)$-degenerate. Let $U \subseteq \{u \in V(G) : d(v) = h(v)\}$. If $G - U$ is weakly $(h-1)$-degenerate, then every component of $G[U]$ is a GDP-tree.
\end{lemma}

The following lemma will be frequently used to find a reducible configuration. 

\begin{lemma}\label{Weak-Sequence}
Let $G$ be a graph, and let $A$ be a subset of $V(G)$. Assume $G - A$ is weakly $3$-degenerate. If the vertices in $A$ can be ordered as a sequence $a_{1}, a_{2}, \dots, a_{t}$ satisfying 
\begin{enumerate}
\item $a_{1}a_{t} \in E(G)$; and 
\item $a_{t}$ has degree four in $G$; and 
\item $a_{t}$ has more than $d_{G - A}(a_{1})$ neighbors in $G - A$; and
\item $a_{i}$ has at most three neighbors in $G - \{a_{i+1}, \dots, a_{t}\}$ for each $i \in \{2, 3, \dots, t-1\}$;  
\end{enumerate}
then $G$ is weakly $3$-degenerate. 
\end{lemma}
\begin{proof}
Since $G - A$ is weakly $3$-degenerate, starting from $G - A$ and the constant function of value $3$, we remove all vertices from $G - A$ by a sequence of legal applications of the operations \textsf{Delete} and \textsf{DeleteSave}. This naturally defines a function $g: A \longrightarrow \mathbb{N}$ by $g(a) = 3 - |N_{G}(a) \cap (V(G) - A)|$ for each $a \in A$. Since $a_{1}a_{t} \in E(G)$ and $a_{t}$ has more than $d_{G - A}(a_{1})$ neighbors in $G - A$, we have $g(a_{1}) > g(a_{t})$, and we can remove $a_{1}$ by a legal application of the operation \textsf{DeleteSave}$(G[A], g, a_{1}, a_{t})$. Note that $a_{i}$ has at most three neighbors in $G - \{a_{i+1}, \dots, a_{t}\}$ for each $i \in \{2, 3, \dots, t-1\}$, we can further remove $a_{2}, a_{3}, \dots, a_{t-1}$ in this order by a sequence of legal applications of the operation \textsf{Delete}. Finally, as $a_{t}$ has degree $4$ in $G$, and \textsf{DeleteSave}$(G[A], g, a_{1}, a_{t})$ is applied, so we can remove $a_{t}$ by a legal application of the operation \textsf{Delete}. Hence, $G$ is weakly $3$-degenerate. 
\end{proof}

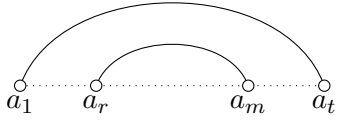
\begin{figure}%
\centering
\begin{tikzpicture}
\def\s{1}
\coordinate (v1) at (-2*\s, 0);
\coordinate (v2) at (-1*\s, 0);
\coordinate (v3) at (\s, 0);
\coordinate (v4) at (2*\s, 0);
\draw (v1)node[below]{$a_{1}$} to [out=70, in=110] (v4)node[below]{$a_{t}$};
\draw (v2)node[below]{$a_{r}$} to [out=70, in=110] (v3)node[below]{$a_{m}$};
\draw (v1)[dotted]--(v4);
\node[circle, inner sep = 1.5, fill = white, draw] () at (v1) {};
\node[circle, inner sep = 1.5, fill = white, draw] () at (v2) {};
\node[circle, inner sep = 1.5, fill = white, draw] () at (v3) {};
\node[circle, inner sep = 1.5, fill = white, draw] () at (v4) {};
\end{tikzpicture}
\caption{Nested pairs in \autoref{WWeak-Sequence}.}
\label{NestPairs}
\end{figure}

The first three conditions in \autoref{Weak-Sequence} make sure that we can ``save a color'' for $a_{t}$ so that it can be legally removed in the last step. We say that $(a_{1}, a_{t})$ is an ordered pair. The fourth condition in \autoref{Weak-Sequence} guarantees that we can greedily remove $a_{2}, a_{3}, \dots, a_{t-1}$ in the order by a sequence of legal applications of the operation \textsf{Delete}. Actually, to legally remove $a_{2}, a_{3}, \dots, a_{t-1}$, we can also use similar conditions like the first three ones. In other words, we can nest a pair $(a_{r}, a_{m})$ on the sequence $a_{2}, a_{3}, \dots, a_{t-1}$. Then we have the following generalization of \autoref{Weak-Sequence}. Note that \autoref{WWeak-Sequence} only demonstrates two nested pairs. 

\begin{lemma}\label{WWeak-Sequence}
Let $G$ be a graph, and let $A$ be a subset of $V(G)$. Assume $G - A$ is weakly $3$-degenerate. If the vertices in $A$ can be ordered as a sequence $a_{1}, a_{2}, \dots, a_{t}$ satisfying 
\begin{enumerate}
\item $a_{1}a_{t}, a_{r}a_{m} \in E(G)$, where $1 < r < m < t$; and 
\item $a_{t}$ has degree four in $G$, and $a_{m}$ has degree four in $G - \{a_{m+1}, \dots, a_{t}\}$; and 
\item $a_{t}$ has more than $d_{G - A}(a_{1})$ neighbors in $G - A$, and $a_{m}$ has more than $d_{G - \{a_{r}, \dots, a_{t}\}}(a_{r})$ neighbors in $G - \{a_{r}, \dots, a_{t}\}$; and
\item $a_{i}$ has at most three neighbors in $G - \{a_{i+1}, \dots, a_{t}\}$ for each $i \in \{2, 3, \dots, t-1\} \setminus \{r, m\}$; 
\end{enumerate}
then $G$ is weakly $3$-degenerate. 
\end{lemma}

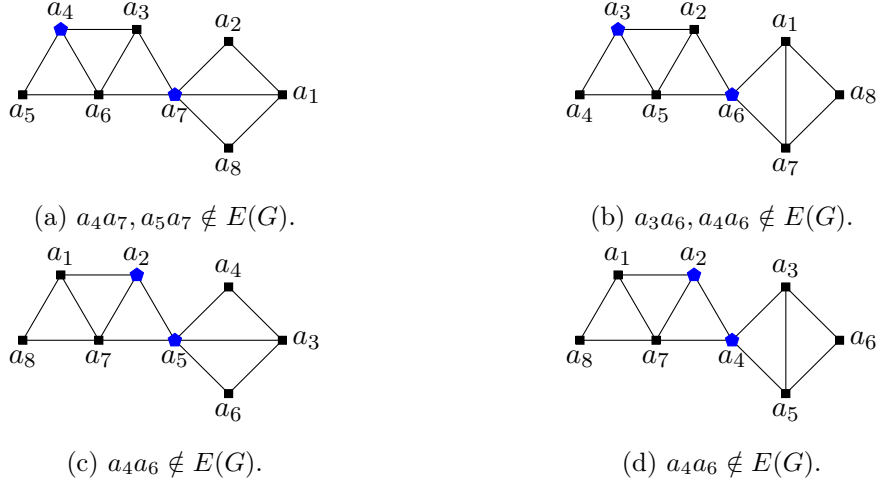
\begin{figure}%
\centering
\subcaptionbox{\label{fig:subfig:RC-3-a}$a_{4}a_{7}, a_{5}a_{7} \notin E(G)$.}[0.45\linewidth]{
\begin{tikzpicture}
\def\s{1}
\coordinate (O) at (0, 0);
\coordinate (v1) at (0:\s);
\coordinate (v2) at (60:\s);
\coordinate (v3) at (120:\s);
\coordinate (v4) at (180:\s);
\coordinate (E1) at ($(45:\s) + (v1)$);
\coordinate (E2) at ($(0:1.414*\s) + (v1)$);
\coordinate (E3) at ($(-45:\s) + (v1)$);
\draw (v1)node[below]{$a_{7}$}--(v2)node[above]{$a_{3}$}--(v3)node[above]{$a_{4}$}--(v4)node[below]{$a_{5}$}--(O)node[below]{$a_{6}$}--cycle;
\draw (O)--(v2);
\draw (O)--(v3);
\draw (E1)node[above]{$a_{2}$}--(E2)node[right]{$a_{1}$}--(E3)node[below]{$a_{8}$}--(v1)--cycle;
\draw (v1)--(E2);
\node[rectangle, inner sep = 1.5, fill, draw] () at (O) {};
\node[regular polygon, inner sep = 1.5, fill=blue, draw=blue] () at (v3) {};
\node[regular polygon, inner sep = 1.5, fill=blue, draw=blue] () at (v1) {};
\node[rectangle, inner sep = 1.5, fill, draw] () at (v2) {};
\node[rectangle, inner sep = 1.5, fill, draw] () at (v4) {};
\node[rectangle, inner sep = 1.5, fill, draw] () at (E1) {};
\node[rectangle, inner sep = 1.5, fill, draw] () at (E2) {};
\node[rectangle, inner sep = 1.5, fill, draw] () at (E3) {};
\end{tikzpicture}}
\subcaptionbox{\label{fig:subfig:RC-3-b}$a_{3}a_{6}, a_{4}a_{6} \notin E(G)$.}[0.45\linewidth]{
\begin{tikzpicture}
\def\s{1}
\coordinate (O) at (0, 0);
\coordinate (v1) at (0:\s);
\coordinate (v2) at (60:\s);
\coordinate (v3) at (120:\s);
\coordinate (v4) at (180:\s);
\coordinate (E1) at ($(45:\s) + (v1)$);
\coordinate (E2) at ($(0:1.414*\s) + (v1)$);
\coordinate (E3) at ($(-45:\s) + (v1)$);
\draw (v1)node[below]{$a_{6}$}--(v2)node[above]{$a_{2}$}--(v3)node[above]{$a_{3}$}--(v4)node[below]{$a_{4}$}--(O)node[below]{$a_{5}$}--cycle;
\draw (O)--(v2);
\draw (O)--(v3);
\draw (E1)node[above]{$a_{1}$}--(E2)node[right]{$a_{8}$}--(E3)node[below]{$a_{7}$}--(v1)--cycle;
\draw (E1)--(E3);
\node[rectangle, inner sep = 1.5, fill, draw] () at (O) {};
\node[regular polygon, inner sep = 1.5, fill=blue, draw=blue] () at (v3) {};
\node[regular polygon, inner sep = 1.5, fill=blue, draw=blue] () at (v1) {};
\node[rectangle, inner sep = 1.5, fill, draw] () at (v2) {};
\node[rectangle, inner sep = 1.5, fill, draw] () at (v4) {};
\node[rectangle, inner sep = 1.5, fill, draw] () at (E1) {};
\node[rectangle, inner sep = 1.5, fill, draw] () at (E2) {};
\node[rectangle, inner sep = 1.5, fill, draw] () at (E3) {};
\end{tikzpicture}}
\subcaptionbox{\label{fig:subfig:RC-3-c}$a_{4}a_{6} \notin E(G)$.}[0.45\linewidth]{
\begin{tikzpicture}
\def\s{1}
\coordinate (O) at (0, 0);
\coordinate (v1) at (0:\s);
\coordinate (v2) at (60:\s);
\coordinate (v3) at (120:\s);
\coordinate (v4) at (180:\s);
\coordinate (E1) at ($(45:\s) + (v1)$);
\coordinate (E2) at ($(0:1.414*\s) + (v1)$);
\coordinate (E3) at ($(-45:\s) + (v1)$);
\draw (v1)node[below]{$a_{5}$}--(v2)node[above]{$a_{2}$}--(v3)node[above]{$a_{1}$}--(v4)node[below]{$a_{8}$}--(O)node[below]{$a_{7}$}--cycle;
\draw (O)--(v2);
\draw (O)--(v3);
\draw (E1)node[above]{$a_{4}$}--(E2)node[right]{$a_{3}$}--(E3)node[below]{$a_{6}$}--(v1)--cycle;
\draw (v1)--(E2);
\node[rectangle, inner sep = 1.5, fill, draw] () at (O) {};
\node[regular polygon, inner sep = 1.5, fill=blue, draw=blue] () at (v2) {};
\node[regular polygon, inner sep = 1.5, fill=blue, draw=blue] () at (v1) {};
\node[rectangle, inner sep = 1.5, fill, draw] () at (v3) {};
\node[rectangle, inner sep = 1.5, fill, draw] () at (v4) {};
\node[rectangle, inner sep = 1.5, fill, draw] () at (E1) {};
\node[rectangle, inner sep = 1.5, fill, draw] () at (E2) {};
\node[rectangle, inner sep = 1.5, fill, draw] () at (E3) {};
\end{tikzpicture}}
\subcaptionbox{\label{fig:subfig:RC-3-d}$a_{4}a_{6} \notin E(G)$.}[0.45\linewidth]{
\begin{tikzpicture}
\def\s{1}
\coordinate (O) at (0, 0);
\coordinate (v1) at (0:\s);
\coordinate (v2) at (60:\s);
\coordinate (v3) at (120:\s);
\coordinate (v4) at (180:\s);
\coordinate (E1) at ($(45:\s) + (v1)$);
\coordinate (E2) at ($(0:1.414*\s) + (v1)$);
\coordinate (E3) at ($(-45:\s) + (v1)$);
\draw (v1)node[below]{$a_{4}$}--(v2)node[above]{$a_{2}$}--(v3)node[above]{$a_{1}$}--(v4)node[below]{$a_{8}$}--(O)node[below]{$a_{7}$}--cycle;
\draw (O)--(v2);
\draw (O)--(v3);
\draw (E1)node[above]{$a_{3}$}--(E2)node[right]{$a_{6}$}--(E3)node[below]{$a_{5}$}--(v1)--cycle;
\draw (E1)--(E3);
\node[rectangle, inner sep = 1.5, fill, draw] () at (O) {};
\node[regular polygon, inner sep = 1.5, fill=blue, draw=blue] () at (v2) {};
\node[regular polygon, inner sep = 1.5, fill=blue, draw=blue] () at (v1) {};
\node[rectangle, inner sep = 1.5, fill, draw] () at (v3) {};
\node[rectangle, inner sep = 1.5, fill, draw] () at (v4) {};
\node[rectangle, inner sep = 1.5, fill, draw] () at (E1) {};
\node[rectangle, inner sep = 1.5, fill, draw] () at (E2) {};
\node[rectangle, inner sep = 1.5, fill, draw] () at (E3) {};
\end{tikzpicture}}
\caption{Note that $|N_{G}(a_{8}) \cap \{a_{1}, \dots, a_{7}\}| = 2$.}
\label{RC-3}
\end{figure}

\begin{theorem}\label{Weak-Reduce}
Let $G$ be a graph without $K_{5}$, and let $\Gamma$ be a subgraph of $G$. Assume $G$ is not weakly $3$-degenerate but $G - V(\Gamma)$ is weakly $3$-degenerate. Then it has the following properties:
\begin{enumerate}
\item $\Gamma$ cannot be a single vertex with degree at most $3$ in $G$; and 
\item $\Gamma$ cannot be isomorphic to the configuration in \autoref{Kite}; and 
\item $\Gamma$ cannot be isomorphic to the configuration in \autoref{F35}; and 
\item $\Gamma$ cannot be isomorphic to a configuration in \autoref{RC} (only consider the solid vertices); and 
\item $\Gamma$ cannot be isomorphic to the configuration in \autoref{fig:subfig:RC1-a}; and 
\item $\Gamma$ cannot be isomorphic to the configuration in \autoref{RC-1}; and 
\item $\Gamma$ cannot be isomorphic to a configuration in \autoref{RC-2}; and 
\item $\Gamma$ cannot be isomorphic to a configuration in \autoref{RC-3}; and 
\end{enumerate}
\end{theorem}

\begin{proof}
Since $G - V(\Gamma)$ is weakly $3$-degenerate, starting from $G - V(\Gamma)$ and the constant function of value $3$, we remove all vertices from $G - V(\Gamma)$ via a sequence of legal applications of the operations \textsf{Delete} and \textsf{DeleteSave}. This naturally defines a function $g: V(\Gamma) \longrightarrow \mathbb{N}$ by $g(u) = 3 - |N_{G}(u) \cap (V(G)\setminus V(\Gamma))|$ for all $u \in V(\Gamma)$. Next, we prove $G[V(\Gamma)]$ is weakly $g$-degenerate, which leads to a contradiction that $G$ is weakly $3$-degenerate. 

Assume $\Gamma$ is a single vertex $w$ with $d(w) \leq 3$. Note that $g(w) \geq 0$, so we can remove $w$ with a legal application of the operation \textsf{Delete}. Then $G$ is weakly $3$-degenerate, a contradiction. 

Assume $\Gamma$ is isomorphic to the configuration in \autoref{Kite}. Note that $g(a_{1}) = g(a_{3}) = 2$ and $g(a_{2}) = g(a_{4}) = 1$. Starting from $\Gamma$ and $g$, we can remove $a_{1}$ by a legal application of the operation \textsf{DeleteSave}$(\Gamma, g, a_{1}, a_{4})$, and then remove the remaining vertices in the order $a_{2}, a_{3}, a_{4}$ by a sequence of legal applications of the operation \textsf{Delete}. Then $G$ is weakly $3$-degenerate, a contradiction. 

Assume $\Gamma$ is isomorphic to the configuration in \autoref{F35}. Since $G$ does not contain $K_{5}$, the subgraph induced by $V(\Gamma)$ is neither a complete graph nor a cycle. By \autoref{Gallai-Weak}, $G[V(\Gamma)]$ is weakly $g$-degenerate. Then $G$ is weakly $3$-degenerate, a contradiction. 

Assume $\Gamma$ is isomorphic to a configuration in \autoref{RC}, \autoref{fig:subfig:RC1-a} or \autoref{RC-1}. Note that the sequence $a_{1}, a_{2}, \dots, a_{t}$ satisfies all the conditions of \autoref{Weak-Sequence}. Then $G$ is weakly $3$-degenerate, a contradiction. 

Assume $\Gamma$ is isomorphic to a configuration in \autoref{RC-2} or \autoref{RC-3}. Note that \autoref{Weak-Sequence} can be used many times with nested pairs for \textsf{DeleteSave}. We take the configuration in \autoref{fig:subfig:RC-2-a} for an example. We remove the vertices in the order $a_{1}, a_{2}, \dots, a_{7}$ by a sequence of legal applications of the operation \textsf{Delete} except \textsf{DeleteSave}$(*, *, a_{1}, a_{7})$ and \textsf{DeleteSave}$(*, *, a_{2}, a_{5})$. Then $G$ is weakly $3$-degenerate, a contradiction. 
\end{proof}

Next, we prove the following stronger result \autoref{PAIRWISE3456-Weak3'}, which is used to prove \autoref{PAIRWISE3456-Weak3}.

\begin{theorem}\label{PAIRWISE3456-Weak3'}
Let $G$ be a connected planar graph without any configuration in \autoref{FIGPAIRWISE3456}, and let $[x_{1}x_{2}\dots x_{l}]$ be a good $4^{-}$-cycle in $G$. Define a function $g: V(G)\setminus \{x_{1}, \dots, x_{l}\} \longrightarrow \mathbb{N}$ by 
\[
g(u) = 3 - |N_{G}(u) \cap \{x_{1}, \dots, x_{l}\}|
\]
for each $u \in V(G)\setminus \{x_{1}, \dots, x_{l}\}$. Then $G - \{x_{1}, \dots, x_{l}\}$ is weakly $g$-degenerate. 
\end{theorem}

\begin{proof}[Proof of \autoref{PAIRWISE3456-Weak3'}]
Let $G$ together with a good $4^{-}$-cycle $[x_{1}x_{2}\dots x_{l}]$ be a counterexample to the statement with $|V(G)| + |E(G)|$ is minimum. Fix a plane embedding for $G$ in the plane.

\begin{enumerate}[label = \textbf{(\arabic*)}, ref = (\arabic*)]
\item Every internal vertex has degree at least $4$. 
\end{enumerate}
\begin{proof}
Suppose that $G$ has an internal vertex $w$ such that $d(w) \leq 3$. By the minimality, $(G - w) - \{x_{1}, \dots, x_{l}\}$ is weakly $g$-degenerate. So we can remove all vertices from $(G - w) - \{x_{1}, \dots, x_{l}\}$ by a sequence of legal applications of the operations \textsf{Delete} and \textsf{DeleteSave}. Since $g(w) = 3 - |N_{G}(w) \cap \{x_{1}, \dots, x_{l}\}|$ and $w$ has degree at most three in $G$, we finally remove $w$ by a legal application of the operation \textsf{Delete}. Then $G - \{x_{1}, \dots, x_{l}\}$ is weakly $g$-degenerate, a contradiction. 
\end{proof}

Note that the induced subgraph isomorphic to \autoref{Kite} is neither a complete graph $K_{4}$ nor a cycle, then by \autoref{Gallai-Weak}, it will not appear in $G$. 
\begin{enumerate}[label = \textbf{(\arabic*)}, ref = (\arabic*), resume]
\item There is no internal induced subgraph isomorphic to \autoref{Kite}.
\end{enumerate}

\begin{enumerate}[label = \textbf{(\arabic*)}, ref = (\arabic*), resume]
\item There is no internal induced subgraph isomorphic to \autoref{F35}.
\end{enumerate}
\begin{proof}
Assume there is an internal induced subgraph isomorphic to \autoref{F35}. Note that the configuration in \autoref{F35} is neither a complete graph $K_{6}$ nor a cycle. Similarly, by \autoref{Gallai-Weak}, the configuration in \autoref{F35} will not appear in $G$. Then there is no internal induced subgraph isomorphic to \autoref{F35}. 
\end{proof}

This contradicts \autoref{PAIRWISE3456}, thus it completes the proof of \autoref{PAIRWISE3456-Weak3'}.
\end{proof}

\begin{proof}[Proof of \autoref{PAIRWISE3456-Weak3}]
It suffices to consider the connected planar graphs. Let $G$ be a connected planar graph without any configuration in \autoref{FIGPAIRWISE3456}. Choose an arbitrary vertex $x$, and add a block which is a triangle $[xyz]$. Note that $G + [xyz]$ is a planar graph without any configuration in \autoref{FIGPAIRWISE3456}. Next, we prove $G + [xyz]$ is weakly $3$-degenerate. We can first remove $x, y, z$ by a sequence of legal applications of the operation \textsf{Delete}. This defines a function $g: V(G)\setminus \{x\} \longrightarrow \mathbb{N}$ by $g(u) = 3 - |N_{G}(u) \cap \{x\}|$ for all $u \in V(G)\setminus \{x\}$. By \autoref{PAIRWISE3456-Weak3'}, we can further remove all vertices from $G - \{x\}$ by a sequence of legal applications of the operations \textsf{Delete} and \textsf{DeleteSave}. Then, starting from $G + [xyz]$ and the constant function of value $3$, we can remove all vertices from $G + [xyz]$ by a sequence of legal applications of the operations \textsf{Delete} and \textsf{DeleteSave}, so $G + [xyz]$ is weakly $3$-degenerate. Observe that $G$ is a subgraph of $G + [xyz]$, then $G$ is also weakly $3$-degenerate. 
\end{proof}

\section{Critical graphs w.r.t. strictly f-degenerate transversal}
\label{sec:5}
In this section, we give some results on the critical graphs with respect to strictly $f$-degenerate transversal. 

Let $G$ be a graph and $(H, f)$ be a valued cover of $G$. The pair $(H, f)$ is \emph{minimal non-strictly $f$-degenerate} if $H$ has no strictly $f$-degenerate transversals, but $(H - L_{v}, f)$ has a strictly $f$-degenerate transversal for any $v \in V(G)$. 

Let $\mathscr{D}$ be the set of all the vertices $v \in V(G)$ such that $f(v, 1) + f(v, 2) + \dots + f(v, s) \geq d_{G}(v)$. 
\begin{theorem}[Lu, Wang and Wang \cite{MR4357325}]\label{CRITICAL}
Let $G$ be a graph and $(H, f)$ be a valued cover of $G$. Let $B$ be a nonempty subset of $\mathscr{D}$ with $G[B]$ having no cut vertex. If $(H, f)$ is a minimal non-strictly $f$-degenerate pair, then
\begin{enumerate}[label = (\roman*)]
\item\label{M1} $G$ is connected and $f(v, 1) + f(v, 2) + \dots + f(v, s) \leq d_{G}(v)$ for each $v \in V(G)$, and 
\item\label{M2} $G[B]$ is a cycle or a complete graph or $d_{G[B]}(v) \leq \max_{q} \{f(v, q)\}$ for each $v \in B$. \qed
\end{enumerate}
\end{theorem}

\begin{theorem}[Nakprasit and Nakprasit \cite{MR4114324}]\label{NN}
Let $k$ be an integer with $k \geq 3$, and let $K$ be an induced subgraph of $G$ and the vertices of $K$ can be ordered as $v_{1}, v_{2}, \dots, v_{m}$ such that the following hold, 
\begin{enumerate}[label = (\roman*)]
\item\label{NN-1} $k - (d_{G}(v_{1}) - d_{K}(v_{1})) \geq 3$; and 
\item\label{NN-2} $d_{G}(v_{m}) \leq k$ and $N_{K}(v_{m}) = \{v_{1}, v_{a}\}$; and  
\item\label{NN-3} for $2 \leq i \leq m - 1$, $v_{i}$ has at most $k - 1$ neighbors in $G - \{v_{i+1}, \dots, v_{m}\}$.
\end{enumerate}
Let $H$ be a cover of $G$ and $f$ be a function from $V(H)$ to $\{0, 1, 2\}$. If $f(v, 1) + \dots + f(v, s) \geq k$ for each vertex $v \in V(G)$, then any strictly $f$-degenerate transversal of $H - \bigcup_{v \in V(K)}L_{v}$ can be extended to that of $H$. \qed
\end{theorem}

We enhance \autoref{NN} to the following result in which $d_{K}(v_{m})$ is relaxed. 
\begin{theorem}\label{WW}
Let $k$ be an integer with $k \geq 3$, and let $K$ be an induced subgraph of $G$ and the vertices of $K$ can be ordered as $v_{1}, v_{2}, \dots, v_{m}$ such that the following hold, 
\begin{enumerate}[label = (\roman*)]
\item\label{WW-1} $k - (d_{G}(v_{1}) - d_{K}(v_{1})) > k - (d_{G}(v_{m}) - d_{K}(v_{m}))$; and 
\item\label{WW-2} $d_{G}(v_{m}) \leq k$ and $v_{1}v_{m} \in E(G)$; and 
\item\label{WW-3} for $2 \leq i \leq m - 1$, $v_{i}$ has at most $k - 1$ neighbors in $G - \{v_{i+1}, \dots, v_{m}\}$.
\end{enumerate}
Let $H$ be a cover of $G$ and $f$ be a function from $V(H)$ to $\{0, 1, 2\}$. If $f(v, 1) + \dots + f(v, s) \geq k$ for each vertex $v \in V(G)$, then any strictly $f$-degenerate transversal of $H - \bigcup_{v \in V(K)}L_{v}$ can be extended to that of $H$. \qed
\end{theorem}

\begin{proof}
Let $V_{K} \coloneqq \bigcup_{v \in V(K)}L_{v}$, and let $S_{1}$ be a strictly $f$-degenerate order of a strictly $f$-degenerate transversal $R_{0}$ of $H - V_{K}$. Define a new function $f^{*}$ on $V_{K}$ by 
\[
f^{*}(v, j) = \max\left\{0, \text{ }f(v, j) - \Big|\big\{(u, i) \in R_{0}: (u, i)(v, j) \in E(H)\big\}\Big|\right\}
\]
for every vertex $(v, j)$. It suffices to consider the case that $f^{*}(v_{m}, 1) + \dots + f^{*}(v_{m}, s) = k - (d_{G}(v_{m}) - d_{K}(v_{m}))$. By the condition \ref{WW-1}, we have that 
\begin{equation}\label{IEQ}
f^{*}(v_{1}, 1) + \dots + f^{*}(v_{1}, s) \geq k - (d_{G}(v_{1}) - d_{K}(v_{1})) > f^{*}(v_{m}, 1) + \dots + f^{*}(v_{m}, s). \tag{$\ast$}
\end{equation}
By renaming the colors, we can assume that $(v_{1}, j)(v_{m}, j) \in E(H)$ for each $j \in [s]$. According to the inequality \eqref{IEQ}, we can further assume that $f^{*}(v_{1}, 1) > f^{*}(v_{m}, 1)$. Thus, we can choose $(v_{1}, 1)$. By the condition \ref{WW-3}, we can use the greedy coloring algorithm to choose $(v_{2}, j_{2}), \dots, (v_{m-1}, j_{m-1})$. Let $R$ denote the sequence $(v_{1}, 1), (v_{2}, j_{2}), \dots, (v_{m-1}, j_{m-1})$. 

If there exists a vertex $(v_{m}, j_{m})$ such that it has less than $f^{*}(v_{m}, j_{m})$ neighbors in $R$, then $S_{1}$ followed by $(v_{1}, 1), (v_{2}, j_{2}), \dots, (v_{m-1}, j_{m-1}), (v_{m}, j_{m})$ is a strictly $f$-degenerate transversal of $H$, a contradiction. So we may assume that each vertex $(v_{m}, j)$ has at least $f^{*}(v_{m}, j)$ neighbors in $R$. Recall that $f^{*}(v_{m}, 1) + \dots + f^{*}(v_{m}, s) = k - (d_{G}(v_{m}) - d_{K}(v_{m})) \geq d_{K}(v_{m})$, then each vertex $(v_{m}, j)$ has exactly $f^{*}(v_{m}, j)$ neighbors in $R$. In particular, $(v_{m}, 1)$ has exactly $f^{*}(v_{m}, 1)$ neighbors in $R$. Note that $(v_{1}, 1)$ and $(v_{m}, 1)$ are adjacent in $H$, thus $f^{*}(v_{m}, 1) \geq 1$. Since $2 \geq f^{*}(v_{1}, 1) > f^{*}(v_{m}, 1) \geq 1$, it follows that $f^{*}(v_{1}, 1) = 2$ and $f^{*}(v_{m}, 1) = 1$. Therefore, $(v_{m}, 1)$ has exactly one neighbor $(v_{1}, 1)$ in $R$, and $(v_{m}, 1)$ followed by $R$ is a strictly $f^{*}$-degenerate transversal of $H[V_{K}]$, a contradiction. 
\end{proof}

\autoref{WW} is an enhancement of \autoref{NN}, and it is also an improvement of a result in \cite[Theorem~2.2]{MR3886261}. The arguments in \autoref{WW} can be nested to obtain the following theorem.

\begin{theorem}\label{EXNN}
Let $k$ be an integer with $k \geq 3$, and let $K$ be an induced subgraph of $G$ and the vertices of $K$ can be ordered as $v_{1}, v_{2}, \dots, v_{m}$ such that the following hold, 
\begin{enumerate}[label = (\roman*)]
\item $d_{G}(v_{1}) - d_{K}(v_{1}) < d_{G}(v_{m}) - d_{K}(v_{m})$ and $d_{G}(v_{r}) - d_{K'}(v_{r}) < d_{G}(v_{t}) - d_{K'}(v_{t}) $, where $1 < r < t < m$ and $K' = K - \{v_{1}, \dots, v_{r-1}\}$; and 
\item $d_{G}(v_{m}) \leq k$ and $v_{1}v_{m} \in E(G)$; $d_{\Gamma}(v_{t}) \leq k$ and $v_{r}v_{t} \in E(G)$, where $\Gamma = G - \{v_{t+1}, \dots, v_{m}\}$; and 
\item for $2 \leq i \leq m - 1$ and $i \notin \{r, t\}$, $v_{i}$ has at most $k - 1$ neighbors in $G - \{v_{i+1}, \dots, v_{m}\}$. 
\end{enumerate}
Let $H$ be a cover of $G$ and $f$ be a function from $V(H)$ to $\{0, 1, 2\}$. If $f(v, 1) + \dots + f(v, s) \geq k$ for each vertex $v \in V(G)$, then any strictly $f$-degenerate transversal of $H - \bigcup_{v \in V(K)}L_{v}$ can be extended to that of $H$. \qed
\end{theorem}

Actually, the arguments in \autoref{WW} can be nested more times, we leave it to the readers. Let $\mathscr{G}$ be a class of graphs which is closed under deleting vertices, \ie every induced subgraph is also in $\mathscr{G}$. 

\autoref{NN} and \autoref{EXNN} together with the sequence $w_{1}, w_{2}, \dots$ imply the following theorem. Note that some adjacency conditions are presented in the captions of the figures.

\begin{theorem}\label{EXNN-RC-1}
Let $G$ be a graph in $\mathscr{G}$, and let $K$ be a configuration in \autoref{RC-1}, \autoref{RC-2} or \autoref{RC-3}. Let $H$ be a cover of $G$ and $f$ be a function from $V(H)$ to $\{0, 1, 2\}$. If $f(v, 1) + \dots + f(v, s) \geq 4$ for each $v \in V(G)$, then any strictly $f$-degenerate transversal of $H - \bigcup_{v \in V(K)}L_{v}$ can be extended to that of $H$. \qed
\end{theorem}

\begin{figure}%
\centering
\subcaptionbox{\label{fig:subfig:RC-4-a}$a_{4}a_{6} \notin E(G)$.}[0.45\linewidth]{\begin{tikzpicture}
\def\s{1}
\coordinate (O) at (0, 0);
\coordinate (I) at (0.5*\s, 0);
\coordinate (E1) at ($(45: \s) + (I)$);
\coordinate (W1) at ($(135: \s) - (I)$);
\coordinate (W3) at ($(225: \s) - (I)$);
\coordinate (E3) at ($(-45: \s) + (I)$);
\coordinate (W2) at ($(-1.414*\s, 0) - (I)$);
\coordinate (E2) at ($(1.414*\s, 0) + (I)$);
\coordinate (N1) at ($(O)-(I)$);
\coordinate (N2) at (I);
\coordinate (N3) at ($(N2) + (0, \s)$);
\coordinate (N4) at ($(N1) + (0, \s)$);
\draw (W1)node[above]{$a_{1}$}--(W2)node[left]{$a_{10}$}--(W3)node[below]{$a_{9}$}--(N1)node[below]{$a_{5}$}--cycle;
\draw (E1)node[above]{$a_{2}$}--(E2)node[right]{$a_{8}$}--(E3)node[below]{$a_{7}$}--(N2)node[below]{$a_{4}$}--cycle;
\draw (E1)--(E3);
\draw (W1)--(W3);
\draw (N1)--(N2)--(N3)node[above]{$a_{3}$}--(N4)node[above]{$a_{6}$}--cycle;
\draw (N1)--(N3);
\node[rectangle, inner sep = 1.5, fill, draw] () at (W1) {};
\node[rectangle, inner sep = 1.5, fill, draw] () at (W2) {};
\node[rectangle, inner sep = 1.5, fill, draw] () at (W3) {};
\node[rectangle, inner sep = 1.5, fill, draw] () at (E1) {};
\node[rectangle, inner sep = 1.5, fill, draw] () at (E2) {};
\node[rectangle, inner sep = 1.5, fill, draw] () at (E3) {};
\node[regular polygon, inner sep = 1.5, fill=blue, draw=blue] () at (N1) {};
\node[regular polygon, inner sep = 1.5, fill=blue, draw=blue] () at (N2) {};
\node[rectangle, inner sep = 1.5, fill, draw] () at (N3) {};
\node[rectangle, inner sep = 1.5, fill, draw] () at (N4) {};
\end{tikzpicture}}
\subcaptionbox{\label{fig:subfig:RC-4-b}$a_{5}a_{7} \notin E(G)$.}[0.45\linewidth]{\begin{tikzpicture}
\def\s{1}
\coordinate (O) at (0, 0);
\coordinate (I) at (0.5*\s, 0);
\coordinate (E1) at ($(45: \s) + (I)$);
\coordinate (W1) at ($(135: \s) - (I)$);
\coordinate (W3) at ($(225: \s) - (I)$);
\coordinate (E3) at ($(-45: \s) + (I)$);
\coordinate (W2) at ($(-1.414*\s, 0) - (I)$);
\coordinate (E2) at ($(1.414*\s, 0) + (I)$);
\coordinate (N1) at ($(O)-(I)$);
\coordinate (N2) at (I);
\coordinate (N3) at ($(N2) + (0, \s)$);
\coordinate (N4) at ($(N1) + (0, \s)$);
\draw (W1)node[above]{$a_{1}$}--(W2)node[left]{$a_{10}$}--(W3)node[below]{$a_{9}$}--(N1)node[below]{$a_{6}$}--cycle;
\draw (E1)node[above]{$a_{3}$}--(E2)node[right]{$a_{2}$}--(E3)node[below]{$a_{8}$}--(N2)node[below]{$a_{5}$}--cycle;
\draw (N2)--(E2);
\draw (W1)--(W3);
\draw (N1)--(N2)--(N3)node[above]{$a_{4}$}--(N4)node[above]{$a_{7}$}--cycle;
\draw (N1)--(N3);
\node[rectangle, inner sep = 1.5, fill, draw] () at (W1) {};
\node[rectangle, inner sep = 1.5, fill, draw] () at (W2) {};
\node[rectangle, inner sep = 1.5, fill, draw] () at (W3) {};
\node[rectangle, inner sep = 1.5, fill, draw] () at (E1) {};
\node[rectangle, inner sep = 1.5, fill, draw] () at (E2) {};
\node[rectangle, inner sep = 1.5, fill, draw] () at (E3) {};
\node[regular polygon, inner sep = 1.5, fill=blue, draw=blue] () at (N1) {};
\node[regular polygon, inner sep = 1.5, fill=blue, draw=blue] () at (N2) {};
\node[rectangle, inner sep = 1.5, fill, draw] () at (N3) {};
\node[rectangle, inner sep = 1.5, fill, draw] () at (N4) {};
\end{tikzpicture}}
\subcaptionbox{\label{fig:subfig:RC-4-c}$a_{4}a_{6} \notin E(G)$.}[0.45\linewidth]{\begin{tikzpicture}
\def\s{0.707}
\coordinate (O) at (0, 0);
\coordinate (N1) at (-2*\s, \s);
\coordinate (N2) at (0, \s);
\coordinate (N3) at (2*\s, \s);
\coordinate (S1) at (-2*\s, -\s);
\coordinate (S2) at (0, -\s);
\coordinate (S3) at (2*\s, -\s);
\coordinate (M1) at (-3*\s, 0);
\coordinate (M2) at (-\s, 0);
\coordinate (M3) at (\s, 0);
\coordinate (M4) at (3*\s, 0);
\draw (M1)node[left]{$a_{10}$}--(N1)node[above]{$a_{1}$}--(M2)node[above]{$a_{6}$}--(N2)node[above]{$a_{3}$}--(M3)node[above]{$a_{4}$}--(N3)node[above]{$a_{2}$}--(M4)node[right]{$a_{8}$}--(S3)node[below]{$a_{7}$}--(M3)--(S2)node[below]{$a_{5}$}--(M2)--(S1)node[below]{$a_{9}$}--cycle;
\draw (N1)--(S1);
\draw (N2)--(S2);
\draw (N3)--(S3);
\node[rectangle, inner sep = 1.5, fill, draw] () at (N1) {};
\node[rectangle, inner sep = 1.5, fill, draw] () at (N2) {};
\node[rectangle, inner sep = 1.5, fill, draw] () at (N3) {};
\node[rectangle, inner sep = 1.5, fill, draw] () at (S1) {};
\node[rectangle, inner sep = 1.5, fill, draw] () at (S2) {};
\node[rectangle, inner sep = 1.5, fill, draw] () at (S3) {};
\node[rectangle, inner sep = 1.5, fill, draw] () at (M1) {};
\node[regular polygon, inner sep = 1.5, fill=blue, draw=blue] () at (M2) {};
\node[regular polygon, inner sep = 1.5, fill=blue, draw=blue] () at (M3) {};
\node[rectangle, inner sep = 1.5, fill, draw] () at (M4) {};
\end{tikzpicture}}
\subcaptionbox{\label{fig:subfig:RC-4-d}$a_{4}a_{6} \notin E(G)$.}[0.45\linewidth]{\begin{tikzpicture}
\def\s{0.707}
\coordinate (O) at (0, 0);
\coordinate (N1) at (-2*\s, \s);
\coordinate (N2) at (0, \s);
\coordinate (N3) at (2*\s, \s);
\coordinate (S1) at (-2*\s, -\s);
\coordinate (S2) at (0, -\s);
\coordinate (S3) at (2*\s, -\s);
\coordinate (M1) at (-3*\s, 0);
\coordinate (M2) at (-\s, 0);
\coordinate (M3) at (\s, 0);
\coordinate (M4) at (3*\s, 0);
\draw (M1)node[left]{$a_{10}$}--(N1)node[above]{$a_{1}$}--(M2)node[above]{$a_{5}$}--(N2)node[above]{$a_{4}$}--(M3)node[above]{$a_{3}$}--(N3)node[above]{$a_{2}$}--(M4)node[right]{$a_{8}$}--(S3)node[below]{$a_{7}$}--(M3)--(S2)node[below]{$a_{6}$}--(M2)--(S1)node[below]{$a_{9}$}--cycle;
\draw (N1)--(S1);
\draw (M2)--(M3);
\draw (N3)--(S3);
\node[rectangle, inner sep = 1.5, fill, draw] () at (N1) {};
\node[rectangle, inner sep = 1.5, fill, draw] () at (N2) {};
\node[rectangle, inner sep = 1.5, fill, draw] () at (N3) {};
\node[rectangle, inner sep = 1.5, fill, draw] () at (S1) {};
\node[rectangle, inner sep = 1.5, fill, draw] () at (S2) {};
\node[rectangle, inner sep = 1.5, fill, draw] () at (S3) {};
\node[rectangle, inner sep = 1.5, fill, draw] () at (M1) {};
\node[regular polygon, inner sep = 1.5, fill=blue, draw=blue] () at (M2) {};
\node[regular polygon, inner sep = 1.5, fill=blue, draw=blue] () at (M3) {};
\node[rectangle, inner sep = 1.5, fill, draw] () at (M4) {};
\end{tikzpicture}}
\caption{Note that $|N_{G}(a_{10}) \cap \{a_{1}, \dots, a_{9}\}| = 2$ and $|N_{G}(a_{8}) \cap \{a_{2}, \dots, a_{7}\}| = 2$.}
\label{RC-4}
\end{figure}
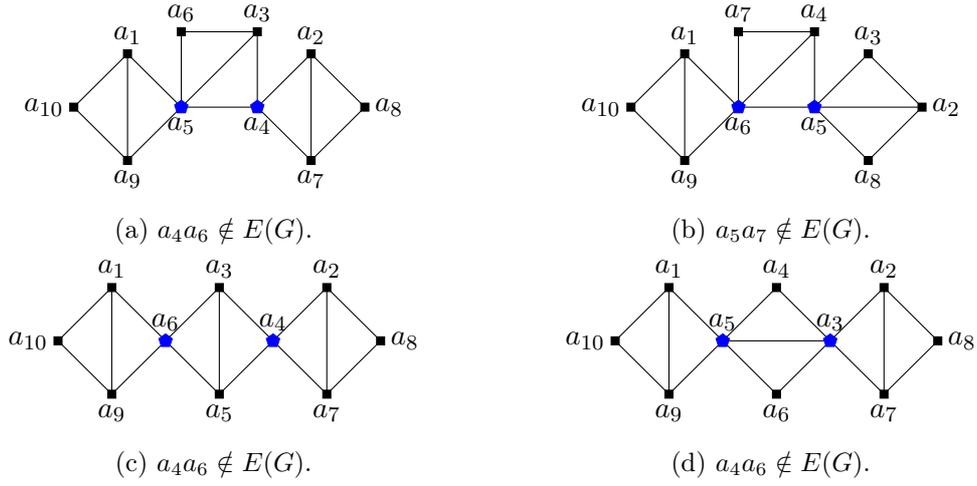

\begin{figure}%
\centering
\subcaptionbox{\label{fig:subfig:RC-5-a}}[0.45\linewidth]
{\begin{tikzpicture}
\def\s{0.707}
\coordinate (O) at (0, 0);
\coordinate (N1) at (-2*\s, \s);
\coordinate (N2) at (0, \s);
\coordinate (N3) at (2*\s, \s);
\coordinate (S1) at (-2*\s, -\s);
\coordinate (S2) at (0, -\s);
\coordinate (S3) at (2*\s, -\s);
\coordinate (M1) at (-3*\s, 0);
\coordinate (M2) at (-\s, 0);
\coordinate (M3) at (\s, 0);
\coordinate (M4) at (3*\s, 0);
\draw (M1)node[left]{$a_{1}$}--(N1)node[above]{$a_{2}$}--(M2)node[above]{$a_{7}$}--(N2)node[above]{$a_{4}$}--(M3)node[above]{$a_{5}$}--(N3)node[above]{$a_{3}$}--(M4)node[right]{$a_{9}$}--(S3)node[below]{$a_{8}$}--(M3)--(S2)node[below]{$a_{6}$}--(M2)--(S1)node[below]{$a_{10}$}--cycle;
\draw (M1)--(M2);
\draw (N2)--(S2);
\draw (N3)--(S3);
\node[rectangle, inner sep = 1.5, fill, draw] () at (N1) {};
\node[rectangle, inner sep = 1.5, fill, draw] () at (N2) {};
\node[rectangle, inner sep = 1.5, fill, draw] () at (N3) {};
\node[rectangle, inner sep = 1.5, fill, draw] () at (S1) {};
\node[rectangle, inner sep = 1.5, fill, draw] () at (S2) {};
\node[rectangle, inner sep = 1.5, fill, draw] () at (S3) {};
\node[rectangle, inner sep = 1.5, fill, draw] () at (M1) {};
\node[regular polygon, inner sep = 1.5, fill=blue, draw=blue] () at (M2) {};
\node[regular polygon, inner sep = 1.5, fill=blue, draw=blue] () at (M3) {};
\node[rectangle, inner sep = 1.5, fill, draw] () at (M4) {};
\end{tikzpicture}}
\subcaptionbox{\label{fig:subfig:RC-5-b}}[0.45\linewidth]
{\begin{tikzpicture}
\def\s{0.707}
\coordinate (O) at (0, 0);
\coordinate (N1) at (-2*\s, \s);
\coordinate (N2) at (0, \s);
\coordinate (N3) at (2*\s, \s);
\coordinate (S1) at (-2*\s, -\s);
\coordinate (S2) at (0, -\s);
\coordinate (S3) at (2*\s, -\s);
\coordinate (M1) at (-3*\s, 0);
\coordinate (M2) at (-\s, 0);
\coordinate (M3) at (\s, 0);
\coordinate (M4) at (3*\s, 0);
\draw (M1)node[left]{$a_{1}$}--(N1)node[above]{$a_{2}$}--(M2)node[above]{$a_{8}$}--(N2)node[above]{$a_{5}$}--(M3)node[above]{$a_{6}$}--(N3)node[above]{$a_{4}$}--(M4)node[right]{$a_{3}$}--(S3)node[below]{$a_{9}$}--(M3)--(S2)node[below]{$a_{7}$}--(M2)--(S1)node[below]{$a_{10}$}--cycle;
\draw (M1)--(M2);
\draw (N2)--(S2);
\draw (M3)--(M4);
\node[rectangle, inner sep = 1.5, fill, draw] () at (N1) {};
\node[rectangle, inner sep = 1.5, fill, draw] () at (N2) {};
\node[rectangle, inner sep = 1.5, fill, draw] () at (N3) {};
\node[rectangle, inner sep = 1.5, fill, draw] () at (S1) {};
\node[rectangle, inner sep = 1.5, fill, draw] () at (S2) {};
\node[rectangle, inner sep = 1.5, fill, draw] () at (S3) {};
\node[rectangle, inner sep = 1.5, fill, draw] () at (M1) {};
\node[regular polygon, inner sep = 1.5, fill=blue, draw=blue] () at (M2) {};
\node[regular polygon, inner sep = 1.5, fill=blue, draw=blue] () at (M3) {};
\node[rectangle, inner sep = 1.5, fill, draw] () at (M4) {};
\end{tikzpicture}}
\caption{Note that $|N_{G}(a_{10}) \cap \{a_{1}, \dots, a_{9}\}| = 2$ and $|N_{G}(a_{9}) \cap \{a_{3}, \dots, a_{8}\}| = 2$.}
\label{RC-5}
\end{figure}
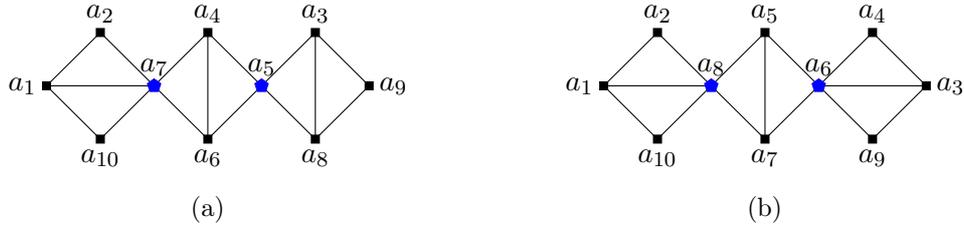

\begin{figure}%
\centering
\subcaptionbox{\label{fig:subfig:RC-6-a}$a_{6}a_{9}, a_{7}a_{9} \notin E(G)$.}[0.45\linewidth]
{\begin{tikzpicture}
\def\s{1}
\coordinate (O) at (0, 0);
\coordinate (v1) at (0:\s);
\coordinate (v2) at (60:\s);
\coordinate (v3) at (120:\s);
\coordinate (v4) at (180:\s);
\coordinate (E1) at ($(45:\s) + (v1)$);
\coordinate (E2) at ($(0:1.414*\s) + (v1)$);
\coordinate (E3) at ($(-45:\s) + (v1)$);
\coordinate (W1) at ($(135:\s) + (v4)$);
\coordinate (W2) at ($(180:1.414*\s) + (v4)$);
\coordinate (W3) at ($(-135:\s) + (v4)$);
\draw (v1)node[below]{$a_{9}$}--(v2)node[above]{$a_{5}$}--(v3)node[above]{$a_{6}$}--(v4)node[below]{$a_{7}$}--(O)node[below]{$a_{8}$}--cycle;
\draw (O)--(v2);
\draw (O)--(v3);
\draw (E1)node[above]{$a_{4}$}--(E2)node[right]{$a_{3}$}--(E3)node[below]{$a_{10}$}--(v1)--cycle;
\draw (v1)--(E2);
\draw (W1)node[above]{$a_{2}$}--(W2)node[left]{$a_{1}$}--(W3)node[below]{$a_{11}$}--(v4)--cycle;
\draw (W2)--(v4);
\node[rectangle, inner sep = 1.5, fill, draw] () at (O) {};
\node[regular polygon, inner sep = 1.5, fill=blue, draw=blue] () at (v3) {};
\node[regular polygon, inner sep = 1.5, fill=blue, draw=blue] () at (v1) {};
\node[rectangle, inner sep = 1.5, fill, draw] () at (v2) {};
\node[regular polygon, inner sep = 1.5, fill=blue, draw=blue] () at (v4) {};
\node[rectangle, inner sep = 1.5, fill, draw] () at (E1) {};
\node[rectangle, inner sep = 1.5, fill, draw] () at (E2) {};
\node[rectangle, inner sep = 1.5, fill, draw] () at (E3) {};
\node[rectangle, inner sep = 1.5, fill, draw] () at (W1) {};
\node[rectangle, inner sep = 1.5, fill, draw] () at (W2) {};
\node[rectangle, inner sep = 1.5, fill, draw] () at (W3) {};
\end{tikzpicture}}
\subcaptionbox{\label{fig:subfig:RC-6-b}$a_{5}a_{8}, a_{6}a_{8} \notin E(G)$.}[0.45\linewidth]
{\begin{tikzpicture}
\def\s{1}
\coordinate (O) at (0, 0);
\coordinate (v1) at (0:\s);
\coordinate (v2) at (60:\s);
\coordinate (v3) at (120:\s);
\coordinate (v4) at (180:\s);
\coordinate (E1) at ($(45:\s) + (v1)$);
\coordinate (E2) at ($(0:1.414*\s) + (v1)$);
\coordinate (E3) at ($(-45:\s) + (v1)$);
\coordinate (W1) at ($(135:\s) + (v4)$);
\coordinate (W2) at ($(180:1.414*\s) + (v4)$);
\coordinate (W3) at ($(-135:\s) + (v4)$);
\draw (v1)node[below]{$a_{8}$}--(v2)node[above]{$a_{4}$}--(v3)node[above]{$a_{5}$}--(v4)node[below]{$a_{6}$}--(O)node[below]{$a_{7}$}--cycle;
\draw (O)--(v2);
\draw (O)--(v3);
\draw (E1)node[above]{$a_{3}$}--(E2)node[right]{$a_{10}$}--(E3)node[below]{$a_{9}$}--(v1)--cycle;
\draw (E1)--(E3);
\draw (W1)node[above]{$a_{2}$}--(W2)node[left]{$a_{1}$}--(W3)node[below]{$a_{11}$}--(v4)--cycle;
\draw (W2)--(v4);
\node[rectangle, inner sep = 1.5, fill, draw] () at (O) {};
\node[regular polygon, inner sep = 1.5, fill=blue, draw=blue] () at (v3) {};
\node[regular polygon, inner sep = 1.5, fill=blue, draw=blue] () at (v1) {};
\node[rectangle, inner sep = 1.5, fill, draw] () at (v2) {};
\node[regular polygon, inner sep = 1.5, fill=blue, draw=blue] () at (v4) {};
\node[rectangle, inner sep = 1.5, fill, draw] () at (E1) {};
\node[rectangle, inner sep = 1.5, fill, draw] () at (E2) {};
\node[rectangle, inner sep = 1.5, fill, draw] () at (E3) {};
\node[rectangle, inner sep = 1.5, fill, draw] () at (W1) {};
\node[rectangle, inner sep = 1.5, fill, draw] () at (W2) {};
\node[rectangle, inner sep = 1.5, fill, draw] () at (W3) {};
\end{tikzpicture}}
\caption{Note that $|N_{G}(a_{11}) \cap \{a_{1}, \dots, a_{10}\}|$ = 2 and $|N_{G}(a_{10}) \cap \{a_{3}, \dots, a_{9}\}| = 2$.}
\label{RC-6}
\end{figure}
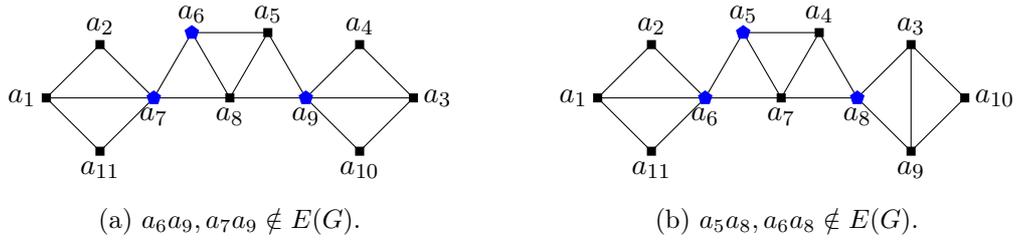

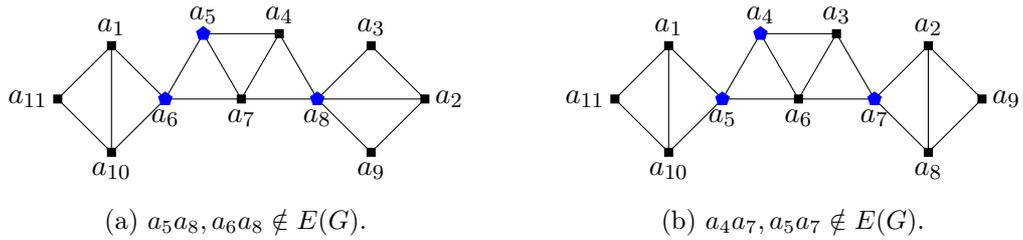
\begin{figure}%
\centering
\subcaptionbox{\label{fig:subfig:RC-7-a}$a_{5}a_{8}, a_{6}a_{8} \notin E(G)$.}[0.45\linewidth]
{\begin{tikzpicture}
\def\s{1}
\coordinate (O) at (0, 0);
\coordinate (v1) at (0:\s);
\coordinate (v2) at (60:\s);
\coordinate (v3) at (120:\s);
\coordinate (v4) at (180:\s);
\coordinate (E1) at ($(45:\s) + (v1)$);
\coordinate (E2) at ($(0:1.414*\s) + (v1)$);
\coordinate (E3) at ($(-45:\s) + (v1)$);
\coordinate (W1) at ($(135:\s) + (v4)$);
\coordinate (W2) at ($(180:1.414*\s) + (v4)$);
\coordinate (W3) at ($(-135:\s) + (v4)$);
\draw (v1)node[below]{$a_{8}$}--(v2)node[above]{$a_{4}$}--(v3)node[above]{$a_{5}$}--(v4)node[below]{$a_{6}$}--(O)node[below]{$a_{7}$}--cycle;
\draw (O)--(v2);
\draw (O)--(v3);
\draw (E1)node[above]{$a_{3}$}--(E2)node[right]{$a_{2}$}--(E3)node[below]{$a_{9}$}--(v1)--cycle;
\draw (v1)--(E2);
\draw (W1)node[above]{$a_{1}$}--(W2)node[left]{$a_{11}$}--(W3)node[below]{$a_{10}$}--(v4)--cycle;
\draw (W1)--(W3);
\node[rectangle, inner sep = 1.5, fill, draw] () at (O) {};
\node[regular polygon, inner sep = 1.5, fill=blue, draw=blue] () at (v3) {};
\node[regular polygon, inner sep = 1.5, fill=blue, draw=blue] () at (v1) {};
\node[rectangle, inner sep = 1.5, fill, draw] () at (v2) {};
\node[regular polygon, inner sep = 1.5, fill=blue, draw=blue] () at (v4) {};
\node[rectangle, inner sep = 1.5, fill, draw] () at (E1) {};
\node[rectangle, inner sep = 1.5, fill, draw] () at (E2) {};
\node[rectangle, inner sep = 1.5, fill, draw] () at (E3) {};
\node[rectangle, inner sep = 1.5, fill, draw] () at (W1) {};
\node[rectangle, inner sep = 1.5, fill, draw] () at (W2) {};
\node[rectangle, inner sep = 1.5, fill, draw] () at (W3) {};
\end{tikzpicture}}
\subcaptionbox{\label{fig:subfig:RC-7-b}$a_{4}a_{7}, a_{5}a_{7} \notin E(G)$.}[0.45\linewidth]
{\begin{tikzpicture}
\def\s{1}
\coordinate (O) at (0, 0);
\coordinate (v1) at (0:\s);
\coordinate (v2) at (60:\s);
\coordinate (v3) at (120:\s);
\coordinate (v4) at (180:\s);
\coordinate (E1) at ($(45:\s) + (v1)$);
\coordinate (E2) at ($(0:1.414*\s) + (v1)$);
\coordinate (E3) at ($(-45:\s) + (v1)$);
\coordinate (W1) at ($(135:\s) + (v4)$);
\coordinate (W2) at ($(180:1.414*\s) + (v4)$);
\coordinate (W3) at ($(-135:\s) + (v4)$);
\draw (v1)node[below]{$a_{7}$}--(v2)node[above]{$a_{3}$}--(v3)node[above]{$a_{4}$}--(v4)node[below]{$a_{5}$}--(O)node[below]{$a_{6}$}--cycle;
\draw (O)--(v2);
\draw (O)--(v3);
\draw (E1)node[above]{$a_{2}$}--(E2)node[right]{$a_{9}$}--(E3)node[below]{$a_{8}$}--(v1)--cycle;
\draw (E1)--(E3);
\draw (W1)node[above]{$a_{1}$}--(W2)node[left]{$a_{11}$}--(W3)node[below]{$a_{10}$}--(v4)--cycle;
\draw (W1)--(W3);
\node[rectangle, inner sep = 1.5, fill, draw] () at (O) {};
\node[regular polygon, inner sep = 1.5, fill=blue, draw=blue] () at (v3) {};
\node[regular polygon, inner sep = 1.5, fill=blue, draw=blue] () at (v1) {};
\node[rectangle, inner sep = 1.5, fill, draw] () at (v2) {};
\node[regular polygon, inner sep = 1.5, fill=blue, draw=blue] () at (v4) {};
\node[rectangle, inner sep = 1.5, fill, draw] () at (E1) {};
\node[rectangle, inner sep = 1.5, fill, draw] () at (E2) {};
\node[rectangle, inner sep = 1.5, fill, draw] () at (E3) {};
\node[rectangle, inner sep = 1.5, fill, draw] () at (W1) {};
\node[rectangle, inner sep = 1.5, fill, draw] () at (W2) {};
\node[rectangle, inner sep = 1.5, fill, draw] () at (W3) {};
\end{tikzpicture}}
\caption{Note that $|N_{G}(a_{11}) \cap \{a_{1}, \dots, a_{10}\}|$ = 2 and $|N_{G}(a_{9}) \cap \{a_{2}, \dots, a_{8}\}| = 2$.}
\label{RC-7}
\end{figure}

By using the arguments in \autoref{WW} three times together with the sequence $a_{1}, a_{2}, \dots$, we can obtain the following structural result. 

\begin{theorem}\label{EXNN-RC-2}
Let $G$ be a graph in $\mathscr{G}$, and let $K$ be a configuration in \autoref{RC-4}, \autoref{RC-5}, \autoref{RC-6} or \autoref{RC-7}. Let $H$ be a cover of $G$ and $f$ be a function from $V(H)$ to $\{0, 1, 2\}$. If $f(v, 1) + \dots + f(v, s) \geq 4$ for each vertex $v \in V(G)$, then any strictly $f$-degenerate transversal of $H - \bigcup_{v \in V(K)}L_{v}$ can be extended to that of $H$. \qed
\end{theorem}

Note that the configurations in \autoref{RC-2} and \autoref{RC-3} may have two blocks, and they are reducible configurations by using twice arguments in \autoref{WW}. The configurations in \autoref{RC-4}--\autoref{RC-7} may have three blocks, and they are reducible configurations by using three times arguments in \autoref{WW}. Although some configurations in \autoref{RC-2}--\autoref{RC-7} will not be used in the following sections, but we believe that they and the ideas will be used in similar results for planar graphs without $7$-cycles or chordal $6$-cycles. 

\section{Proof of \autoref{PAIRWISE3456-SFDT}}
\label{sec:6}
Similar to the above section, we prove the following stronger result \autoref{PAIRWISE3456-SFDT'}, which is used to prove \autoref{PAIRWISE3456-SFDT}
\begin{theorem}\label{PAIRWISE3456-SFDT'}
Let $G$ be a planar graph without any configuration in \autoref{FIGPAIRWISE3456}, and let $[x_{1}x_{2}\dots x_{l}]$ be a good $4^{-}$-cycle in $G$. Let $H$ be a cover of $G$ and $f$ be a function from $V(H)$ to $\{0, 1, 2\}$. If $f(v, 1) + f(v, 2) + \dots + f(v, s) \geq 4$ for each $v \in V(G)$, then each strictly $f$-degenerate transversal $R_{0}$ of $H_{0} = H[\bigcup_{i \in [l]} L_{x_{i}}]$ can be extended to a strictly $f$-degenerate transversal of $H$.  
\end{theorem}

\begin{proof}
Assume $(G, H, f, R_{0})$ is a counterexample to the statement with $|V(G)| + |E(G)|$ is minimum. Note that $G$ is connected. Fix a plane embedding for $G$ in the plane.

\begin{enumerate}[label = \textbf{(\arabic*)}, ref = (\arabic*)]
\item Every internal vertex $v$ satisfies $f(v, 1) + f(v, 2) + \dots + f(v, s) \leq d(v)$. As a consequence, every internal vertex has degree at least $4$. 
\end{enumerate}
\begin{proof}
Suppose that $G$ has an internal vertex $w$ such that $f(w, 1) + f(w, 2) + \dots + f(w, s) > d_{G}(w)$. By the minimality, $R_{0}$ can be extended to a strictly $f$-degenerate transversal $T$ of $H - L_{w}$. Note that 
\[
d_{T}(w, 1) + d_{T}(w, 2) + \dots + d_{T}(w, s) \leq d_{G}(w) < f(w, 1) + f(w, 2) + \dots + f(w, s),
\]
then there exists a vertex $(w, q)$ such that $d_{T}(w, q) < f(w, q)$. Hence, $(w, q)$ followed by an $f$-removing order of $T$ is an $f$-removing order of $T \cup \{(w, q)\}$, a contradiction. 
\end{proof}

\begin{enumerate}[label = \textbf{(\arabic*)}, ref = (\arabic*), resume]
\item There is no internal induced subgraph isomorphic to \autoref{Kite}.
\end{enumerate}
\begin{proof}
Assume there is an internal induced subgraph $\Gamma$ isomorphic to \autoref{Kite}. Note that $\Gamma$ is neither a cycle nor a complete graph $K_{4}$. Moreover, $d_{\Gamma}(a_{1}) = 3 > 2 \geq \max\limits_{i}\{f(x, i)\}$. This contradicts \autoref{CRITICAL}. 
\end{proof}

\begin{enumerate}[label = \textbf{(\arabic*)}, ref = (\arabic*), resume]
\item There is no internal induced subgraph isomorphic to \autoref{F35}.
\end{enumerate}
\begin{proof}
Assume there is an internal induced subgraph $\Gamma$ isomorphic to \autoref{F35}, where $xy$ is the chord of the $6$-cycle. Note that $\Gamma$ is neither a complete graph $K_{6}$ nor a cycle. Moreover, $d_{\Gamma}(x) = 3 > 2 \geq \max\limits_{i}\{f(x, i)\}$. This contradicts \autoref{CRITICAL}. 
\end{proof}

This contradicts \autoref{PAIRWISE3456}, thus it completes the proof of \autoref{PAIRWISE3456-SFDT'}. 
\end{proof}

\begin{proof}[Proof of \autoref{PAIRWISE3456-SFDT}]
Assume that $G$ is triangle-free. By \autoref{MRTHREE}, $H$ has a strictly $f$-degenerate transversal. So we may assume that $G$ contains a $3$-cycle $C$. Note that every triangle is good in $G$. Note that $H_{C}$ has a strictly $f$-degenerate transversal $T$. By \autoref{PAIRWISE3456-SFDT'}, $T$ can be extended to a strictly $f$-degenerate transversal of $H$. 
\end{proof}

\section{Proof of Theorems \ref{T345}, \ref{T345-Weak3} and \ref{T345-SFDT}}
\label{sec:7}

\begin{proof}[Proof of \autoref{T345}]
Suppose that $G$ is a counterexample to \autoref{T345}. Then $G$ is connected and $\delta(G) \geq 4$. Since there is no subgraph isomorphic to the configuration in \autoref{fig:subfig:TA}, every $5$-cycle has no chords.

\begin{lemma}\label{V-label}
Let $[v_{1}v_{2}v_{3}v_{4}v_{5}]$ be a $5$-face adjacent to five $3$-faces $[u_{1}v_{1}v_{2}], [u_{2}v_{2}v_{3}], [u_{3}v_{3}v_{4}], [u_{4}v_{4}v_{5}]$ and $[u_{5}v_{5}v_{1}]$.  \begin{enumerate}[label = (\roman*)]
\item\label{i} The vertices $v_{1}, \dots, v_{5}, u_{1}, \dots, u_{5}$ are distinct, and they induce a subgraph with fifteen edges. 
\item\label{ii} Let $v_{1}$ be a $5^{+}$-vertex and $[v_{1}u_{5}u_{6}u_{7}u_{8}]$ be a $5$-face. Then $v_{1}, u_{8}, u_{7}, u_{6}, u_{5}, v_{5}, v_{4}, v_{3}, v_{2}$ are distinct, and $v_{2}$ has exactly two neighbors $v_{1}$ and $v_{3}$ in this sequence. 
\item\label{iv} Let $v_{1}$ be a $6^{+}$-vertex and $[v_{1}u_{5}u_{6}u_{7}u_{8}]$ be a $5$-face adjacent to a $3$-face $[u_{8}u_{9}v_{1}]$. Then $v_{1}, u_{9}, u_{8}, u_{7}, u_{6}$, $u_{5}, v_{5}, v_{4}, v_{3}, v_{2}$ are distinct, and $v_{2}$ has exactly two neighbors $v_{1}$ and $v_{3}$ in this sequence. 
\end{enumerate}
\end{lemma}

\begin{proof}
\ref{i} Since every $5$-cycle has no chords, we have that $u_{i} \notin \{v_{1}, v_{2}, \dots, v_{5}\}$ for each $i \in \{1, 2, \dots, 5\}$. If $u_{i} = u_{i+1}$, then $v_{i+1}$ is a $3$-vertex, a contradiction. If $u_{i} = u_{i+2}$, then the $5$-cycle $[v_{i}v_{i+1}v_{i+2}v_{i+3}u_{i}]$ has two chords $u_{i}v_{i+1}$ and $u_{i}v_{i+2}$, a contradiction. Therefore, $v_{1}, \dots, v_{5}, u_{1}, \dots, u_{5}$ are ten distinct vertices. Let $A \coloneqq \{v_{1}, \dots, v_{5}, u_{1}, \dots, u_{5}\}$. So there are only these fifteen edges in the induced subgraph $G[A]$, for otherwise there are chordal $5$-cycles. 

\ref{ii} Since $u_{8}$ is a neighbor of $v_{1}$, we have that $u_{8} \notin \{u_{1}, v_{2}, u_{5}, v_{5}\}$. Recall that $G[A]$ has the known fifteen edges. It follows that $u_{8} \notin A$. Similarly, we can obtain that $u_{6} \notin A$. Since $u_{7}$ has two neighbors $u_{6}$ and $u_{8}$, we can check that $u_{7} \notin A \setminus \{u_{3}\}$. Therefore, $v_{1}, \dots, v_{5}, u_{1}, u_{2}, u_{4}, u_{5}, \dots, u_{8}$ are distinct. 

If $u_{8}v_{2} \in E(G)$, then $[v_{1}u_{5}u_{6}u_{7}u_{8}]$ together with $v_{2}$ and $u_{1}$ form a configuration isomorphic to \autoref{fig:subfig:TB}, a contradiction. If $v_{2}$ has neighbors in $\{v_{4}, v_{5}, u_{5}, u_{6}, u_{7}\}$, then there exists a chordal $5$-cycle, a contradiction. Hence, $v_{2}$ has exactly two neighbors $v_{1}$ and $v_{3}$ in the sequence $v_{1}, u_{8}, u_{7}, u_{6}, u_{5}, v_{5}, v_{4}, v_{3}, v_{2}$. 

\ref{iv} It is easy to check that $v_{1}, u_{9}, u_{8}, u_{7}, u_{6}, u_{5}, v_{5}, v_{4}, v_{3}, v_{2}$ are distinct by \ref{i}, \ref{ii} and the fact that every $5$-cycle has no chords. If $u_{9}v_{2} \in E(G)$, then the $5$-cycle $[v_{1}u_{5}u_{6}u_{7}u_{8}]$ together with $u_{9}$ and $v_{2}$ form a configuration isomorphic to \autoref{fig:subfig:TB}, a contradiction. By this and \ref{ii}, $v_{2}$ has exactly two neighbors $v_{1}$ and $v_{3}$ in the sequence $v_{1}, u_{9}, u_{8}, u_{7}, u_{6}, u_{5}, v_{5}, v_{4}, v_{3}, v_{2}$. 
\end{proof}

An \emph{$F_{5}$-face} is a $(5^{+}, 4, 4, 4, 4)$-face adjacent to five $3$-faces. Since the configurations in \autoref{RC} are forbidden in $G$, \autoref{V-label} implies the following lemma. 
\begin{lemma}\label{Lem-RC}
Let $[x_{1}x_{2}x_{3}x_{4}x_{5}]$ be an $F_{5}$-face adjacent to five $3$-faces $[x_{1}x_{2}y_{1}], [x_{2}x_{3}y_{2}], [x_{3}x_{4}y_{3}]$, $[x_{4}x_{5}y_{4}]$ and $[x_{5}x_{1}y_{5}]$.  \begin{enumerate}[label = (\roman*)]
\item\label{Lem-RC1} If $x_{1}$ is a $5$-vertex, then either $y_{1}$ or $y_{5}$ is a $5^{+}$-vertex. 
\item\label{Lem-RC2} If $x_{1}$ is a $5$-vertex, then $x_{1}y_{5}$ is not incident with a $(5, 4, 4, 4, 4)$-face. 
\item\label{Lem-RC3} If $x_{1}$ is a $6$-vertex, then $x_{1}y_{5}$ cannot be incident with an $F_{5}$-face $[x_{1}y_{5}y_{6}y_{7}y_{8}]$ adjacent to a $(6, 4, 4)$-face $[x_{1}y_{8}y_{9}]$. 
\end{enumerate}
\end{lemma}

A 3-face is \emph{good} if it is not adjacent to any other 3-face. If $f = [w_{1}w_{2}w_{3}]$ is a 3-face and $w_{1}w_{2}$ is adjacent to another 3-face, then we say that $f$ is a \emph{bad face}, and $w_{1}w_{3}$ is a \emph{bad edge}, and $w_{1}$ is a \emph{bad vertex} associated with the bad edge $w_{1}w_{3}$. Since every $5$-cycle has no chords, $w_{2}w_{3}$ cannot be incident with another $3$-face. Then $w_{1}$ and $w_{3}$ are not symmetrical, so we cannot associate $w_{3}$ with the bad edge $w_{1}w_{3}$. Therefore, every edge is associated with at most one bad vertex. An \emph{$\epsilon$-face} is a $5$-face adjacent to at most three $3$-faces or a $6^{+}$-face. 

It is easy to obtain the following lemma. 
\begin{lemma}\label{BAD6+}
Every bad edge is incident with a $6^{+}$-face. 
\end{lemma}

A \emph{$4$-regular face} is a face on which all the vertices have degree four in $G$. If a $4$-regular $5$-face $f$ is adjacent to a $3$-face $[uvw]$ with common edge $uv$, then $w$ is said to be a \emph{pendent vertex} of the $4$-regular face $f$. Since there are no subgraphs isomorphic to the configuration in \autoref{F35}, we have the following result. 

\begin{lemma}\label{CAP}
Every pendent vertex of a $4$-regular $5$-face is a $5^{+}$-vertex. 
\end{lemma}

Next, we will use discharging method to complete the proof. For each vertex $v \in V(G)$, we assign an initial charge $\mu(v) = d(v) - 4$; for every face $f \in F(G)$, we assign an initial charge $\mu(f) = d(f) - 4$. Euler's formula yields 
\[
\sum_{v \in V(G)}(d(v) - 4) + \sum_{f \in F(G)}(d(f) - 4) = 0.
\]
On the other hand, we design the following discharging rules to redistribute the charges among the vertices and faces, preserving the sum, such that every element in $V(G) \cup F(G)$ has a nonnegative final charge $\mu'$ but the sum of the final charges is positive, which leads to a contradiction. 

\noindent\textbf{Discharging Rules:}
\begin{enumerate}[label = \textbf{R\arabic*.}, ref = R\arabic*]
\item\label{R1} Every good 3-face receives $\frac{1}{3}$ from each adjacent $5^{+}$-face.
\item\label{R2} Every bad 3-face receives $\frac{1}{2}$ from each adjacent $6^{+}$-face. 
\item\label{R3} Each $4$-regular $5$-face receives $\frac{2}{15}$ from each pendent vertex. 
\item\label{R4} If $f$ is a $5$-face incident with a $5^{+}$-vertex and $\triangledown(f) \geq 4$, then $f$ receives $\frac{\frac{1}{3}\triangledown(f)-1}{n(f)}$ from each incident $5^{+}$-vertex, where $n(f)$ is the number of $5^{+}$-vertices on $f$. 
\item\label{R5} Let $f$ be a $6^{+}$-face and $abc$ be a path on the boundary. Assume each of $ab$ and $bc$ is incident with a $3$-face. 
\begin{enumerate}
\item If $b$ is a bad vertex associated with two bad edges in $\{ab, bc\}$, then $f$ receives $\frac{1}{3}$ from $b$.
\item If $b$ is a bad vertex associated with exactly one bad edge in $\{ab, bc\}$, then $f$ receives $\frac{1}{6}$ from $b$. 
\end{enumerate}
\end{enumerate}
\begin{remark}
Since there are no chordal $5$-cycles, every $3$-face is adjacent to at most one another $3$-face. Then the vertex $b$ in \ref{R5} must be a $5^{+}$-vertex. 
\end{remark}

\begin{claim}
Every face $f$ has nonnegative final charge. 
\end{claim}
Let $f$ be a face in $F(G)$.
\begin{case}
If $f$ is a good $3$-face, then it is adjacent to three $5^{+}$-faces, and then $\mu'(f) = \mu(f) + 3 \times \frac{1}{3} = 0$ by \ref{R1}. If $f$ is a bad $3$-face, then it is adjacent to two $6^{+}$-faces, and then $\mu'(f) = \mu(f) + 2 \times \frac{1}{2} = 0$ by \ref{R2}. 
\end{case}

\begin{case}
If $f$ is a $4$-face, then it is not involved in the discharging procedure, and then $\mu'(f) = \mu(f) = 0$.
\end{case}

\begin{case}
Let $f$ be a $5$-face. Since the configuration in \autoref{fig:subfig:TB} is forbidden, $f$ cannot be adjacent to any bad $3$-face. Assume $f$ is a $4$-regular $5$-face. By \autoref{CAP}, $f$ has $\triangledown(f)$ pendent $5^{+}$-vertices, and then 
\[
\mu'(f) = \mu(f) - \triangledown(f) \times \frac{1}{3} + \triangledown(f) \times \frac{2}{15} = d(f) - 4 - d(f) \times \frac{1}{5} = 5 - 4 - 5 \times \frac{1}{5} = 0.
\]
So we may assume that $f$ is incident with at least one $5^{+}$-vertex. If $f$ is adjacent to at most three $3$-faces, then $\mu'(f) \geq \mu(f) - 3 \times \frac{1}{3} = 0$. If $f$ is adjacent to at least four $3$-faces, then 
\begin{equation*}
\mu'(f) = 5 - 4 - \triangledown(f) \times \frac{1}{3} + n(f) \times \frac{\frac{1}{3} \triangledown(f)- 1}{n(f)} = 0.
\end{equation*}
\end{case}

\begin{case}
Let $f$ be a $6^{+}$-face. Assume that $uvw$ is a path on the boundary of $f$. If $v$ is associated with two bad edges $uv$ and $vw$, then $f$ sends out $\frac{1}{2}$ via each of $uv$ and $vw$, the value of $\frac{1}{2}$ is $\frac{1}{6}$ larger than that of $\frac{1}{3}$, but $f$ gets back $\frac{1}{3} = 2 \times \frac{1}{6}$ from the vertex $v$, thus it is viewed that $f$ sends out $\frac{1}{3}$ in total via each of $uv$ and $vw$. Similarly, if $vw$ is incident with a $3$-face and $v$ is associated with exactly one bad edge $uv$, then $f$ sends out $\frac{1}{2}$ via $uv$ and gets back $\frac{1}{6}$ from $v$, so it is viewed that $f$ sends out $\frac{1}{3}$ via each of $uv$ and $vw$; if $v$ is not associated with any edge in $\{uv, vw\}$, then $f$ sends out at most $\frac{1}{3}$ via each of $uv$ and $vw$. In conclusion, $f$ sends out an average of $\frac{1}{3}$ via each incident edge. This implies that $\mu'(f) \geq \mu(f) - d(f) \times \frac{1}{3} \geq 0$. 
\end{case}

\begin{claim}
Every vertex $v$ has nonnegative final charge. Furthermore, there is a vertex $x$ such that $\mu'(x) > 0$. 
\end{claim}

Let $v$ be a $k$-vertex with neighbors $v_{1}, v_{2}, \dots, v_{k}$ in a cyclic order, and let $f_{i}$ be the face incident with $vv_{i}$ and $vv_{i+1}$, where the subscripts are taken module $k$. 

\begin{lemma}\label{NPendent}
Let $v$ be a $5^{+}$-vertex incident with an $F_{5}$-face $f_{2}$. Then $v$ is not a pendent vertex of a $4$-regular $5$-face via $f_{1}$ or $f_{3}$. 
\end{lemma}
\begin{proof}
Let $f_{2} = [vv_{2}xyv_{3}]$ be an $F_{5}$-face. In other words, $d(v_{2}) = d(x) = d(y) = d(v_{3}) = 4$, and each edge on $f_{2}$ is incident with a $3$-face. Assume $v_{1}v_{2}$ is incident with a $4$-regular $5$-face. Then $x$ is also a pendent vertex of this $4$-regular $5$-face, but this contradicts \autoref{CAP}. 
\end{proof}

We claim that if a $5^{+}$-vertex $v$ is incident with a $5^{+}$-face $f$ which is not an $F_{5}$-face, then it sends at most $\frac{1}{3}$ to $f$. Assume $f$ is a $5$-face. By definition, $f$ is incident with at least two $5^{+}$-vertices, or $f$ is adjacent at most four $3$-faces. By \ref{R4}, $v$ sends at most $\frac{1}{3}$ to $f$. By \ref{R5}, $v$ also sends at most $\frac{1}{3}$ when $f$ is a $6^{+}$-face. This proves the claim.  

\begin{case}
Note that each $4$-vertex is not involved in the discharging procedure, thus its final charge is zero.
\end{case}

\begin{case}
Let $v$ be a $6$-vertex. Thus, it is incident with at most four $3$-faces. 
\begin{itemize}
\item Let $\triangledown(v) = 4$. Then $v$ is incident with two $6^{+}$-faces and four bad $3$-faces. Note that $v$ is not a pendent vertex of any $5$-face. Then only \ref{R5} can be applied to $v$, and then $\mu'(v) \geq 6 - 4 - 2 \times \frac{1}{3} > 0$. 

\item Let $\triangledown(v) = 1$. Then $v$ is a pendent vertex of at most one $5$-face, and \ref{R4} can be applied on at most two incident faces. It follows that $\mu'(v) \geq 6 - 4 - \frac{2}{15} - 2 \times \frac{1}{3} > 0$. 

\item Let $\triangledown(v) = 0$. Then $\mu'(v) = \mu(v) = 6 - 4 > 0$. 

\item Let $\triangledown(v) = 2$. If $f_{1}$ and $f_{2}$ are $3$-faces, then no rules can be applied on incident faces, and $\mu'(v) = \mu(v) = 6 - 4 > 0$. If $f_{1}$ and $f_{3}$ are $3$-faces, then $v$ sends at most $\frac{2}{3}$ to $f_{2}$ by \ref{R4}, at most $\frac{2}{15}$ via each of $f_{1}$ and $f_{3}$ by \ref{R3}, at most $\frac{1}{3}$ to each of $f_{4}$ and $f_{6}$ by \ref{R4}, thus $\mu'(v) \geq 6- 4 - \frac{2}{3}  - 2 \times \frac{2}{15} - 2 \times \frac{1}{3} > 0$. If $f_{1}$ and $f_{4}$ are $3$-faces, then $v$ sends at most $\frac{2}{15}$ via each of $f_{1}$ and $f_{4}$ by \ref{R3}, at most $\frac{1}{3}$ to each of the other incident faces by \ref{R4}, and then $\mu'(v) \geq 6 - 4 - 4 \times \frac{1}{3} - 2 \times \frac{2}{15} > 0$. 

\item Let $\triangledown(v) = 3$. 
\begin{itemize}
\item Let $f_{1}, f_{2}$ and $f_{4}$ be $3$-faces. By \autoref{BAD6+}, $f_{3}$ and $f_{6}$ are $6^{+}$-faces. Then $v$ sends $\frac{1}{6}$ to $f_{3}$ by \ref{R5}, at most $\frac{2}{15}$ via $f_{4}$ by \ref{R3}, at most $\frac{1}{3}$ to $f_{5}$ by \ref{R4}, and then $\mu'(v) \geq 6 - 4 - \frac{1}{6} - \frac{2}{15} - \frac{1}{3} > 0$. 
\item Let $f_{1}, f_{3}$ and $f_{5}$ be $3$-faces. Assume that $v$ is a pendent vertex of a $4$-regular $5$-face incident with $v_{3}v_{4}$. By \autoref{NPendent}, neither $f_{2}$ nor $f_{4}$ is an $F_{5}$-face. Hence, $v$ sends at most $\frac{1}{3}$ to each of $f_{2}$ and $f_{4}$, which implies that $\mu'(v) \geq 6 - 4 - 3 \times \frac{2}{15} - 2 \times \frac{1}{3} - \frac{2}{3} > 0$. Assume that $v$ is not a  pendent vertex of any $4$-regular $5$-face. By \autoref{Lem-RC}\ref{Lem-RC3}, $v$ is incident with at most two $F_{5}$-face. It follows that $\mu'(v) > 6 - 4 - 3 \times \frac{2}{3} = 0$. 
\end{itemize}
\end{itemize}
\end{case}

\begin{case}
Let $v$ be a $7^{+}$-vertex. By the discharging rules, $v$ sends at most $\frac{1}{3}$ to/via each incident face $f_{i}$, except that $f_{i}$ is an $F_{5}$-face. Assume $f_{i}$ is an $F_{5}$-face. By \autoref{BAD6+}, $f_{i-1}$ and $f_{i+1}$ are good $3$-faces. By \autoref{NPendent}, $v$ is not a pendent vertex of a $4$-regular $5$-face via $f_{i-1}$ or $f_{i+1}$. It follows that $v$ sends $\frac{2}{3}$ to $f_{i}$, and sends zero to $f_{i-1}$ and $f_{i+1}$. Then $v$ averagely sends at most $\frac{1}{3}$ to/via each incident face, and $\mu'(v) \geq d(v) - 4 - d(v) \times \frac{1}{3} > 0$. 
\end{case}

\begin{case}\label{5S}
Let $v$ be a $5$-vertex. It follows that $\triangledown(v) \leq 3$. If $\triangledown(v) = 3$ and $f_{1}, f_{2}, f_{4}$ are $3$-faces, then $f_{3}$ and $f_{5}$ are $6^{+}$-faces, and then $v$ sends $\frac{1}{6}$ to each of $f_{3}$ and $f_{5}$, sends at most $\frac{2}{15}$ via $f_{4}$, this implies that $\mu'(v) \geq 5 - 4 - \frac{2}{15} - 2 \times \frac{1}{6} > 0$. If $\triangledown(v) = 1$, then $\mu'(v) \geq 5 - 4 - \frac{2}{15} - 2 \times \frac{1}{3} > 0$. If $\triangledown(v) = 0$, then $\mu'(v) = 5 - 4 > 0$. If $\triangledown(v) = 2$ and $f_{1}, f_{2}$ are $3$-faces, then $f_{3}$ and $f_{5}$ are $6^{+}$-faces, and then $\mu'(v) = 5 - 4 > 0$. 

Next, assume that $v$ is incident with two $3$-faces which are not adjacent, say $f_{1}$ and $f_{3}$. Then each of $f_{2}, f_{4}$ and $f_{5}$ is $5^{+}$-face. 

\begin{figure}%
\centering
\begin{tikzpicture}[scale=0.9]
\foreach \x in {0,...,4}
{
\def\pointnameA{A\x}
\coordinate (\pointnameA) at ($(\x*360/5 + 90:1)$);
}
\draw (A0)node[above]{$v_{2}$}--(A1)node[left]{$v_{3}$};
\draw (A0)--(A4)node[right]{$v_{1}$};
\draw (A2)node[left]{$v_{4}$}--(A3)node[right]{$v_{5}$};
\foreach \x in {0,...,4}
{
\draw (A\x)--(0, 0);
\node[circle, inner sep = 1.5, fill = white, draw] () at (A\x) {};
}
\node[regular polygon, inner sep = 1.5, fill=blue, draw=blue] () at (0, 0) {};
\end{tikzpicture}
\begin{tikzpicture}[scale=0.9]
\foreach \x in {0,...,4}
{
\def\pointnameA{A\x}
\coordinate (\pointnameA) at ($(\x*360/5 + 90:1)$);
}
\draw (A0)node[above]{$v_{2}$}--(A1)node[left]{$v_{3}$};
\draw (A0)--(A4)node[right]{$v_{1}$};
\foreach \x in {0,...,4}
{
\draw (A\x)--(0, 0);
\node[circle, inner sep = 1.5, fill = white, draw] () at (A\x) {};
}
\node[regular polygon, inner sep = 1.5, fill=blue, draw=blue] () at (0, 0) {};
\draw (A2)node[left]{$v_{4}$};
\draw (A3)node[right]{$v_{5}$};
\end{tikzpicture}
\begin{tikzpicture}[scale=0.9]
\foreach \x in {0,...,4}
{
\def\pointnameA{A\x}
\coordinate (\pointnameA) at ($(\x*360/5 + 90:1)$);
}
\draw (A2)node[left]{$v_{1}$}--(A3)node[right]{$v_{2}$};
\foreach \x in {0,...,4}
{
\draw (A\x)--(0, 0);
\node[circle, inner sep = 1.5, fill = white, draw] () at (A\x) {};
}
\node[regular polygon, inner sep = 1.5, fill=blue, draw=blue] () at (0, 0) {};
\draw (A4)node[right]{$v_{3}$};
\draw (A0)node[above]{$v_{4}$};
\draw (A1)node[left]{$v_{5}$};
\end{tikzpicture}
\begin{tikzpicture}[scale=0.9]
\foreach \x in {0,...,4}
{
\def\pointnameA{A\x}
\coordinate (\pointnameA) at ($(\x*360/5 + 90:1)$);
}
\draw (A1)node[left]{$v_{1}$}--(A2)node[left]{$v_{2}$};
\draw (A3)node[right]{$v_{3}$}--(A4)node[right]{$v_{4}$};
\coordinate (B0) at ($(A0)+(90:0.6)$);
\coordinate (B1) at ($(A0)+(45:1)$);
\coordinate (B2) at ($(A0)+(135:1)$);
\draw (A0)node[right]{$v_{5}$}--(B1);
\draw (A0)--(B2);
\draw (A0)--(B0)--(B1);
\draw (B0)[dotted]--(B2);
\foreach \x in {0,...,4}
{
\draw (A\x)--(0, 0);
\node[circle, inner sep = 1.5, fill = white, draw] () at (A\x) {};
}
\node[regular polygon, inner sep = 1.5, fill=blue, draw=blue] () at (0, 0) {};
\node[circle, inner sep = 1.5, fill = white, draw] () at (B0) {};
\node[circle, inner sep = 1.5, fill = white, draw] () at (B1) {};
\node[circle, inner sep = 1.5, fill = white, draw] () at (B2) {};
\end{tikzpicture}
\begin{tikzpicture}[scale=0.9]
\foreach \x in {0,...,4}
{
\def\pointnameA{A\x}
\coordinate (\pointnameA) at ($(\x*360/5 + 90:1)$);
}
\draw (A1)node[left]{$v_{1}$}--(A2)node[left]{$v_{2}$};
\draw (A3)node[right]{$v_{3}$}--(A4)node[right]{$v_{4}$};
\foreach \x in {0,...,4}
{
\draw (A\x)--(0, 0);
\node[circle, inner sep = 1.5, fill = white, draw] () at (A\x) {};
}
\node[regular polygon, inner sep = 1.5, fill=blue, draw=blue] () at (0, 0) {};
\draw (A0)node[right]{$v_{5}$};
\end{tikzpicture}
\caption{Some subcases in \autoref{5S}.}
\end{figure}
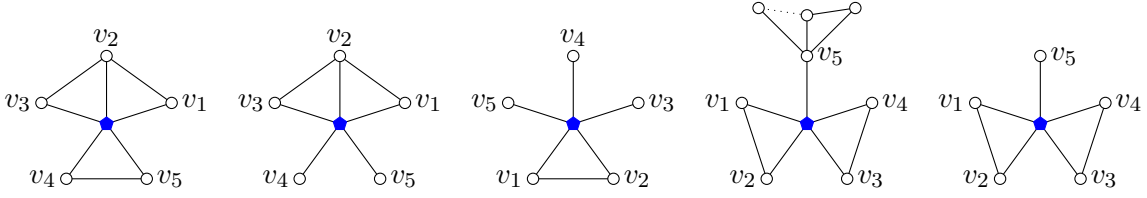

\begin{subcase}
Suppose that $f_{2}$ is not an $F_{5}$-face. Then $v$ sends at most $\frac{1}{3}$ to $f_{2}$. If $v_{5}$ is a $5^{+}$-vertex, then $v$ sends at most $\frac{1}{6}$ to each of $f_{4}$ and $f_{5}$ by \ref{R4}, and then $\mu'(v) \geq 5 - 4 - \frac{1}{3} - 2 \times \frac{1}{6} - 2 \times \frac{2}{15} > 0$. If $v_{5}$ is a $4$-vertex, then $f_{4}$ or $f_{5}$ is an $\epsilon$-face, and $v$ sends zero to each incident $\epsilon$-face, so $\mu'(v) \geq 5 - 4 - 2 \times \frac{1}{3} - 2 \times \frac{2}{15} > 0$. 
\end{subcase}

\begin{subcase}
Suppose that $f_{2}$ is an $F_{5}$-face. By \ref{R4}, we have $\tau(v \rightarrow f_{2}) = \frac{2}{3}$. By \autoref{NPendent}, $v$ cannot be a pendent vertex of any $4$-regular face. By \autoref{Lem-RC}\ref{Lem-RC2}, neither $f_{4}$ nor $f_{5}$ is a $(5, 4, 4, 4, 4)$-face. 

\begin{subsubcase}
Let $v_{5}$ be a $4$-vertex. It follows that $f_{4}$ or $f_{5}$ is an $\epsilon$-face, otherwise there exists a configuration isomorphic to \autoref{fig:subfig:TB}, a contradiction. But if $f_{4}$ or  $f_{5}$ is an $\epsilon$-face, then $v$ sends nothing to it. If $v$ sends a positive charge to $f_{4}$ or $f_{5}$, then the face must be a $5$-face incident with at least two $5^{+}$-vertices, and then $v$ sends at most $\frac{\frac{4}{3} - 1}{2} = \frac{1}{6}$ to $f_{4}$ and $f_{5}$. Thus, $\mu'(v) \geq 5 - 4 - \frac{2}{3} - \frac{1}{6} > 0$. 
\end{subsubcase}

\begin{subsubcase}
Let $v_{5}$ be a $5$-vertex. By \autoref{Lem-RC}\ref{Lem-RC1}, either $v_{1}$ or $v_{4}$ is a $5^{+}$-vertex. Then one of $f_{4}$ and $f_{5}$ is incident with at least three $5^{+}$-vertices, and the other one is incident with at least two $5^{+}$-vertices. By \ref{R4} and \ref{R5}, $v$ sends at most $\frac{\frac{4}{3} - 1}{3} = \frac{1}{9}$ to one of $f_{4}$ and $f_{5}$, sends at most $\frac{\frac{4}{3} - 1}{2} = \frac{1}{6}$ to the other one, thus $\mu'(v) \geq 5 - 4 - \frac{2}{3} - \frac{1}{6} - \frac{1}{9} > 0$. 
\end{subsubcase}

\begin{subsubcase}\label{S}
Let $v_{5}$ be a $6^{+}$-vertex. Since each of $f_{4}$ and $f_{5}$ is incident with at least two $5^{+}$-vertices, $v$ sends at most $\frac{1}{6}$ to each of $f_{4}$ and $f_{5}$ by \ref{R4}, and then $\mu'(v) \geq 5 - 4 - \frac{2}{3} - 2 \times \frac{1}{6} = 0$. By the previous arguments, every $6^{+}$-vertex has positive final charge, so we have that $\mu'(v_{5}) > 0$. 
\end{subsubcase}
\end{subcase}
\end{case}

Since $G$ is not a $4$-regular graph, $G$ must contain some $5^{+}$-vertices. By the previous arguments, all the $5^{+}$-vertices have positive final charge, except the $5$-vertices in \autoref{S}, but in that case there is a vertex such that its final charge is positive. This leads to a contradiction, thus it completes the proof of \autoref{T345}. 
\setcounter{case}{0}
\end{proof}

\begin{proof}[Proof of \autoref{T345-Weak3}]
Let $G$ be a counterexample with $|V(G)|$ as small as possible. Then $G$ is connected.  Suppose that $G$ is $4$-regular. Then $G$ cannot be a GDP-tree, this contradicts \autoref{Brooks-Weak}. Therefore, $G$ is not $4$-regular. By \autoref{Weak-Reduce}, the minimum degree of $G$ is at least $4$, and there are no subgraphs isomorphic to the configurations in \autoref{F35} and \autoref{RC}. But this contradicts \autoref{T345}, so we complete the proof of \autoref{T345-Weak3}. 
\end{proof}

\begin{proof}[Proof of \autoref{T345-SFDT}]
Let $(G, H, f)$ be a counterexample to \autoref{T345-SFDT} with $|V(G)|$ as small as possible. Note that $(H, f)$ is minimal non-strictly $f$-degenerate. By \autoref{CRITICAL}\ref{M1}, $G$ is connected and $\delta(G) \geq 4$. Suppose that $G$ is $4$-regular. Then $G$ cannot be a GDP-tree, but this contradicts \autoref{CRITICAL}\ref{M2}. Therefore, $G$ is not $4$-regular. 

Suppose that $G$ contains a subgraph isomorphic to the configuration in \autoref{F35}. Let $B$ be the set of solid vertices in the configuration. Note the $G[B]$ is neither a complete graph $K_{6}$ nor a cycle, and has a vertex of degree at least three. This contradicts \autoref{CRITICAL}\ref{M2}. 

Suppose that $G$ contains a subgraph isomorphic to a configuration in \autoref{RC}. Let $B$ be the set of solid vertices in the configuration. Note that the sequence $a_{1}, a_{2}, \dots, a_{t}$ satisfies all the conditions of \autoref{NN}. It follows that $H$ has a strictly $f$-degenerate transversal, a contradiction. 

Now, $G$ has minimum degree at least four, none of the configurations in \autoref{F35} and \autoref{RC} will appear in $G$, and it is not $4$-regular, but this contradicts \autoref{T345}. This completes the proof of \autoref{T345-SFDT}. 
\end{proof}

\section{Planar graphs without intersecting 5-cycles}
In this section, we first prove \autoref{Intersecting}.

\begin{proof}[Proof of \autoref{Intersecting}]
Suppose that $G$ is a counterexample to \autoref{Intersecting} with $|V(G)|$ is minimum. If $C_{0}$ is a separating $3$-cycle, then $G - \ext(C_{0})$ is a smaller counterexample, a contradiction. So $C_{0}$ is a facial cycle, we may assume that $C_{0}$ bounds the outer face $D$. Similarly, if $C$ is a separating $3$-cycle in $G$, then $G - \ext(C)$ is a smaller counterexample, a contradiction. Then $G$ has no separating $3$-cycle. In other words, 
\begin{itemize}
\item there are no separating $3$-cycles; and
\item every vertex not in $V(C_{0})$ has degree at least four; and
\item there is no induced subgraph isomorphic to a configuration in \autoref{Kite}, \autoref{F35}, \autoref{fig:subfig:RC1-a}, \autoref{RC-1}, or \autoref{RC-2}.
\end{itemize}

For a vertex $v$, let $\diamond(v)$ denote the number of incident internal $4$-faces. A \emph{special} vertex $v$ is an internal 4-vertex with $\triangledown(v) = 3$, or $\triangledown(v) = 2$ and $\diamond(v) = 1$. Note that every $4$-vertex is incident with at most three $3$-faces. 

\begin{lemma}
\label{Chord1}
The boundary of every $4$-face has no chord. Consequently, if a $4$-face is adjacent to a $3$-face, then they are normally adjacent. 
\end{lemma}
\begin{proof}
Assume a $4$-cycle $u_{1}u_{2}u_{3}u_{4}$ bounds a $4$-face with $u_{1}u_{3} \in E(G)$. Since $u_{1}u_{2}u_{3}$ is a triangle, it bounds a $3$-face. Then $u_{2}$ is a $2$-vertex on $D$, and $u_{1}u_{3}u_{4}$ is a separating $3$-cycle in $G$, a contradiction. 
\end{proof}

\begin{lemma}
\label{Chord2}
The boundary of every internal $5$-face has no chord. Consequently, if an internal $5$-face is adjacent to a $3$-face, then they are normally adjacent. 
\end{lemma}
\begin{proof}
Assume an internal $5$-cycle $u_{1}u_{2}u_{3}u_{4}u_{5}$ bounds a $5$-face with $u_{1}u_{3} \in E(G)$. Since $u_{1}u_{2}u_{3}$ is a triangle, it bounds a $3$-face. Then $u_{2}$ is an internal $2$-vertex, a contradiction. 
\end{proof}

\begin{lemma}\label{SSS}
Let $[x_{1}x_{2}x_{3}x_{4}x_{5}]$ be an internal $5$-face with $d(x_{i}) = 4$ for all $1 \leq i \leq 5$. If $x_{4}x_{5}$ is incident with a $3$-face $[x_{4}x_{5}w]$, then $w$ is on $C_{0}$ or $w$ is an internal $5^{+}$-vertex. 
\end{lemma}
\begin{proof}
Assume $w$ is an internal $4$-vertex. By \autoref{Chord2}, $x_{1}x_{2}x_{3}x_{4}x_{5}x_{1}$ is an induced $5$-cycle and $w \notin \{x_{1}, x_{2}, \dots, x_{5}\}$. If $w$ is adjacent to $x_{2}$, then $wx_{2}x_{1}x_{5}x_{4}w$ and $wx_{2}x_{3}x_{4}x_{5}w$ are two intersecting $5$-cycles, a contradiction. If $w$ is adjacent to $x_{1}$ or $x_{3}$, then $x_{5}$ or $x_{4}$ is an internal $3$-vertex, a contradiction. Then there is an induced subgraph isomorphic to the configuration in \autoref{F35}, contracting the assumption of the counterexample. 
\end{proof}

\begin{lemma}\label{L4}
\mbox{}
\begin{enumerate}[label = (\roman*)]
\item
\label{x.1}
Every special vertex should be as depicted in \autoref{special}. 

\item
\label{x.2}
$w_{1}$ and $w_{3}$ are nonadjacent in \autoref{special}. Symmetrically, $w_{2}$ and $w_{4}$ are nonadjacent in \autoref{fig:subfig:special-1}. 

\item
\label{x.3}
Each of $w_{1}w_{2}, w_{2}w_{3}, w_{3}w_{4}, ww_{4}$ and $ww_{1}$ in \autoref{fig:subfig:special-1} is incident with a $6^{+}$-face; each of $w_{1}w_{2}, w_{2}w_{3}$ and $ww_{1}$ in \autoref{fig:subfig:special-2} is also incident with a $6^{+}$-face.

\item
\label{x.4}
A $5$-face is not incident with any special vertex. 

\item
There are no adjacent special vertices. 

\item
\label{x.5}
Every special vertex is incident with at most one internal $(4, 4, 4)$-face. 
\end{enumerate}
\end{lemma}
\begin{proof}
If a special vertex is incident with three $3$-faces, then it must be as depicted in \autoref{fig:subfig:special-1}. If a special vertex is incident with exactly two $3$-faces and one $4$-face, then the two $3$-faces must be adjacent (see \autoref{fig:subfig:special-2}), for otherwise there are intersecting $5$-cycles. Since there are no intersecting $5$-cycles, $w_{1}$ and $w_{3}$ are nonadjacent in \autoref{special}. For the same reason, each of $w_{1}w_{2}, w_{2}w_{3}, w_{3}w_{4}, ww_{4}$ and $ww_{1}$ in \autoref{fig:subfig:special-1} is incident with a $6^{+}$-face; each of $w_{1}w_{2}, w_{2}w_{3}$ and $ww_{1}$ in \autoref{fig:subfig:special-2} is also incident with a $6^{+}$-face. Consequently, a $5$-face is not incident with a special vertex. In \autoref{fig:subfig:special-2}, the edge $xw_{4}$ cannot be incident with a $3$-face, thus $w_{4}$ cannot be a special vertex. According to the required faces in a cyclic order around a special vertex, none of $w_{1}, w_{2}, w_{3}, w_{4}$ can be a special vertex, so there are no adjacent special vertices. It follows from \ref{x.2} and the absence of configuration in \autoref{Kite} that every special vertex is incident with at most one internal $(4, 4, 4)$-face.      
\end{proof}

Since there are no intersecting $5$-cycles, it is easy to obtain the following result. 
\begin{lemma}\label{TRIANGULAR}
Every internal $d$-vertex is incident with at most $\frac{2d+1}{3}$ triangular faces and at most one $5$-face. 
\end{lemma}

\begin{lemma}\label{BIGFACE}
If an internal vertex $v$ is incident with a $3$-face, then $v$ is incident with at least one $6^{+}$-face. 
\end{lemma}
\begin{proof}
Since $v$ has degree at least four, it is incident with at least one $4^{+}$-face. Assume $v$ is incident with exactly one $4^{+}$-face. Then the incident $3$-faces form a $5$-cycle, which implies that the only $4^{+}$-face must be a $6^{+}$-face, for otherwise there are intersecting $5$-cycles. So we may assume that $v$ is incident with at least two $4^{+}$-faces. Let $f_{1}, f_{2}, \dots, f_{k}$ be consecutive faces in a cyclic order. Suppose that $f_{1}, f_{t}$ are $4^{+}$-faces, and $f_{2}, \dots, f_{t-1}$ are $3$-faces. It is easy to check that one of $f_{1}$ and $f_{t}$ is a $6^{+}$-face.  
\end{proof}

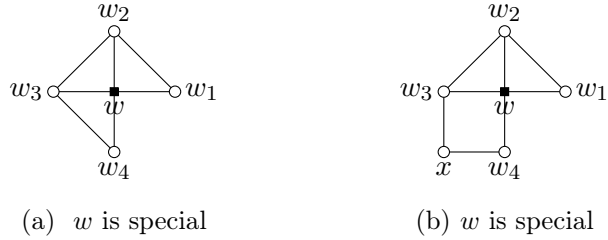
\begin{figure}%
\centering
\subcaptionbox{\label{fig:subfig:special-1} $w$ is special}{\begin{tikzpicture}
\def\s{0.8}
\coordinate (O) at (0, 0);
\coordinate (v1) at (\s, 0);
\coordinate (v2) at (0, \s);
\coordinate (v3) at (-\s, 0);
\coordinate (v4) at (0, -\s);
\draw (v1)node[right]{$w_{1}$}--(v2)node[above]{$w_{2}$}--(v3)node[left]{$w_{3}$}--(v4)node[below]{$w_{4}$}--(O)node[below]{$w$}--cycle;
\draw (O)--(v2);
\draw (O)--(v3);
\node[rectangle, inner sep = 1.5, fill, draw] () at (O) {};
\node[circle, inner sep = 1.5, fill = white, draw] () at (v1) {};
\node[circle, inner sep = 1.5, fill = white, draw] () at (v2) {};
\node[circle, inner sep = 1.5, fill = white, draw] () at (v3) {};
\node[circle, inner sep = 1.5, fill = white, draw] () at (v4) {};
\end{tikzpicture}}\hspace{2cm}
\subcaptionbox{\label{fig:subfig:special-2}$w$ is special}{\begin{tikzpicture}
\def\s{0.8}
\coordinate (O) at (0, 0);
\coordinate (v1) at (\s, 0);
\coordinate (v2) at (0, \s);
\coordinate (v3) at (-\s, 0);
\coordinate (v4) at (0, -\s);
\coordinate (A) at (-\s, -\s);
\draw (v1)node[right]{$w_{1}$}--(v2)node[above]{$w_{2}$}--(v3)node[left]{$w_{3}$}--(A)node[below]{$x$}--(v4)node[below]{$w_{4}$}--(O)node[below]{$w$}--cycle;
\draw (O)--(v2);
\draw (O)--(v3);
\node[rectangle, inner sep = 1.5, fill, draw] () at (O) {};
\node[circle, inner sep = 1.5, fill = white, draw] () at (v1) {};
\node[circle, inner sep = 1.5, fill = white, draw] () at (v2) {};
\node[circle, inner sep = 1.5, fill = white, draw] () at (v3) {};
\node[circle, inner sep = 1.5, fill = white, draw] () at (v4) {};
\node[circle, inner sep = 1.5, fill = white, draw] () at (A) {};
\end{tikzpicture}}
\caption{Two types of special vertex.}
\label{special}
\end{figure}

Next, we assign each vertex $v$ an initial charge $\mu(v) = 2d(v) - 6$, assign each bounded face $f$ an initial charge $\mu(f) = d(f) - 6$, and assign the outer face $D$ an initial charge $\mu(D) = d(D) + 6$. By Euler's formula, we can rewrite it in the following form, 
\[
\sum_{v \in V(G)} (2d(v) - 6) + \sum_{f \in F(G)\setminus \{D\}}(d(f) - 6) + (\mu(D) + 6) = 0. 
\]
We design some discharging rules to redistribute charges, such that after the discharging process, a new charge function $\mu'$ is produced, and $\mu'(x) \geq 0$ for all $x \in V(G) \cup F(G)$, in particular $\mu'(D) > 0$, which leads to a contradiction. 

The followings are the discharging rules.
\begin{enumerate}[label = \textbf{R\arabic*.}, ref = R\arabic*]
\item
\label{S1}
If $f$ is an internal $(4, 4, 4)$-face or an internal $3$-face with no incident special vertex, then $f$ receives $1$ from each incident vertex. 
\item
\label{S2}
Let $f = [uvw]$ be an internal 3-face, and let $v$ be a special vertex and $d(w) \geq 5$ ($u$ cannot be special). Then 
\begin{align*}
\tau(v \rightarrow f) = \alpha = 
&\begin{cases}
\frac{2}{3}, & \text{if $\triangledown(v) = 3$ and no internal $(4, 4, 4)$-face is incident with $v$;}\\[0.3cm]
\frac{3}{4}, & \text{if $\triangledown(v) = 2$ and no internal $(4, 4, 4)$-face is incident with $v$;}\\[0.3cm]
\frac{1}{2}, & \text{otherwise}.
\end{cases}\\[0.3cm]
\tau(u \rightarrow f) = 
&\begin{cases}
1, & \text{if $d(u) = 4$;}\\[0.3cm]
\frac{3 - \alpha}{2}, & \text{otherwise}.
\end{cases}\\[0.3cm]
\tau(w \rightarrow f) = 
&\begin{cases}
2 - \alpha, & \text{if $d(u) = 4$;}\\[0.3cm]
\frac{3 - \alpha}{2}, & \text{otherwise}.
\end{cases}
\end{align*}

\item
\label{S3}
Every internal 4-face receives $\frac{1}{2}$ from each incident vertex. 
\item
\label{S4}
Let $f$ be an internal 5-face. 
\begin{enumerate}
\item
\label{S4-a}
If $f$ is incident with a $5^{+}$-vertex, then $f$ receives $\frac{1}{n(f)}$ from each incident $5^{+}$-vertex, where $n(f)$ is the number of $5^{+}$-vertices incident with $f$. 
\item
\label{S4-b}
If $f$ is a $(4, 4, 4, 4, 4)$-face incident with a vertex $v$ and $\triangledown(v) \leq 1$, then $f$ receives $\frac{1}{2}$ from $v$. 
\item
\label{S4-c}
If $f$ is a $(4, 4, 4, 4, 4)$-face incident with an edge $uv$ and $[uvw]$ is an internal 3-face, then $f$ receives $\frac{1}{4}$ from $w$ (By \autoref{SSS}, $w$ is an internal $5^{+}$-vertex). 
\item
\label{S4-d}
If $f$ is a $(4, 4, 4, 4, 4)$-face incident with an edge $uv$ and $[uvw]$ is a 3-face with $w \in V(C_{0})$, then $f$ receives $\frac{1}{4}$ from the outer face $D$. 
\end{enumerate}
\item
\label{S5}
If $[uvw]$ is a $3$-face having exactly one internal vertex $w$, then $w$ sends $1$ to $f$. 
\item
\label{S6}
If $[uvw]$ is a $3$-face having exactly two internal vertices $u$ and $v$, then each of $u$ and $v$ sends $\frac{1}{2}$ to $f$. 
\item
\label{S7}
Every vertex $v$ on $C_{0}$ sends its initial charge $\mu(v)$ to $D$. 
\item
\label{S8}
If $f$ is a bounded $5^{-}$-face having a common vertex with $D$, then $D$ sends $2$ to $f$. 
\end{enumerate}

By \ref{S1} and \ref{S2}, every internal $5^{+}$-vertex sends at most $\frac{3}{2}$ to each incident 3-face. If an internal $5^{+}$-vertex $w$ sends $\frac{1}{4}$ via a 3-face $g_{1}$ to an internal $(4, 4, 4, 4, 4)$-face $g_{2}$, then it follows from \autoref{L4}\ref{x.3} that no vertex on $g_{1}$ is special, and $w$ sends $1$ to $g_{1}$, and then $\tau(w \rightarrow g_{1}) + \tau(w \rightarrow g_{2}) = \frac{5}{4}$. By \ref{S4}, every internal $5^{+}$-vertex sends at most $1$ to incident internal $5$-face.

We claim that $\mu'(x) \geq 0$ for all $x \in V(G) \cup F(G)$. If $f$ is a $3$-face with $|b(f) \cap V(C_{0})| = 2$, then $\mu'(f) = \mu(f) + 2 + 1 = 0$ by \ref{S7} and \ref{S5}. If $f$ is a $3$-face with $|b(f) \cap V(C_{0})| = 1$, then $\mu'(f) = \mu(f) + 2 + 2 \times \frac{1}{2} = 0$ by \ref{S7} and \ref{S6}. It is impossible that $f \neq D$ and $|b(f) \cap V(C_{0})| = 3$. If $f$ is a $4^{+}$-face with $b(f) \cap V(C_{0}) \neq \emptyset$, then $\mu'(f) \geq \mu(f) + 2 \geq 0$ by \ref{S7}. If $f$ is an internal $3$-face, then $\mu'(f) = \mu(f) + 3 = 0$ by \ref{S1} and \ref{S2}. If $f$ is an internal $4$-face, then $\mu'(f) = \mu(f) + 4 \times \frac{1}{2} = 0$ by \ref{S3}. If $f$ is an internal $6^{+}$-face, then $\mu'(f) = \mu(f) = d(f) - 6 \geq 0$. 

Let $f$ be an internal 5-face. If $f$ is incident with a $5^{+}$-vertex, then $\mu'(f) = \mu(f) + n(f) \times \frac{1}{n(f)} = 0$ by \ref{S4-a}. So we may assume that $f$ is an internal $(4, 4, 4, 4, 4)$-face $[v_{1}v_{2}v_{3}v_{4}v_{5}]$. By \autoref{L4}\ref{x.4}, $f$ is not incident with any special vertex. Then every vertex $v_{i}$ on $f$ has $\triangledown(v_{i}) \leq 2$. If there are two incident vertices $v_{i}$ and $v_{j}$ such that $\triangledown(v_{i}) \leq 1$ and $\triangledown(v_{j}) \leq 1$, then $\mu'(f) \geq \mu(f) + 2 \times \frac{1}{2} = 0$ by \ref{S4-b}. If $\triangledown(v_{i}) = 2$ for all $1 \leq i \leq 5$, then $f$ is adjacent to at least four 3-faces, implying that $\mu'(f) \geq \mu(f) + 4 \times \frac{1}{4} = 0$ by \ref{S4-c} and \ref{S4-d}. Without loss of generality, we can assume that $\triangledown(v_{1}) \leq 1$ and $\triangledown(v_{i}) = 2$ for all $2 \leq i \leq 5$. In this case, $f$ is adjacent to at least two 3-faces, thus $\mu'(f) \geq \mu(f) + \frac{1}{2} + 2 \times \frac{1}{4} = 0$ by \ref{S4}. 

Let $v$ be an internal 4-vertex. Recall that every $4$-vertex is incident with at most three $3$-faces. 
\begin{itemize}
\item If $\triangledown(v) = 0$, then $v$ sends at most $\frac{1}{2}$ to each incident face, and then $\mu'(v) \geq 2 - 4 \times \frac{1}{2} = 0$. 

\item If $\triangledown(v) = 1$, then it follows from \autoref{BIGFACE} that $v$ is incident with a $6^{+}$-face, and then $\mu'(v) \geq 2 - 1 - 2 \times \frac{1}{2} = 0$. 

\item Assume $\triangledown(v) = 2$ and $\diamond(v) = 1$. Then $v$ is the special vertex as depicted in \autoref{fig:subfig:special-2}. It follows from \autoref{L4}\ref{x.3} that each of the three incident $4^{-}$-faces has at most one common vertex with $D$. If $v$ is not incident with an internal $(4, 4, 4)$-face, then it sends at most $\frac{3}{4}$ to each incident $3$-face and $\frac{1}{2}$ to the incident internal $4$-face, implying that $\mu'(v) \geq 2 - 2 \times \frac{3}{4} - \frac{1}{2} = 0$. If $v$ is incident with an internal $(4, 4, 4)$-face, then by \autoref{L4}\ref{x.5}, the other $3$-face is incident with an internal $5^{+}$-vertex or a vertex on $C_{0}$, implying $\mu'(v) = 2 - 1 - \frac{1}{2} - \frac{1}{2} = 0$. 

\item If $\triangledown(v) = 2$ and $\diamond(v) = 0$, then $v$ only sends charge to incident 3-faces, and then $\mu'(v) \geq 2 - 2 \times 1 = 0$. 

\item It is impossible that $\triangledown(v) = \diamond(v) = 2$. 

\item Assume $\triangledown(v) = 3$. Then $v$ is the special vertex as depicted in \autoref{fig:subfig:special-1}. It follows from \autoref{L4}\ref{x.3} that each of the three incident $3$-faces has at most one common vertex with $D$. If $v$ is not incident with an internal $(4, 4, 4)$-face, then $\mu'(v) \geq 2 - 3 \times \frac{2}{3} = 0$. If $v$ is incident with an internal $(4, 4, 4)$-face, then by \autoref{L4}\ref{x.5}, each of the other two $3$-face is incident with a boundary vertex or an internal $5^{+}$-vertex, implying $\mu'(v) \geq 2 - 1 - 2 \times \frac{1}{2} = 0$. 
\end{itemize}

Let $v$ be an internal 6-vertex. By \autoref{TRIANGULAR}, $v$ is incident with at most four $3$-faces. Note that $v$ sends at most $\frac{3}{2}$ to each incident $3$-face. If $\triangledown(v) \leq 2$, then $\mu'(v) \geq 6 - 2 \times \frac{3}{2} - 1 - 3 \times \frac{1}{2} > 0$. If $\triangledown(v) = 3$, then $v$ is incident with a $6^{+}$-face, and then $\mu'(v) \geq 6 - 3 \times \frac{3}{2} - 1 - \frac{1}{2} = 0$. If $\triangledown(v) = 4$ and $v$ is not incident with a 5-face, then $v$ is not incident with a $4$-face, implying $\mu'(v) \geq 6 - 4 \times \frac{3}{2} = 0$. If $\triangledown(v) = 4$ and $v$ is incident with a 5-face, then by \autoref{L4}\ref{x.3}, $v$ is not adjacent to a special vertex, implying $\mu'(v) \geq 6 - 4 \times (1 + \frac{1}{4}) - 1 = 0$. 

Let $v$ be an internal 7-vertex. By \autoref{TRIANGULAR}, $v$ is incident with at most five $3$-faces. If $\triangledown(v) = 5$, then $v$ is incident with two $6^{+}$-faces, and $\mu'(v) \geq 8 - 5 \times \frac{3}{2} > 0$. If $\triangledown(v) \leq 4$, then $\mu'(v) \geq 8 - 4 \times \frac{3}{2} - 1 - 2 \times \frac{1}{2} = 0$. 

Let $v$ be an internal $d$-vertex with $d \geq 8$. If $d= 8$, then $\triangledown(v) \leq 5$ and $\mu'(v) \geq 10 - 5 \times \frac{3}{2} - 1 - 2 \times \frac{1}{2} > 0$. If $d \geq 9$, then 
\[
\mu'(v) \geq 2d - 6 - \frac{2d+1}{3} \times \frac{3}{2} - 1 - \left(d - \frac{2d+1}{3} - 1\right) \times \frac{1}{2} = \frac{5d - 41}{6} > 0. 
\]

\textsl{In the following, assume that $v$ is an internal 5-vertex with neighbors $v_{1}, v_{2}, v_{3}, v_{4}, v_{5}$ in a cyclic order, and $f_{i}$ is the face incident with $vv_{i}$ and $vv_{i+1}$, where the subscripts are taken module $5$.} 

By \autoref{TRIANGULAR}, $\triangledown(v) \leq 3$. If $\triangledown(v) = 0$, then $v$ sends $\frac{1}{2}$ to each incident internal 4-face and sends at most $1$ to incident internal 5-face, and then $\mu'(v) \geq 4 - 1 - 4 \times \frac{1}{2} = 1$. If $\triangledown(v) = 1$, then it follows from \autoref{BIGFACE} that $v$ is incident with a $6^{+}$-face, and $\mu'(v) \geq 4 - \frac{3}{2} - 1 - 2 \times \frac{1}{2} > 0$. 

Assume $\triangledown(v) = 2$, $f_{1}$ and $f_{2}$ are two adjacent 3-faces. Since there are no intersecting 5-cycles, at least one of $f_{3}$ and $f_{5}$ is a $6^{+}$-face, at most one of $v_{1}v_{2}$ and $v_{2}v_{3}$ is incident with an internal $(4, 4, 4, 4, 4)$-face. By \autoref{L4}\ref{x.3}, neither $v_{1}$ nor $v_{3}$ is a special vertex. If $v_{2}$ is not a special vertex, then $\mu'(v) \geq 4 - 2 \times 1 - \frac{1}{2} - 1 - \frac{1}{4} > 0$. If $v_{2}$ is a special vertex, then it follows from \autoref{L4} that $f_{3}$ and $f_{5}$ are $6^{+}$-faces, and $v_{2}$ is not incident with an internal $(4, 4, 4, 4, 4)$-face, thus $\mu'(v) \geq 4 - 2 \times \frac{3}{2} - 1 = 0$. 

Assume $\triangledown(v) = 2$, $f_{1}$ and $f_{3}$ are two nonadjacent 3-faces. Since there are no intersecting 5-cycles, at least two of $f_{2}, f_{4}$ and $f_{5}$ are $6^{+}$-faces. This implies that $\mu'(v) \geq 4 - 2 \times \frac{3}{2} - 1 = 0$. 

Assume $\triangledown(v) = 3$,  and $f_{1}, f_{2}, f_{3}$ are 3-faces. Since there are no intersecting 5-cycles, none of $v_{1}, v_{2}, v_{3}$ and $v_{4}$ is a special vertex, both $f_{4}$ and $f_{5}$ are $6^{+}$-faces. For the same reason, none of $v_{1}v_{2}, v_{2}v_{3}$ and $v_{3}v_{4}$ is incident with an internal $(4, 4, 4, 4, 4)$-face. Then $\mu'(v) \geq 4 - 3 \times 1 = 1$.

\textsl{In the following, we can assume that $\triangledown(v) = 3$,  and $f_{1}, f_{2}, f_{4}$ are 3-faces.} Since there are no intersecting 5-cycles, we may further assume that $d(f_{3}) \geq 5$ and $d(f_{5}) \geq 6$. Moreover, neither $v_{1}$ nor $v_{3}$ is a special vertex; at most one of $v_{1}v_{2}$ and $v_{2}v_{3}$ is incident with an internal $(4, 4, 4, 4, 4)$-face.

Assume $f_{3}$ is a 5-face $[vv_{3}uwv_{4}]$. It follows from \autoref{L4}\ref{x.3} and \ref{x.4} that $v$ is not adjacent to a special vertex. Since there are no intersecting 5-cycles, neither $v_{2}v_{3}$ nor $v_{4}v_{5}$ is incident with an internal $(4, 4, 4, 4, 4)$-face. If $v_{1}v_{2}$ is not incident with an internal $(4, 4, 4, 4, 4)$-face, then $\mu'(v) \geq 4 - 3 \times 1 - 1 = 0$. Assume $v_{1}v_{2}$ is incident with a $(4, 4, 4, 4, 4)$-face $f'$. Then $v_{3}$ is on $C_{0}$ or an internal $5^{+}$-vertex, otherwise $f_{1}, f_{2}$ and $f'$ form an induced subgraph isomorphic to the configuration in \autoref{fig:subfig:RC1-a}. In this situation, $f_{3}$ is incident with two internal $5^{+}$-vertices or it has a common vertex with $D$, thus $v$ sends at most $\frac{1}{2}$ to $f_{3}$, and $\mu'(v) \geq 4 - 2 \times 1 - \frac{1}{2} - (1 + \frac{1}{4}) > 0$.

Assume that $f_{3}$ is a $6^{+}$-face. If $v$ is incident with a 3-face having a common vertex with $D$, then $\mu'(v) \geq 4 - 1 - 2 \times \frac{3}{2} = 0$. Then we may assume that the three incident 3-faces are all internal faces. If $v_{2}$ is not a special vertex, then $\mu'(v) \geq 4 - 1 - (1 + \frac{1}{4}) - \frac{3}{2} > 0$. So we may assume that $v_{2}$ is a special vertex. It is obvious that $v_{2}$ is not incident with an internal $(4, 4, 4, 4, 4)$-face. If $\triangledown(v_{2}) = 2$, then $v_{2}$ is not incident with an internal $(4, 4, 4)$-face, implying that $\tau(v_{2} \rightarrow f_{1}) = \tau(v_{2} \rightarrow f_{2}) = \frac{3}{4}$ and $\mu'(v) \geq 4 - 2 \times (2 - \frac{3}{4}) - \frac{3}{2} = 0$. So we may assume that $v_{2}$ is a special vertex with $\triangledown(v_{2}) = 3$ and $v_{1}v_{2}$ is incident with another $3$-face $[v_{1}v_{2}x]$. By \autoref{L4}\ref{x.3}, $v_{2}$ is incident with a $6^{+}$-face. Now, $v$ is contained in a $5$-cycle $[vv_{1}xv_{2}v_{3}]$. Since there are no intersecting $5$-cycles, neither $vv_{4}$ nor $vv_{5}$ is contained in a $5$-cycle.

Suppose $\tau(v_{2} \rightarrow f_{1}) = \tau(v_{2} \rightarrow f_{2}) = \frac{2}{3}$. Recall that $v$ is contained in a 5-cycle $[v_{1}vv_{3}v_{2}x]$. If neither $v_{4}$ nor $v_{5}$ is a special vertex, then $\mu'(v) \geq 4 - 2 \times (2 - \frac{2}{3}) - (1 + \frac{1}{4}) > 0$. Then $v_{4}$ is a special vertex with $\triangledown(v_{4}) = 2$, for otherwise $v$ is contained in two intersecting 5-cycles. In this situation, if $\tau(v_{4} \rightarrow f_{4}) = \frac{3}{4}$, then $\tau(v \rightarrow f_{4}) \leq \frac{5}{4}$ and $\mu'(v) \geq 4 - 2 \times (2 - \frac{2}{3}) - \frac{5}{4} > 0$. Thus, $\tau(v_{4} \rightarrow f_{4}) = \frac{1}{2}$ and $v_{4}$ is incident with an internal $(4, 4, 4)$-face $[v_{4}v_{5}y]$. Note that $v_{1}$ and $v_{3}$ are nonadjacent, for otherwise $v_{1}vv_{3}$ is a separating $3$-cycle. Since there are no intersecting $5$-cycles, $y$ has exactly two neighbors in $\{v, v_{1}, v_{2}, \dots, v_{5}\}$. If $d(v_{1}) = d(v_{3}) = 4$, then $\{v, v_{1}, v_{2}, v_{3}, v_{4}, v_{5}, y\}$ induces a subgraph isomorphic to the configuration in \autoref{fig:subfig:RC-2-b}, a contradiction. Then $v_{1}$ or $v_{3}$ is an internal $5^{+}$-vertex, $\tau(v \rightarrow f_{1}) + \tau(v \rightarrow f_{2}) \leq (2 - \frac{2}{3}) + \frac{3 - \frac{2}{3}}{2} = \frac{5}{2}$, implying $\mu'(v) \geq 4 - \frac{5}{2} - \frac{3}{2} = 0$.

Suppose $\tau(v_{2} \rightarrow f_{1}) = \tau(v_{2} \rightarrow f_{2}) = \frac{1}{2}$. By \ref{S2}, $[v_{1}v_{2}x]$ is an internal $(4, 4, 4)$-face and $\tau(v \rightarrow f_{1}) = \frac{3}{2}$. Since there is no induced subgraph isomorphic to the configuration in \autoref{RC-1}, $v_{3}$ must be an internal $5^{+}$-vertex, thus $\tau(v \rightarrow f_{2}) = \frac{5}{4}$. If neither $v_{4}$ nor $v_{5}$ is a special vertex, then $\tau(v \rightarrow f_{4}) = 1$ and $\mu'(v) \geq 4 - \frac{3}{2} - \frac{5}{4} - (1 + \frac{1}{4}) = 0$. Without loss of generality, we may assume that $v_{4}$ is a special vertex. It follows that $\triangledown(v_{4}) = 2$, for otherwise $v$ is contained in two different 5-cycles. If $v_{4}$ is incident with an internal $(4, 4, 4)$-face, then there is an induced subgraph isomorphic to the configuration in \autoref{fig:subfig:RC-2-a}, a contradiction. Thus, $v_{4}$ is not incident with an internal $(4, 4, 4)$-face and $\tau(v_{4} \rightarrow f_{4}) = \frac{3}{4}$, this implies that $\tau(v \rightarrow f_{4}) \leq \frac{5}{4}$ and $\mu'(v) \geq 4 - \frac{3}{2} - \frac{5}{4} - \frac{5}{4} = 0$. 

By \ref{S7}, every vertex in $V(C_{0})$ has final charge zero. Finally, consider the final charge of the outer face $D$. Let $g = [uvxyz]$ be an internal $5$-face, and $[uvw]$ is a $3$-face with $w \in V(C_{0})$. Since there are no intersecting $5$-cycles, either $wu$ or $wv$ is incident with a $6^{+}$-face. Without loss of generality, let $wu$ be incident with a $6^{+}$-face $h$. By the discharging rules, $D$ does not sends charges to $h$, but it sends $\frac{1}{4}$ to $g$ via the 3-face $[uvw]$. We can regard the charge $\frac{1}{4}$ sent via $[uvw]$ is through the $6^{+}$-face $h$. Then $D$ sends at most $2 \times \frac{1}{4} = \frac{1}{2}$ through the $6^{+}$-face $h$. Then $D$ sends at most $2$ to/via each face having a common vertex with $D$. Note that there are $\sum_{v\, \in\, V(D)} \big(d(v) - 2\big)$ faces having a common vertex with $D$. It follows that
\[
\mu'(D) \geq \mu(D) + \sum_{v\, \in\, V(D)} \Big(2d(v) - 6\Big) - 2\left(\sum_{v\, \in\, V(D)} \Big(d(v) - 2\Big)\right) = \mu(D) - 2d(D) = 6 - d(D) = 3.
\]
This completes the proof.
\end{proof}

We can prove the following result which imply \autoref{Intersecting-AT}. The proof of \autoref{Intersecting-AT'} is analogous to that of \autoref{PAIRWISE3456-AT'}, we leave it to the reader. The orientations for the related configurations are depicted in \autoref{OR} and \autoref{OR-C5}.
\begin{theorem}\label{Intersecting-AT'}
Let $G$ be a plane graph without intersecting 5-cycles, and let $C_{0} = [x_{1}x_{2}x_{3}]$ be a 3-cycle in $G$. Then $G - E(C_{0})$ has a $4$-AT-orientation such that all the edges incident with $V(C_{0})$ are directed to $V(C_{0})$. 
\end{theorem}

\begin{figure}%
\centering
\def\s{1}
\begin{tikzpicture}
\coordinate (A) at (45:\s);
\coordinate (B) at (135:\s);
\coordinate (C) at (225:\s);
\coordinate (D) at (-45:\s);
\coordinate (H) at (90:1.414*\s);
\coordinate (X) at ($(A)+(0, \s)$);
\coordinate (Y) at (45:2*\s);
\path [draw=black, postaction={on each segment={mid arrow=red}}]
(A)--(Y)--(X)--(H)--(B)--(C)--(D)--cycle
(H)--(A)--(X)
(Y)--($(Y)!0.5cm!90:(A)$)
(Y)--($(Y)!0.5cm!180:(A)$)
(X)--($(X)!0.5cm!80:(Y)$)
(X)--($(X)!0.5cm!-80:(H)$)
(H)--($(H)!0.5cm!180:(A)$)
(B)--($(B)!0.5cm!60:(H)$)
(B)--($(B)!0.5cm!120:(H)$)
(C)--($(C)!0.5cm!180:(D)$)
(C)--($(C)!0.5cm!-90:(D)$)
(D)--($(D)!0.5cm!90:(C)$)
(D)--($(D)!0.5cm!180:(C)$);
\node[rectangle, inner sep = 1.5, fill, draw] () at (A) {};
\node[rectangle, inner sep = 1.5, fill, draw] () at (B) {};
\node[rectangle, inner sep = 1.5, fill, draw] () at (C) {};
\node[rectangle, inner sep = 1.5, fill, draw] () at (D) {};
\node[rectangle, inner sep = 1.5, fill, draw] () at (H) {};
\node[regular polygon, inner sep = 1.5, fill=blue, draw=blue] () at (X) {};
\node[rectangle, inner sep = 1.5, fill, draw] () at (Y) {};
\end{tikzpicture}\hspace{0.5cm}
\begin{tikzpicture}
\coordinate (O) at (0, 0);
\coordinate (v1) at (0:\s);
\coordinate (v2) at (60:\s);
\coordinate (v3) at (120:\s);
\coordinate (v4) at (180:\s);
\path [draw=black, postaction={on each segment={mid arrow=red}}]
(v1)--(O)--(v4)--(v3)--(v2)--cycle
(v2)--(O)
(O)--(v3)
(v1)--($(v1)!0.5!90:(O)$)
(v1)--($(v1)!0.5!-90:(v2)$)
(v2)--($(v2)!0.5!180:(O)$)
(v3)--($(v3)!0.5!90:(v2)$)
(v3)--($(v3)!0.5!-90:(v4)$)
(v4)--($(v4)!0.5!90:(v3)$)
(v4)--($(v4)!0.5!-90:(O)$);
\node[rectangle, inner sep = 1.5, fill, draw] () at (O) {};
\node[regular polygon, inner sep = 1.5, fill=blue, draw=blue] () at (v3) {};
\node[rectangle, inner sep = 1.5, fill, draw] () at (v1) {};
\node[rectangle, inner sep = 1.5, fill, draw] () at (v2) {};
\node[rectangle, inner sep = 1.5, fill, draw] () at (v4) {};
\end{tikzpicture}\hspace{0.5cm}
\begin{tikzpicture}
\coordinate (O) at (0, 0);
\coordinate (v1) at (45: \s);
\coordinate (v2) at (135: \s);
\coordinate (v3) at (225: \s);
\coordinate (v4) at (-45: \s);
\coordinate (A) at (-1.414*\s, 0);
\coordinate (B) at (1.414*\s, 0);
\path [draw=black, postaction={on each segment={mid arrow=red}}]
(v1)--(O)--(v4)--(B)--cycle
(O)--(v2)--(A)--(v3)--cycle
(v2)--(v3)
(v1)--(v4)
(v1)--($(v1)!0.5!135:(B)$)
(v4)--($(v4)!0.5!-135:(B)$)
(v2)--($(v2)!0.5!-135:(A)$)
(v3)--($(v3)!0.5!135:(A)$)
(A)--($(A)!0.5!90:(v2)$)
(A)--($(A)!0.5!180:(v2)$)
(B)--($(B)!0.5!-90:(v1)$)
(B)--($(B)!0.5!-180:(v1)$)
(O)--(0, -0.5*\s);
\node[rectangle, inner sep = 1.5, fill, draw] () at (v1) {};
\node[rectangle, inner sep = 1.5, fill, draw] () at (v2) {};
\node[rectangle, inner sep = 1.5, fill, draw] () at (v3) {};
\node[rectangle, inner sep = 1.5, fill, draw] () at (v4) {};
\node[rectangle, inner sep = 1.5, fill, draw] () at (A) {};
\node[rectangle, inner sep = 1.5, fill, draw] () at (B) {};
\node[regular polygon, inner sep = 1.5, fill=blue, draw=blue] () at (O) {};
\end{tikzpicture}\hspace{0.5cm}
\begin{tikzpicture}
\coordinate (O) at (0, 0);
\coordinate (v1) at (45: \s);
\coordinate (v2) at (135: \s);
\coordinate (v3) at (225: \s);
\coordinate (v4) at (-45: \s);
\coordinate (A) at (-1.414*\s, 0);
\coordinate (B) at (1.414*\s, 0);
\path [draw=black, postaction={on each segment={mid arrow=red}}]
(v1)--(O)--(v4)--(B)--cycle
(O)--(v2)--(A)--(v3)--cycle
(v2)--(v3)
(O)--(B)
(v1)--($(v1)!0.5!90:(B)$)
(v1)--($(v1)!0.5!180:(B)$)
(v4)--($(v4)!0.5!-90:(B)$)
(v4)--($(v4)!0.5!-180:(B)$)
(v2)--($(v2)!0.5!-135:(A)$)
(v3)--($(v3)!0.5!135:(A)$)
(A)--($(A)!0.5!90:(v2)$)
(A)--($(A)!0.5!180:(v2)$)
(B)--($(B)!0.5!-135:(v1)$);
\node[rectangle, inner sep = 1.5, fill, draw] () at (v1) {};
\node[rectangle, inner sep = 1.5, fill, draw] () at (v2) {};
\node[rectangle, inner sep = 1.5, fill, draw] () at (v3) {};
\node[rectangle, inner sep = 1.5, fill, draw] () at (v4) {};
\node[rectangle, inner sep = 1.5, fill, draw] () at (A) {};
\node[rectangle, inner sep = 1.5, fill, draw] () at (B) {};
\node[regular polygon, inner sep = 1.5, fill=blue, draw=blue] () at (O) {};
\end{tikzpicture}
\caption{Orientations of some configurations.}
\label{OR-C5}
\end{figure}
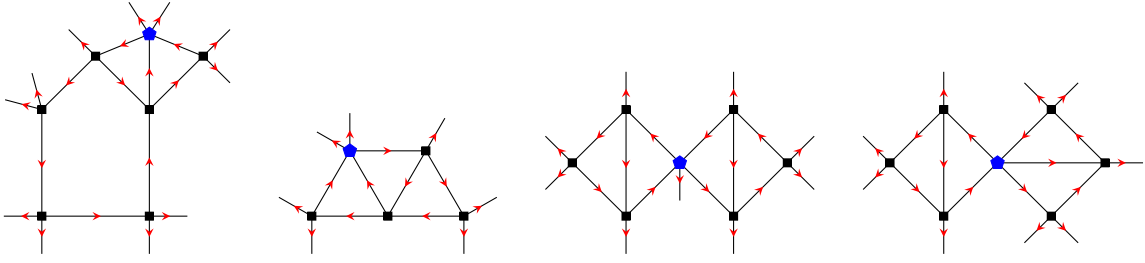

To prove \autoref{Intersecting-Weak3}, we can show the following stronger result as \autoref{PAIRWISE3456-Weak3'}. Using this, we can prove \autoref{Intersecting-Weak3} as the proof of \autoref{PAIRWISE3456-Weak3}.

\begin{theorem}\label{Intersecting-Weak3'}
Let $G$ be a connected planar graph without intersecting $5$-cycles, and let $C_{0}=[x_{1}x_{2}x_{3}]$ be a $3$-cycle in $G$. Define a function $g: V(G)\setminus \{x_{1}, x_{2}, x_{3}\} \longrightarrow \mathbb{N}$ by 
\[
g(u) = 3 - |N_{G}(u) \cap \{x_{1}, x_{2}, x_{3}\}|
\]
for each $u \in V(G)\setminus \{x_{1}, x_{2}, x_{3}\}$. Then $G - \{x_{1}, x_{2}, x_{3}\}$ is weakly $g$-degenerate. 
\end{theorem}
\begin{proof}
Let $G$ together with a $3$-cycle $[x_{1}x_{2}x_{3}]$ be a counterexample to the statement with $|V(G)|$ as small as possible. Fix a plane embedding for $G$ in the plane. It follows that $G - \{x_{1}, x_{2}, x_{3}\}$ is not weakly $g$-degenerate, but $G - \{x_{1}, x_{2}, x_{3}\} - X$ is weakly $g$-degenerate, where $X$ is a nonempty subset of internal vertices. By \autoref{Weak-Reduce}, every internal vertex has degree at least four, and none of the configurations in \autoref{Kite}, \autoref{F35}, \autoref{fig:subfig:RC1-a}, \autoref{RC-1} and \autoref{RC-2} will appear in $G$, but this contradicts \autoref{Intersecting}.
\end{proof}

\begin{proof}[Proof of \autoref{IntersectingSFDT}]
Let $(G, C_{0}, H, f)$ be a counterexample to \autoref{IntersectingSFDT} with $|V(G)|$ as small as possible. Note that $H$ has no strictly $f$-degenerate transversal, but $H - H_{X}$ has one, where $X$ is any nonempty subset of internal vertices in $G$. Then $G$ is connected and every internal vertex has degree at least four.

Suppose that $G$ contains an internal induced subgraph isomorphic to the configuration in \autoref{Kite} or \autoref{F35}. Let $B$ be the set of solid vertices in the configuration. Note the $G[B]$ is neither a complete graph nor a cycle, and has a vertex of degree at least three. This contradicts \autoref{CRITICAL}\ref{M2}. 

Suppose that $G$ contains an internal induced subgraph isomorphic to the configuration in \autoref{fig:subfig:RC1-a}. Let $B$ be the set of solid vertices in the configuration. Note the the sequence $a_{1}, a_{2}, \dots, a_{7}$ satisfies all the conditions of \autoref{NN}. It follows that $H$ has a strictly $f$-degenerate transversal, a contradiction. 

By \autoref{EXNN-RC-1}, the configurations in \autoref{RC-1} and \autoref{RC-2} cannot be an internal induced subgraph of $G$. Therefore, every internal vertex of $G$ has degree at least four, and none of the configurations in \autoref{Kite}, \autoref{F35}, \autoref{fig:subfig:RC1-a}, \autoref{RC-1} and \autoref{RC-2} can be an internal induced subgraph of $G$, but this contradicts \autoref{Intersecting}. This completes the proof of \autoref{IntersectingSFDT}. 
\end{proof}


\begin{thebibliography}{10}

\bibitem{Bernshteyn2021a}
A.~Bernshteyn and E.~Lee, Weak degeneracy of graphs, arXiv:2111.05908  (2021)
  \url{https://doi.org/10.48550/arXiv.2111.05908}.

\bibitem{MR3761240}
H.~Cai, J.~Wu and L.~Sun, Vertex arboricity of planar graphs without
  intersecting 5-cycles, J. Comb. Optim. 35~(2) (2018) 365--372.

\bibitem{MR3996735}
L.~Chen, R.~Liu, G.~Yu, R.~Zhao and X.~Zhou, {DP}-4-colorability of two classes
  of planar graphs, Discrete Math. 342~(11) (2019) 2984--2993.

\bibitem{MR3508765}
M.~Chen, L.~Huang and W.~Wang, List vertex-arboricity of toroidal graphs
  without 4-cycles adjacent to 3-cycles, Discrete Math. 339~(10) (2016)
  2526--2535.

\bibitem{MR2889524}
M.~Chen, A.~Raspaud and W.~Wang, Vertex-arboricity of planar graphs without
  intersecting triangles, European J. Combin. 33~(5) (2012) 905--923.

\bibitem{MR3634481}
I.~Choi, Toroidal graphs containing neither {$K^-_5$} nor 6-cycles are
  4-choosable, J. Graph Theory 85~(1) (2017) 172--186.

\bibitem{MR3233411}
I.~Choi and H.~Zhang, Vertex arboricity of toroidal graphs with a forbidden
  cycle, Discrete Math. 333 (2014) 101--105.

\bibitem{MR4135602}
X.~Cui, W.~Teng, X.~Liu and H.~Wang, A note of vertex arboricity of planar
  graphs without 4-cycles intersecting with 6-cycles, Theoret. Comput. Sci. 836
  (2020) 53--58.

\bibitem{MR4152773}
J.~Grytczuk and X.~Zhu, The {A}lon-{T}arsi number of a planar graph minus a
  matching, J. Combin. Theory Ser. B 145 (2020) 511--520.

\bibitem{MR3648207}
D.~Hu and J.-L. Wu, Planar graphs without intersecting 5-cycles are
  4-choosable, Discrete Math. 340~(8) (2017) 1788--1792.

\bibitem{MR4112063}
D.~Huang and Y.~Ling, A local condition on the vertex-arboricity of planar
  graphs, Adv. Math. (China) 49~(2) (2020) 146--158.

\bibitem{MR3979933}
D.~Huang, Y.~Ling and W.~Wang, Vertex arboricity of planar graphs without
  3-cycles adjacent to 6-cycles, Int. J. Math. Stat. 20~(3) (2019) 74--88.

\bibitem{MR4212281}
D.~Huang and J.~Qi, D{P}-coloring on planar graphs without given adjacent short
  cycles, Discrete Math. Algorithms Appl. 13~(2) (2021) 2150013.

\bibitem{MR2926103}
D.~Huang, W.~C. Shiu and W.~Wang, On the vertex-arboricity of planar graphs
  without 7-cycles, Discrete Math. 312~(15) (2012) 2304--2315.

\bibitem{MR3164029}
F.~Huang, X.~Wang and J.~Yuan, On the vertex-arboricity of {$K_5$}-minor-free
  graphs of diameter 2, Discrete Math. 322 (2014) 1--4.

\bibitem{MR3320048}
L.~Huang, M.~Chen and W.~Wang, Toroidal graphs without 3-cycles adjacent to
  5-cycles have list vertex-arboricity at most 2, Int. J. Math. Stat. 16~(1)
  (2015) 97--105.

\bibitem{MR3802151}
S.-J. Kim and K.~Ozeki, A sufficient condition for {DP}-4-colorability,
  Discrete Math. 341~(7) (2018) 1983--1986.

\bibitem{MR3969022}
S.-J. Kim and X.~Yu, Planar graphs without 4-cycles adjacent to triangles are
  {DP}-4-colorable, Graphs Combin. 35~(3) (2019) 707--718.

\bibitem{MR0360334}
H.~V. Kronk and J.~Mitchem, Critical point-arboritic graphs, J. London Math.
  Soc. (2) 9 (1974/75) 459--466.

\bibitem{Li2019}
R.~Li and T.~Wang, Variable degeneracy on toroidal graphs, arXiv:1907.07141,
  \url{https://doi.org/10.48550/arXiv.1907.07141}.

\bibitem{MR4294211}
R.~Li and T.~Wang, D{P}-4-coloring of planar graphs with some restrictions on
  cycles, Discrete Math. 344~(11) (2021) 112568.

\bibitem{MR4362322}
X.~Li, J.-B. Lv and M.~Zhang, D{P}-4-colorability of planar graphs without
  intersecting 5-cycles, Discrete Math. 345~(4) (2022) 112790.

\bibitem{Zhang2019}
X.~Li and M.~Zhang, Every planar graph without 5-cycles adjacent to 6-cycles is
  {DP}-4-colorable,   submitted for publication.

\bibitem{MR3881665}
R.~Liu and X.~Li, Every planar graph without 4-cycles adjacent to two triangles
  is {DP}-4-colorable, Discrete Math. 342~(3) (2019) 623--627.

\bibitem{MR4078909}
R.~Liu, X.~Li, K.~Nakprasit, P.~Sittitrai and G.~Yu, {DP}-4-colorability of
  planar graphs without adjacent cycles of given length, Discrete Appl. Math.
  277 (2020) 245--251.

\bibitem{MR3886261}
R.~Liu, S.~Loeb, Y.~Yin and G.~Yu, {DP}-3-coloring of some planar graphs,
  Discrete Math. 342~(1) (2019) 178--189.

\bibitem{MR4357325}
F.~Lu, Q.~Wang and T.~Wang, Cover and variable degeneracy, Discrete Math.
  345~(4) (2022) 112765.

\bibitem{MR4051856}
H.~Lu and X.~Zhu, The {A}lon-{T}arsi number of planar graphs without cycles of
  lengths 4 and {$l$}, Discrete Math. 343~(5) (2020) 111797.

\bibitem{MR4114324}
K.~M. Nakprasit and K.~Nakprasit, A generalization of some results on list
  coloring and {DP}-coloring, Graphs Combin. 36~(4) (2020) 1189--1201.

\bibitem{MR2408378}
A.~Raspaud and W.~Wang, On the vertex-arboricity of planar graphs, European J.
  Combin. 29~(4) (2008) 1064--1075.

\bibitem{MR2578908}
U.~Schauz, Flexible color lists in {A}lon and {T}arsi's theorem, and time
  scheduling with unreliable participants, Electron. J. Combin. 17~(1) (2010)
  R13.

\bibitem{MR4089638}
P.~Sittitrai and K.~Nakprasit, Every planar graph without pairwise adjacent 3-,
  4-, and 5-cycle is {DP}-4-colorable, Bull. Malays. Math. Sci. Soc. 43~(3)
  (2020) 2271--2285.

\bibitem{MR1871345}
R.~\v{S}krekovski, On the critical point-arboricity graphs, J. Graph Theory
  39~(1) (2002) 50--61.

\bibitem{MR3699856}
L.~Xue, List vertex arboricity of planar graphs with 5-cycles not adjacent to
  3-cycles and 4-cycles, Ars Combin. 133 (2017) 401--406.

\bibitem{MR3570576}
H.~Zhang, On list vertex 2-arboricity of toroidal graphs without cycles of
  specific length, Bull. Iranian Math. Soc. 42~(5) (2016) 1293--1303.

\bibitem{MR3906645}
X.~Zhu, The {A}lon-{T}arsi number of planar graphs, J. Combin. Theory Ser. B
  134 (2019) 354--358.

\end{thebibliography}
\end{document}